\documentclass[11pt]{amsart}
\usepackage{amssymb,mathrsfs}
\title{Geometry and arithmetic on the Siegel-Jacobi Space}

\begin{document}

\author{Jae-Hyun Yang}
%\date{August 16, 2013}
%{\textit{To the memory of Professor Kobayashi Shoshichi}}

\address{Department of Mathematics, Inha University,
Incheon 402-751, Korea}
%\newline
%Present address\,: Department of Mathematics, Inha University,
%Incheon 402-751, Korea}
\email{jhyang@inha.ac.kr }

%\begin{document}

\newtheorem{theorem}{Theorem}[section]
\newtheorem{lemma}{Lemma}[section]
\newtheorem{proposition}{Proposition}[section]
\newtheorem{remark}{Remark}[section]
\newtheorem{definition}{Definition}[section]

\renewcommand{\theequation}{\thesection.\arabic{equation}}
\renewcommand{\thetheorem}{\thesection.\arabic{theorem}}
\renewcommand{\thelemma}{\thesection.\arabic{lemma}}
\newcommand{\BR}{\mathbb R}
\newcommand{\BQ}{\mathbb Q}
\newcommand{\BT}{\mathbb T}
\newcommand{\BM}{\mathbb M}
\newcommand{\bn}{\bf n}
\def\charf {\mbox{{\text 1}\kern-.24em {\text l}}}
\newcommand{\BC}{\mathbb C}
\newcommand{\BZ}{\mathbb Z}

\thanks{
\\ \indent
The author was supported by Basic Science Program through the National Research Foundation of Korea
(NRF) funded by the Ministry of Education, Science
and Technology\,(47724-1)}

%\maketitle

\begin{abstract}
{The Siegel-Jacobi space is a non-symmetric homogeneous space which is very important geometrically and arithmetically.
In this paper, we discuss the theory of the geometry and the arithmetic of the Siegel-Jacobi space.}
\end{abstract}

\maketitle

%{\bf Keywords}: Siegel modular variety, fundamental domain, Siegel modular forms, toroidal compactification
%\tableofcontents
\newcommand\tr{\triangleright}
\newcommand\al{\alpha}
\newcommand\be{\beta}
\newcommand\g{\gamma}
\newcommand\gh{\Cal G^J}
\newcommand\G{\Gamma}
\newcommand\de{\delta}
\newcommand\e{\epsilon}
\newcommand\z{\zeta}
\newcommand\vth{\vartheta}
\newcommand\vp{\varphi}
\newcommand\om{\omega}
\newcommand\p{\pi}
\newcommand\la{\lambda}
\newcommand\lb{\lbrace}
\newcommand\lk{\lbrack}
\newcommand\rb{\rbrace}
\newcommand\rk{\rbrack}
\newcommand\s{\sigma}
\newcommand\w{\wedge}
\newcommand\fgj{{\frak g}^J}
\newcommand\lrt{\longrightarrow}
\newcommand\lmt{\longmapsto}
\newcommand\lmk{(\lambda,\mu,\kappa)}
\newcommand\Om{\Omega}
\newcommand\ka{\kappa}
\newcommand\ba{\backslash}
\newcommand\ph{\phi}
\newcommand\M{{\Cal M}}
\newcommand\bA{\bold A}
\newcommand\bH{\bold H}
\newcommand\D{\Delta}

\newcommand\Hom{\text{Hom}}
\newcommand\cP{\Cal P}

\newcommand\cH{\Cal H}

\newcommand\pa{\partial}

\newcommand\pis{\pi i \sigma}
\newcommand\sd{\,\,{\vartriangleright}\kern -1.0ex{<}\,}
\newcommand\wt{\widetilde}
\newcommand\fg{\frak g}
\newcommand\fk{\frak k}
\newcommand\fp{\frak p}
\newcommand\fs{\frak s}
\newcommand\fh{\frak h}
\newcommand\Cal{\mathcal}

\newcommand\fn{{\frak n}}
\newcommand\fa{{\frak a}}
\newcommand\fm{{\frak m}}
\newcommand\fq{{\frak q}}
\newcommand\CP{{\mathcal P}_n}
\newcommand\Hnm{{\mathbb H}_n \times {\mathbb C}^{(m,n)}}
\newcommand\BD{\mathbb D}
\newcommand\BH{\mathbb H}
\newcommand\CCF{{\mathcal F}_n}
\newcommand\CM{{\mathcal M}}
\newcommand\Gnm{\Gamma_{n,m}}
\newcommand\Cmn{{\mathbb C}^{(m,n)}}
\newcommand\Yd{{{\partial}\over {\partial Y}}}
\newcommand\Vd{{{\partial}\over {\partial V}}}

\newcommand\Ys{Y^{\ast}}
\newcommand\Vs{V^{\ast}}
\newcommand\LO{L_{\Omega}}
\newcommand\fac{{\frak a}_{\mathbb C}^{\ast}}

\begin{center}
\textit{To the memory of my teacher, Professor Shoshichi Kobayashi}
\end{center}
\vskip 2.1cm

\centerline{\large \bf Table of Contents}

\vskip 0.5cm $ \qquad\qquad\qquad\qquad\textsf{\large \ 1.
Introduction}$\vskip 0.021cm
%\par

%$\qquad\qquad\qquad\qquad \textsf{\large\ 2. The Jacobi Group $G^J$ }$
%\vskip 0.021cm
%\par

$\qquad\qquad\qquad\qquad \textsf{\large\ 2. Invariant Metrics and Laplacians on the Siegel-Jacobi Space }$
\vskip 0.021cm
%\par

$ \qquad\qquad\qquad\qquad  \textsf{\large\ 3. Invariant Differential Operators on the Siegel-Jacobi Space }$
\vskip 0.021cm
%\par

$ \qquad\qquad\qquad\qquad \textsf{\large\ 4. The Partial Cayley Transform}$
\vskip 0.021cm
%\par

$ \qquad\qquad\qquad\qquad  \textsf{\large\ 5. Invariant Metrics and Laplacians on the Siegel-Jacobi Disk}$
\vskip 0.021cm
%\par

$ \qquad\qquad\qquad\qquad \textsf{\large\ 6. A Fundamental Domain for the Siegel-Jacobi Space}$
\vskip 0.021cm
%\par

$ \qquad\qquad\qquad\qquad \textsf{\ 7. Jacobi Forms } $
\vskip 0.021cm
\par

$ \qquad\qquad\qquad\qquad \textsf{\large\ 8. Singular Jacobi Forms}$
\vskip 0.021cm
%\par

$ \qquad\qquad\qquad\qquad  \textsf{\large \ 9. The Siegel-Jacobi Operator}$
%\par

$ \qquad\qquad\qquad\qquad  \textsf{\large 10. Construction of Vector-Valued Modular Forms from Jacobi Forms }$
\vskip 0.021cm
%\par

$ \qquad\qquad\qquad\qquad  \textsf{\large 11. Maass-Jacobi Forms}$
\vskip 0.021cm

$ \qquad\qquad\qquad\qquad  \textsf{\large 12 The Schr{\"o}dinger-Weil Representation  }$
\vskip 0.021cm
%\par

$ \qquad\qquad\qquad\qquad  \textsf{\large 13. Final Remarks and Open Problems}$
\vskip 0.021cm
%\par

$ \qquad\qquad\qquad\qquad\textsf{\large Acknowledgements }$
%\par

$ \qquad\qquad\qquad\qquad\textsf{\large References }$
%\par

\newpage

\begin{section}{{\bf Introduction}}
\setcounter{equation}{0} For a given fixed positive integer $n$,
we let
$${\mathbb H}_n=\,\{\,\Omega\in \BC^{(n,n)}\,|\ \Om=\,^t\Om,\ \ \ \text{Im}\,\Om>0\,\}$$
be the Siegel upper half plane of degree $n$ and let
$$Sp(n,\BR)=\{ M\in \BR^{(2n,2n)}\ \vert \ ^t\!MJ_nM= J_n\ \}$$
be the symplectic group of degree $n$, where $F^{(k,l)}$ denotes
the set of all $k\times l$ matrices with entries in a commutative
ring $F$ for two positive integers $k$ and $l$, $^t\!M$ denotes
the transposed matrix of a matrix $M$ and
$$J_n=\begin{pmatrix} 0&I_n\\
                   -I_n&0\end{pmatrix}.$$
Then $Sp(n,\BR)$ acts on $\BH_n$ transitively by
\begin{equation}
M\cdot\Om=(A\Om+B)(C\Om+D)^{-1},
\end{equation} where $M=\begin{pmatrix} A&B\\
C&D\end{pmatrix}\in Sp(n,\BR)$ and $\Om\in \BH_n.$ Let
$$\G_n=Sp(n,\BZ)=\left\{ \begin{pmatrix} A&B\\
C&D\end{pmatrix}\in Sp(n,\BR) \,\big| \ A,B,C,D\
\textrm{integral}\ \right\}$$ be the Siegel modular group of
degree $n$. This group acts on $\BH_n$ properly discontinuously.
C. L. Siegel investigated the geometry of $\BH_n$ and automorphic
forms on $\BH_n$ systematically. Siegel\,\cite{Si1} found a
fundamental domain ${\mathcal F}_n$ for $\G_n\ba\BH_n$ and
described it explicitly. Moreover he calculated the volume of
$\CCF.$ We also refer to \cite{Ig},\,\cite{M2},\,\cite{Si2} for
some details on $\CCF.$

\vskip 0.1cm For two
positive integers $m$ and $n$, we consider the Heisenberg group
$$H_{\BR}^{(n,m)}=\big\{\,(\la,\mu;\ka)\,|\ \la,\mu\in \BR^{(m,n)},\ \kappa\in\BR^{(m,m)},\
\ka+\mu\,^t\la\ \text{symmetric}\ \big\}$$ endowed with the
following multiplication law
$$\big(\la,\mu;\ka\big)\circ \big(\la',\mu';\ka'\big)=\big(\la+\la',\mu+\mu';\ka+\ka'+\la\,^t\mu'-
\mu\,^t\la'\big)$$
with $\big(\la,\mu;\ka\big),\big(\la',\mu';\ka'\big)\in H_{\BR}^{(n,m)}.$
We define the {\it Jacobi group} $G^J$ of degree $n$ and index $m$ that is the semidirect product of
$Sp(n,\BR)$ and $H_{\BR}^{(n,m)}$
$$G^J=Sp(n,\BR)\ltimes H_{\BR}^{(n,m)}$$
endowed with the following multiplication law
$$
\big(M,(\lambda,\mu;\kappa)\big)\cdot\big(M',(\lambda',\mu';\kappa'\,)\big)
=\, \big(MM',(\tilde{\lambda}+\lambda',\tilde{\mu}+ \mu';
\kappa+\kappa'+\tilde{\lambda}\,^t\!\mu'
-\tilde{\mu}\,^t\!\lambda'\,)\big)$$ with $M,M'\in Sp(n,\BR),
(\lambda,\mu;\kappa),\,(\lambda',\mu';\kappa') \in
H_{\BR}^{(n,m)}$ and
$(\tilde{\lambda},\tilde{\mu})=(\lambda,\mu)M'$. Then $G^J$ acts
on $\BH_n\times \BC^{(m,n)}$ transitively by
\begin{equation}
\big(M,(\lambda,\mu;\kappa)\big)\cdot
(\Om,Z)=\Big(M\cdot\Om,(Z+\lambda \Om+\mu)
(C\Omega+D)^{-1}\Big), \end{equation} where $M=\begin{pmatrix} A&B\\
C&D\end{pmatrix} \in Sp(n,\BR),\ (\lambda,\mu; \kappa)\in
H_{\BR}^{(n,m)}$ and $(\Om,Z)\in \BH_n\times \BC^{(m,n)}.$ We note
that the Jacobi group $G^J$ is {\it not} a reductive Lie group and
the homogeneous space ${\mathbb H}_n\times \BC^{(m,n)}$ is not a
symmetric space. From now on, for brevity we write
$\BH_{n,m}=\BH_n\times \BC^{(m,n)}.$ The homogeneous space $\BH_{n,m}$ is called the
{\it Siegel-Jacobi space} of degree $n$ and index $m$.

\vskip 0.21cm
The aim of this paper is to discuss and survey the geometry and the arithmetic of the Siegel-Jacobi space $\BH_{n,m}$.
This article is organized as follows. In Section 2, we provide Riemannian metrics which are invariant under the action (1.2) of the Jacobi group and
their Laplacians. In Section 3, we discuss $G^J$-invariant differential operators on the Siegel-Jacobi space and give some related results.
In Section 4, we describe the partial Cayley transform of the Siegel-Jacobi disk onto the Siegel-Jacobi space which gives a partially bounded realization of the
Siegel-Jacobi space. We provide a compatibility result of a partial Cayley transform. In Section 5, we provide Riemannian metrics on the Siegel-Jacobi disk
which is invariant under the action (4.8) of the Jacobi group $G^J_*$ and their Laplacians using the partial Cayley transform. In Section 6, we find a fundamental domain for the
Siegel-Jacobi space with respect to the Siegel-Jacobi modular group. In Section 7, we give the canonical automorphic factor for the Jacobi group $G^J$ which is
obtained by a geometrical method and review the concept of Jacobi forms. In Section 8, we characterize singular Jacobi forms in terms of a certain differential
operator and their weights. In Section 9, we define the notion of the Siegel-Jacobi operator. We give the result about the compatibility with the Hecke-Jacobi operator.
In Section 10, we differentiate a given Jacobi form with respect to the toroidal variables by applying a homogeneous pluriharmonic
differential operator to a Jacobi form and then obtain a vector-valued modular form of a new weight. As an application, we provide an identity for an Eisenstein series.
In Section 11, we discuss the notion of Maass-Jacobi forms. In Section 12, we construct the Schr{\"o}dinger-Weil representation and give some results on
theta sums constructed from the Schr{\"o}dinger-Weil representation. In Section 13, we give some remarks and propose some open problems about the geometry and the arithmetic of
the Siegel-Jacobi space.

\vskip 0.51cm \noindent {\bf Notations:} \ \ We denote by
$\BQ,\,\BR$ and $\BC$ the field of rational numbers, the field of
real numbers and the field of complex numbers respectively. We
denote by $\BZ$ and $\BZ^+$ the ring of integers and the set of
all positive integers respectively. The symbol ``:='' means that
the expression on the right is the definition of that on the left.
For two positive integers $k$ and $l$, $F^{(k,l)}$ denotes the set
of all $k\times l$ matrices with entries in a commutative ring
$F$. For a square matrix $A\in F^{(k,k)}$ of degree $k$,
$\sigma(A)$ denotes the trace of $A$. For any $M\in F^{(k,l)},\
^t\!M$ denotes the transpose of a matrix $M$. $I_n$ denotes the
identity matrix of degree $n$. For $A\in F^{(k,l)}$ and $B\in
F^{(k,k)}$, we set $B[A]=\,^tABA.$ For a complex matrix $A$,
${\overline A}$ denotes the complex {\it conjugate} of $A$. For
$A\in \BC^{(k,l)}$ and $B\in \BC^{(k,k)}$, we use the abbreviation
$B\{ A\}=\,^t{\overline A}BA.$
For a number field $F$, we denote
by ${\mathbb A}_F$ the ring of adeles of $F$. If $F=\BQ$, the
subscript will be omitted.

\vskip 0.1cm

\end{section}

\begin{section}{{\bf Invariant Metrics and Laplacians on the Siegel-Jacobi Space}}
\setcounter{equation}{0}
\newcommand\POB{ {{\partial}\over {\partial{\overline \Omega}}} }
\newcommand\PZB{ {{\partial}\over {\partial{\overline Z}}} }
\newcommand\PX{ {{\partial}\over{\partial X}} }
\newcommand\PY{ {{\partial}\over {\partial Y}} }
\newcommand\PU{ {{\partial}\over{\partial U}} }
\newcommand\PV{ {{\partial}\over{\partial V}} }
\newcommand\PO{ {{\partial}\over{\partial \Omega}} }
\newcommand\PZ{ {{\partial}\over{\partial Z}} }

\vskip 0.21cm For $\Om=(\omega_{ij})\in\BH_n,$ we write $\Om=X+iY$
with $X=(x_{ij}),\ Y=(y_{ij})$ real. We put $d\Om=(d\om_{ij})$ and $d{\overline\Om}=(d{\overline\om}_{ij})$. We
also put
$$\PO=\,\left(\,
{ {1+\delta_{ij}}\over 2}\, { {\partial}\over {\partial \om_{ij} }
} \,\right) \qquad\text{and}\qquad \POB=\,\left(\, {
{1+\delta_{ij}}\over 2}\, { {\partial}\over {\partial {\overline
{\om}}_{ij} } } \,\right).$$ C. L. Siegel \cite{Si1} introduced
the symplectic metric $ds_{n;A}^2$ on $\BH_n$ invariant under the action
(1.1) of $Sp(n,\BR)$ that is given by
\begin{equation}
ds_{n;A}^2=A\,\s (Y^{-1}d\Om\, Y^{-1}d{\overline\Om}),\qquad A>0
\end{equation}
and H.
Maass \cite{M1} proved that its Laplacian is given by
\begin{equation}
\Delta_{n;A}=\,{4\over A}\,\s \left(\,Y\,
{}^{{}^{{}^{{}^\text{\scriptsize $t$}}}}\!\!\!
\left(Y\POB\right)\PO\right).\end{equation} And
\begin{equation}
dv_n(\Om)=(\det Y)^{-(n+1)}\prod_{1\leq i\leq j\leq n}dx_{ij}\,
\prod_{1\leq i\leq j\leq n}dy_{ij}\end{equation} is a
$Sp(n,\BR)$-invariant volume element on
$\BH_n$\,(cf.\,\cite{Si3},\,p.\,130).

\vskip 0.2cm
For a coordinate
$(\Om,Z)\in \BH_{n,m}$ with $\Om=(\omega_{\mu\nu})$ and
$Z=(z_{kl})$, we put $d\Om,\,d{\overline \Om},\,\PO,\,\POB$ as
before and set
\begin{eqnarray*}
Z\,&=&U\,+\,iV,\quad\ \ U\,=\,(u_{kl}),\quad\ \ V\,=\,(v_{kl})\ \
\text{real},\\
dZ\,&=&\,(dz_{kl}),\quad\ \ d{\overline Z}=(d{\overline z}_{kl}),
\end{eqnarray*}
%%%%%%%%%%%%%%%%%%%%%%%%%%%%%%%%%%%%%%%%%%%%%%%%%%%%%%%%%%%%%%%%%%%%%%%%%%%%%%%%%%%%%%%%%%%%%%%%%%%%%%%%%%%%%%%%%%%%%
%%%%%%%%%%%%%%%%%%%%%%%%%%%%%%%%%%%%%%%%%%%%%%%%%%%%%%%%%%%%%%%%%%%%%%%%%%%%%%%%%%%%%%%%%%%%%%%%%%%%%%%%%%%%%%%%%%%%%
$$\PZ=\begin{pmatrix} {\partial}\over{\partial z_{11}} & \hdots &
 {\partial}\over{\partial z_{m1}} \\
\vdots&\ddots&\vdots\\
 {\partial}\over{\partial z_{1n}} &\hdots & {\partial}\over
{\partial z_{mn}} \end{pmatrix},\quad \PZB=\begin{pmatrix}
{\partial}\over{\partial {\overline z}_{11} }   &
\hdots&{ {\partial}\over{\partial {\overline z}_{m1} }  }\\
\vdots&\ddots&\vdots\\
{ {\partial}\over{\partial{\overline z}_{1n} }  }&\hdots &
 {\partial}\over{\partial{\overline z}_{mn} }  \end{pmatrix}.$$
%%%%%%%%%%%%%%%%%%%%%%%%%%%%%%%%%%%%%%%%%%%%%%%%%%%%%%%%%%%%%%%%%%%%%%%%%%%%%%%%%%%%%%%%%%%%%%%%%%%%%%%%%%%%%%%%%%%%%
%%%%%%%%%%%%%%%%%%%%%%%%%%%%%%%%%%%%%%%%%%%%%%%%%%%%%%%%%%%%%%%%%%%%%%%%%%%%%%%%%%%%%%%%%%%%%%%%%%%%%%%%%%%%%%%%%%%%%

\newcommand\bw{d{\overline W}}
\newcommand\bz{d{\overline Z}}
\newcommand\bo{d{\overline \Omega}}

\vskip 0.3cm
 Yang proved the following theorems in \cite{YJH9}.

\begin{theorem} For any two positive real numbers
$A$ and $B$,
\begin{eqnarray}
ds_{n,m;A,B}^2&=&\,A\, \sigma\Big( Y^{-1}d\Om\,Y^{-1}d{\overline
\Om}\Big) \nonumber \\
&& \ \ + \,B\,\bigg\{ \sigma\Big(
Y^{-1}\,^tV\,V\,Y^{-1}d\Om\,Y^{-1} d{\overline \Om} \Big)
 +\,\sigma\Big( Y^{-1}\,^t(dZ)\,\bz\Big) \nonumber\\
&&\quad\quad -\sigma\Big( V\,Y^{-1}d\Om\,Y^{-1}\,^t(\bz)\Big)\,
-\,\sigma\Big( V\,Y^{-1}d{\overline \Om}\, Y^{-1}\,^t(dZ)\,\Big)
\bigg\} \nonumber
\end{eqnarray}
is a Riemannian metric on $\BH_{n,m}$ which is invariant under the action (1.2) of $G^J.$
In fact, $ds_{n,m;A,B}^2$ is a K{\"a}hler metric of $\BH_{n,m}.$
\end{theorem}
\vskip 1mm\noindent
{\it Proof.} See Theorem 1.1 in \cite{YJH9}. \hfill $\Box$

\vskip 3mm

\begin{theorem} The Laplacian $\Delta_{m,m;A,B}$ of the $G^J$-invariant metric $ds_{n,m;A,B}^2$ is given by
\begin{equation}
\Delta_{n,m;A,B}=\,{\frac 4A}\,{\mathbb M}_1 + {\frac 4B}
{\mathbb M}_2,
\end{equation}
where
\begin{eqnarray*}
{\mathbb M}_1\,&=&  \sigma\left(\,Y\,
{}^{{}^{{}^{{}^\text{\scriptsize $t$}}}}\!\!\!
\left(Y\POB\right)\PO\,\right)\, +\,\sigma\left(\,VY^{-1}\,^tV\,
{}^{{}^{{}^{{}^\text{\scriptsize $t$}}}}\!\!\!
\left(Y\PZB\right)\,\PZ\,\right)\\
& &\ \
+\,\sigma\left(V\,
{}^{{}^{{}^{{}^\text{\scriptsize $t$}}}}\!\!\!
\left(Y\POB\right)\PZ\,\right)
+\,\sigma\left(\,^tV\,
{}^{{}^{{}^{{}^\text{\scriptsize $t$}}}}\!\!\!
\left(Y\PZB\right)\PO\,\right)\nonumber
\end{eqnarray*}
and
\begin{equation*}
{\mathbb M}_2=\,\sigma\left(\, Y\,\PZ\,
{}^{{}^{{}^{{}^\text{\scriptsize $t$}}}}\!\!\!
\left(
\PZB\right)\,\right).
\end{equation*}
Furthermore ${\mathbb M}_1$ and ${\mathbb M}_2$ are differential operators on $\BH_{n,m}$ invariant under the action (1.2) of $G^J.$
\end{theorem}
\vskip 1mm\noindent
{\it Proof.} See Theorem 1.2 in \cite{YJH9}. \hfill $\Box$

\vskip 3mm
\begin{remark}
Erik Balslev \cite{B} developed the spectral theory of $\Delta_{1,1;1,1}$ on $\Bbb H_{1,1}$ for certain arithmetic subgroups of the Jacobi modular group to
prove that the set of all eigenvalues of $\Delta_{1,1;1,1}$ satisfies the Weyl law.
\end{remark}

\begin{remark}
The sectional curvature of $(\BH_{1,1}, ds^2_{1,1;A,B})$ is $-{3\over A}$ and hence is independent of the parameter $B$. We refer to \cite{YJH14}
for more detail.
\end{remark}

\begin{remark}
For an application of the invariant metric $ds^2_{n,m;A,B}$ we refer to \cite{YY}.
\end{remark}

\end{section}

\vskip 1cm

\begin{section}{{\bf Invariant Differential Operators on the Siegel-Jacobi Space }}
\setcounter{equation}{0}
\vskip 0.21cm
Before we discuss $G^J$-invariant differential operators on the Siegel-Jacobi space $\BH_{n,m}$, we review differential operators on the Siegel upper half plane $\BH_n$
invariant under the action (1.1).

\vskip 0.21cm
\newcommand\POB{ {{\partial\ }\over {\partial{\overline \Omega}}} }
\newcommand\PZB{ {{\partial\ }\over {\partial{\overline Z}}} }
\newcommand\PX{ {{\partial\ }\over{\partial X}} }
\newcommand\PY{ {{\partial\ }\over {\partial Y}} }
\newcommand\PU{ {{\partial\ }\over{\partial U}} }
\newcommand\PV{ {{\partial\ }\over{\partial V}} }
\newcommand\PO{ {{\partial\ }\over{\partial \Omega}} }
\newcommand\PZ{ {{\partial\ }\over{\partial Z}} }

\newcommand\POBS{ {{\partial\ \,}\over {\partial{\overline \Omega}_*} } }
\newcommand\PZBS{ {{\partial\ \,}\over {\partial{\overline Z_*}}} }
\newcommand\PXS{ {{\partial\ \,}\over{\partial X_*}} }
\newcommand\PYS{ {{\partial\ \,}\over {\partial Y_*}} }
\newcommand\PUS{ {{\partial\ \,}\over{\partial U_*}} }
\newcommand\PVS{ {{\partial\ \,}\over{\partial V_*}} }
\newcommand\POS{ {{\partial\ \,}\over{\partial \Omega_*}} }
\newcommand\PZS{ {{\partial\ \,}\over{\partial Z_*}} }

For brevity, we write $G=Sp(n,\BR).$ The isotropy
subgroup $K$ at $iI_n$ for the action (1.1) is a maximal compact
subgroup given by
\begin{equation*}
K=\left\{ \begin{pmatrix} A & -B \\ B & A \end{pmatrix} \Big| \
A\,^t\!A+ B\,^t\!B=I_n,\ A\,^t\!B=B\,^t\!A,\ A,B\in
\BR^{(n,n)}\,\right\}.
\end{equation*}

\noindent Let $\fk$ be the Lie algebra of $K$. Then the Lie
algebra $\fg$ of $G$ has a Cartan decomposition $\fg=\fk\oplus
\fp$, where
\begin{equation*}
\frak g=\left\{ \begin{pmatrix} X_1 & \ \ X_2 \\ X_3 & -\,{}^tX_1
\end{pmatrix}\,\Big|\ X_1,X_2,X_3\in\BR^{(n,n)},\ X_2=\,{}^tX_2,\ X_3=\,{}^tX_3\,\right\},
\end{equation*}

\begin{equation*}
\frak k=\left\{ \begin{pmatrix} X &  -Y \\ Y & \ X
\end{pmatrix}\in\BR^{(2n,2n)}\,\Big|
\ {}^tX+X=0,\ Y=\,{}^tY\,\right\},
\end{equation*}

\begin{equation*}
\fp=\left\{ \begin{pmatrix} X & \ Y \\ Y & -X \end{pmatrix} \Big| \
X=\,^tX,\ Y=\,^tY,\ X,Y\in \BR^{(n,n)}\,\right\}.
\end{equation*}

The subspace $\fp$ of $\fg$ may be regarded as the tangent space
of $\BH_n$ at $iI_n.$ The adjoint representation of $G$ on $\fg$
induces the action of $K$ on $\fp$ given by
\begin{equation}
k\cdot Z=\,kZ\,^tk,\quad k\in K,\ Z\in \fp.
\end{equation}

Let $T_n$ be the vector space of $n\times n$ symmetric complex
matrices. We let $\Psi: \fp\lrt T_n$ be the map defined by
\begin{equation}
\Psi\left( \begin{pmatrix} X & \ Y \\ Y & -X \end{pmatrix}
\right)=\,X\,+\,i\,Y, \quad \begin{pmatrix} X & \ Y \\ Y & -X
\end{pmatrix}\in \fp.
\end{equation}

\noindent We let $\delta:K\lrt U(n)$ be the isomorphism defined by
\begin{equation}
\delta\left( \begin{pmatrix} A & -B \\ B & A \end{pmatrix}
\right)=\,A\,+\,i\,B, \quad \begin{pmatrix} A & -B \\ B & A
\end{pmatrix}\in K,
\end{equation}

\noindent where $U(n)$ denotes the unitary group of degree $n$. We
identify $\fp$ (resp. $K$) with $T_n$ (resp. $U(n)$) through the
map $\Psi$ (resp. $\delta$). We consider the action of $U(n)$ on
$T_n$ defined by
\begin{equation}
h\cdot \omega=\,h\omega \,^th,\quad h\in U(n),\ \omega\in T_n.
\end{equation}

\noindent Then the adjoint action (3.1) of $K$ on $\fp$ is
compatible with the action (3.4) of $U(n)$ on $T_n$ through the
map $\Psi.$ Precisely for any $k\in K$ and $Z\in \fp$, we get
\begin{equation}
\Psi(k\,Z \,^tk)=\delta(k)\,\Psi(Z)\,^t\delta (k).
\end{equation}

\noindent The action (3.4) induces the action of $U(n)$ on the
polynomial algebra $ \textrm{Pol}(T_n)$ and the symmetric algebra
$S(T_n)$ respectively. We denote by $ \textrm{Pol}(T_n)^{U(n)}$
$\Big( \textrm{resp.}\ S(T_n)^{U(n)}\,\Big)$ the subalgebra of $
\textrm{Pol}(T_n)$ $\Big( \textrm{resp.}\ S(T_n)\,\Big)$
consisting of $U(n)$-invariants. The following inner product $(\
,\ )$ on $T_n$ defined by $$(Z,W)= \, \textrm{tr}
\big(Z\,{\overline W}\,\big),\quad Z,W\in T_n$$

\noindent gives an isomorphism as vector spaces
\begin{equation}
T_n\cong T_n^*,\quad Z\mapsto f_Z,\quad Z\in T_n,
\end{equation}

\noindent where $T_n^*$ denotes the dual space of $T_n$ and $f_Z$
is the linear functional on $T_n$ defined by
$$f_Z(W)=(W,Z),\quad W\in T_n.$$

\noindent It is known that there is a canonical linear bijection
of $S(T_n)^{U(n)}$ onto the algebra ${\mathbb D}(\BH_n)$ of
differential operators on $\BH_n$ invariant under the action (1.1)
of $G$. Identifying $T_n$ with $T_n^*$ by the above isomorphism
(3.6), we get a canonical linear bijection
\begin{equation}
\Theta_n:\textrm{Pol}(T_n)^{U(n)} \lrt {\mathbb D}(\BH_n)
\end{equation}

\noindent of $ \textrm{Pol}(T_n)^{U(n)}$ onto ${\mathbb
D}(\BH_n)$. The map $\Theta_n$ is described explicitly as follows.
Similarly the action (3.1) induces the action of $K$ on the
polynomial algebra $ \textrm{Pol}(\fp)$ and the symmetric algebra $S(\fp)$ respectively.
Through the map $\Psi$, the subalgebra $ \textrm{Pol}(\fp)^K$ of $
\textrm{Pol}(\fp)$ consisting of $K$-invariants is isomorphic to $
\textrm{Pol}(T_n)^{U(n)}$. We put $N=n(n+1)$. Let $\left\{
\xi_{\alpha}\,|\ 1\leq \alpha \leq N\, \right\}$ be a basis of a real vector space
$\fp$. If $P\in \textrm{Pol}(\fp)^K$, then
\begin{equation}
\Big(\Theta_n (P)f\Big)(gK)=\left[ P\left( {{\partial\ }\over {\partial
t_{\al}}}\right)f\left(g\,\text{exp}\, \left(\sum_{\al=1}^N
t_{\al}\xi_{\al}\right) K\right)\right]_{(t_{\al})=0},
\end{equation} where $f\in C^{\infty}({\mathbb H}_{n})$. We refer to \cite{He1,He2} for more detail. In
general, it is hard to express $\Phi(P)$ explicitly for a
polynomial $P\in \textrm{Pol}(\fp)^K$.

\vskip 0.3cm According to the work of Harish-Chandra \cite{HC1,HC2}, the
algebra ${\mathbb D}(\BH_n)$ is generated by $n$ algebraically
independent generators and is isomorphic to the commutative algebra
$\BC [x_1,\cdots,x_n]$ with $n$ indeterminates. We note that $n$
is the real rank of $G$. Let $\fg_{\BC}$ be the complexification
of $\fg$. It is known that $\BD(\BH_n)$ is isomorphic to the
center of the universal enveloping algebra of $\fg_{\BC}$.

\vskip 0.3cm Using a classical invariant theory (cf.\,\cite{Ho, W},
we can show that $\textrm{Pol}(T_n)^{U(n)}$ is generated by the
following algebraically independent polynomials
\begin{equation}
q_j (\omega)=\,\textrm{tr}\Big( \big(\omega {\overline
\omega}\big)^j\,\Big),\quad \omega\in T_n, \quad j=1,2,\cdots,n.
\end{equation}

For each $j$ with $1\leq j\leq n,$ the image $\Theta_n(q_j)$ of $q_j$
is an invariant differential operator on $\BH_n$ of degree $2j$.
The algebra ${\mathbb D}(\BH_n)$ is generated by $n$ algebraically
independent generators $\Theta_n(q_1),\Theta_n(q_2),\cdots,\Theta_n(q_n).$ In
particular,
\begin{equation}
\Theta_n(q_1)=\,c_1\, \textrm{tr}\! \left( Y\,
{}^{{}^{{}^{{}^\text{\scriptsize $t$}}}}\!\!\!
\left(Y\POB\right)\!\PO\right)\quad  \textrm{for\ some
constant}\ c_1.
\end{equation}

\noindent We observe that if we take $\omega=x+\,i\,y\in T_n$ with real $x,y$,
then $q_1(\omega)=q_1(x,y)=\,\textrm{tr}\big( x^2 +y^2\big)$ and
\begin{equation*}
q_2(\omega)=q_2(x,y)=\, \textrm{tr}\Big(
\big(x^2+y^2\big)^2+\,2\,x\big(xy-yx)y\,\Big).
\end{equation*}

\vskip 0.3cm It is a natural question to express the images
$\Theta_n(q_j)$ explicitly for $j=2,3,\cdots,n.$ We hope that the images $\Theta_n(q_j)$ for
$j=2,3,\cdots,n$ are expressed in the form of the $\textit{trace}$
as $\Phi(q_1)$.

\vskip 0.3cm H. Maass \cite{M2} found algebraically independent generators $H_1,H_2,\cdots,H_n$ of ${\mathbb D}(\BH_n)$.
We will describe $H_1,H_2,\cdots,H_n$ explicitly. For $M=\begin{pmatrix} A&B\\
C&D\end{pmatrix} \in Sp(n,\BR)$ and $\Omega=X+iY\in \BH_n$ with real $X,Y$, we set
\begin{equation*}
\Omega_*=\,M\!\cdot\!\Omega=\,X_*+\,iY_*\quad \textrm{with}\ X_*,Y_*\ \textrm{real}.
\end{equation*}
We set
\begin{eqnarray*}
K&=&\,\big( \Omega-{\overline\Omega}\,\big)\PO=\,2\,i\,Y \PO,\\
\Lambda&=&\,\big( \Omega-{\overline\Omega}\,\big)\POB=\,2\,i\,Y \POB,\\
K_*&=& \,\big( \Omega_*-{\overline\Omega}_*\,\big)\POS=\,2\,i\,Y_* \POS,\\
\Lambda_*&=&\,\big( \Omega_*-{\overline\Omega}_*\,\big)\POBS=\,2\,i\,Y_* \POBS.
\end{eqnarray*}
Then it is easily seen that
\begin{equation}
K_*=\,{}^t(C{\overline\Om}+D)^{-1}\,{}^t\!\left\{ (C\Omega+D)\,{}^t\!K \right\},
\end{equation}

\begin{equation}
\Lambda_*=\,{}^t(C{\Om}+D)^{-1}\,{}^t\!\left\{ (C{\overline\Omega}+D)\,{}^t\!\Lambda \right\}
\end{equation}
and
\begin{equation}
{}^t\!\left\{ (C{\overline\Omega}+D)\,{}^t\!\Lambda \right\}=\,\Lambda\,{}^t(C{\overline\Omega}+D)
-{{n+1}\over 2} \,\big( \Omega-{\overline\Omega}\,\big)\,{}^t\!C.
\end{equation}

Using Formulas (3.11),\,(3.12) and (3.13), we can show that

\begin{equation}
\Lambda_*K_* \,+\,{{n+1}\over 2}K_*=\,{}^t(C{\Om}+D)^{-1}\,{}^{{}^{{}^{{}^\text{\scriptsize $t$}}}}\!\!\!
\left\{ (C{\Omega}+D)\,{}^{{}^{{}^{{}^\text{\scriptsize $t$}}}}\!\!\!
\left( \Lambda K \,+\, {{n+1}\over 2}K \right)\right\}.
\end{equation}

\noindent Therefore we get
\begin{equation}
\textrm{tr}\!  \left( \Lambda_*K_* \,+\,{{n+1}\over 2}K_* \right) =\,
 \textrm{tr}\!   \left( \Lambda K \,+\, {{n+1}\over 2}K \right).
\end{equation}
We set
\begin{equation}
A^{(1)}=\,\Lambda K \,+\, {{n+1}\over 2}K .
\end{equation}

We define $A^{(j)}\,(j=2,3,\cdots,n)$ recursively by
\begin{eqnarray}
A^{(j)}&=&\, A^{(1)}A^{(j-1)}- {{n+1}\over 2}\,\Lambda\, A^{(j-1)}\,+\,{\frac 12}\,\Lambda\,\textrm{tr}\!\left(
A^{(j-1)} \right)\\
& & \ \ \,+\, {\frac 12}\,\big( \Omega-{\overline\Omega}\,\big) \,
{}^{{}^{{}^\text{\scriptsize $t$}}}\!\!\!\left\{
\big( \Omega-{\overline\Omega}\,\big)^{-1}
\,{}^t\!\left( \,{}^t\!\Lambda\,{}^t\!A^{(j-1)}\right)\right\}.\nonumber
\end{eqnarray}

\noindent We set
\begin{equation}
H_j=\,\textrm{tr}\!  \left( A^{(j)} \right),\quad j=1,2,\cdots,n.
\end{equation}
As mentioned before, Maass proved that $H_1,H_2,\cdots,H_n$ are algebraically independent generators
of ${\mathbb D}(\BH_n)$.

\vskip 0.3cm In fact, we see that
\begin{equation}
 -H_1 =\Delta_{n;1}=\,4\, \textrm{tr}\! \left( Y\,
{}^{{}^{{}^{{}^\text{\scriptsize $t$}}}}\!\!\!
\left(Y\POB\right)\!\PO\right).
\end{equation}
%\noindent
is the Laplacian for the invariant metric $ds^2_{n;1}$ on $\BH_n$.

%\vskip 0.3cm
%\noindent
%{\bf Conjecture.} For $j=2,3,\cdots,n,\ \Theta_n(q_j)=\,c_j\,H_j$ for a suitable constant $c_j.$

\vskip 0.53cm\noindent
$ \textbf{Example 3.1.}$ We consider the
case when $n=1.$ The algebra $ \textrm{Pol}(T_1)^{U(1)}$ is generated
by the polynomial
\begin{equation*}
q(\omega)=\omega\,{\overline \omega},\quad \omega=x+ \,iy\in \BC \ \textrm{with}\ x,y \ \textrm{real}.
\end{equation*}

Using Formula (3.8), we get

\begin{equation*}
\Theta_1 (q)=\,4\,y^2 \left( { {\partial^2}\over {\partial x^2} }+{
{\partial^2}\over {\partial y^2} }\,\right).
\end{equation*}

\noindent Therefore $\BD (\BH_1)=\BC\big[ \Theta_1(q)\big]=\,\BC [H_1].$

\vskip 0.3cm\noindent
$\textbf{Example 3.2.}$ We consider the
case when $n=2.$ The algebra $ \textrm{Pol}(T_2)^{U(2)}$ is generated
by the polynomial
\begin{equation*}
q_1(\omega)=\,\sigma \big(\omega\,{\overline \omega}\,\big),\quad q_2(\omega)=\,\sigma
\Big( \big(\omega\,{\overline \omega}\big)^2\Big), \quad \omega\in T_2.
\end{equation*}

Using Formula (3.8), we may express $\Theta_2(q_1)$ and $\Theta_2(q_2)$
explicitly. $\Theta_2 (q_1)$ is expressed by Formula (3.10). The
computation of $\Theta_2(q_2)$ might be quite tedious. We leave the
detail to the reader. In this case, $\Theta_2 (q_2)$ was essentially
computed in \cite{BC}, Proposition 6. Therefore
\begin{equation*}
\BD (\BH_2)=\BC\big[
\Theta_2(q_1), \Theta_2(q_2)\big]=\,\BC [H_1,H_2].
\end{equation*}
In fact, the center of the universal enveloping algebra
${\mathscr U}(\fg_{\BC})$ was computed in \cite{BC}.

\vskip 0.3cm G. Shimura \cite{Sh3} found canonically defined algebraically independent generators
of $\BD(\BH_n)$. We will describe his way of constructing those generators roughly. Let $K_\BC,\,
{\frak g}_\BC,\,{\mathfrak k}_\BC,{\mathfrak p}_\BC,\cdots$ denote the complexication of $K,\,{\mathfrak g},\,
{\mathfrak k},\,{\mathfrak p},\cdots$ respectively. Then we have the Cartan decomposition
\begin{equation*}
{\frak g}_\BC=\,{\mathfrak k}_\BC + {\mathfrak p}_\BC,\quad
{\mathfrak p}_\BC=\,{\mathfrak p}_\BC^+ + {\mathfrak p}_\BC^-
\end{equation*}
with the properties
\begin{equation*}
[{\mathfrak k}_\BC,{\mathfrak p}_\BC^{\pm}]\subset {\mathfrak p}_\BC^{\pm},\ \ \
[{\mathfrak p}_\BC^{+},{\mathfrak p}_\BC^+]=[{\mathfrak p}_\BC^-,{\mathfrak p}_\BC^-]=\{0\},
\ \ \ [{\mathfrak p}_\BC^+,{\mathfrak p}_\BC^-]=\,{\mathfrak k}_\BC,
\end{equation*}
where
\begin{equation*}
{\frak g}_\BC=\left\{ \begin{pmatrix} X_1 & \ \ X_2 \\ X_3 & -\,{}^tX_1
\end{pmatrix}\,\Big|\ X_1,X_2,X_3\in\BC^{(n,n)},\ X_2=\,{}^tX_2,\ X_3=\,{}^tX_3\,\right\},
\end{equation*}
\begin{equation*}
{\frak k}_\BC=\left\{ \begin{pmatrix} A &  -B \\ B & \ A
\end{pmatrix}\in\BC^{(2n,2n)}\,\Big|
\ {}^tA+A=0,\ B=\,{}^tB\,\right\},
\end{equation*}
\begin{eqnarray*}
{\mathfrak p}_\BC&=&\,\left\{ \begin{pmatrix} X & \ Y \\ Y & -X
\end{pmatrix}\in\BC^{(2n,2n)}\,\Big|
\ X=\,{}^tX,\ Y=\,{}^tY\,\right\},\\
{\mathfrak p}_\BC^+&=&\,\left\{ \begin{pmatrix} Z & iZ \\ iZ & -Z
\end{pmatrix}\in\BC^{(2n,2n)}\,\Big|
\ Z=\,{}^tZ\in \BC^{(n,n)}\,\right\},\\
{\mathfrak p}_\BC^-&=&\,\left\{ \begin{pmatrix} \ Z & -iZ \\ -iZ & \,-Z
\end{pmatrix}\in\BC^{(2n,2n)}\,\Big|
\ Z=\,{}^tZ\in \BC^{(n,n)}\,\right\}.
\end{eqnarray*}

For a complex vector space $W$ and a nonnegative integer $r$, we denote by $ \textrm{Pol}_r(W)$
the vector space of complex-valued homogeneous polynomial functions on $W$ of degree $r$. We put
$$\textrm{Pol}^r(W):=\sum_{s=0}^r \textrm{Pol}_s(W).$$
$ \textrm{Ml}_r(W)$ denotes the vector space of all $\BC$-multilinear maps of $W\times \cdots
\times W\,( r \ \textrm{copies})$ into $\BC$.
An element $Q$ of $ \textrm{Ml}_r(W)$ is called {\it symmetric} if
$$Q(x_1,\cdots,x_r)=\,Q(x_{\pi(1)},\cdots,x_{\pi(r)})$$ for each permutation $\pi$ of
$\{ 1,2,\cdots,r\}.$ Given $P\in \textrm{Pol}_r(W)$, there is a unique element symmetric element
$P_*$ of $ \textrm{Ml}_r(W)$ such that
\begin{equation}
P(x)=\,P_*(x,\cdots,x)\qquad \textrm{for all}\ x\in W.
\end{equation}
Moreover the map $P\mapsto P_*$ is a $\BC$-linear bijection of $\textrm{Pol}_r(W)$ onto the
set of all symmetric elements of $ \textrm{Ml}_r(W)$. We let $S_r(W)$ denote the subspace
consisting of all homogeneous elements of degree $r$ in the symmetric algebra $S(W)$. We note that
$ \textrm{Pol}_r(W)$ and $S_r(W)$ are dual to each other with respect to the pairing
\begin{equation}
\langle \alpha,x_1\cdots x_r \rangle=\,\alpha_*(x_1,\cdots,x_r)\qquad (x_i\in W,\ \alpha
\in \textrm{Pol}_r(W)).
\end{equation}

\vskip 0.2cm
Let ${\mathfrak p}_\BC^*$ be the dual space of ${\mathfrak p}_\BC$, that is, ${\mathfrak p}_\BC^*=
\textrm{Pol}_1({\mathfrak p}_\BC).$ Let $\{ X_1,\cdots,X_N\}$ be a basis of ${\mathfrak p}_\BC$ and
$\{ Y_1,\cdots,Y_N\}$ be the basis of ${\mathfrak p}_\BC^*$ dual to $\{ X_\nu\},$ where
$N=n(n+1)$. We note that $\textrm{Pol}_r({\mathfrak p}_\BC)$ and
$\textrm{Pol}_r({\mathfrak p}_\BC^*)$ are dual to each other with respect to the pairing
\begin{equation}
\langle\alpha,\beta\rangle=\sum \alpha_*(X_{i_1},\cdots,X_{i_r})\,\beta_* (Y_{i_1},\cdots,Y_{i_r}),\
\end{equation}
where $\alpha\in \textrm{Pol}_r({\mathfrak p}_\BC),\ \beta\in \textrm{Pol}_r({\mathfrak p}_\BC^*)$ and
$(i_1,\cdots,i_r)$ runs over $\{1,\cdots,N\}^r.$ Let ${\mathscr U}(\fg_\BC)$ be the universal enveloping
algebra of $\fg_\BC$ and ${\mathscr U}^p(\fg_\BC)$ its subspace spanned by the elements of the form
$V_1\cdots V_s$ with $V_i\in\fg_\BC$ and $s\leq p.$  We recall that there is a $\BC$-linear bijection
$\psi$ of the symmetric algebra $S(\fg_\BC)$ of $\fg_\BC$ onto ${\mathscr U}(\fg_\BC)$ which is
characterized by the property that $\psi(X^r)=X^r$ for all $X\in\fg_\BC.$ For each $\alpha\in
\textrm{Pol}_r({\mathfrak p}_\BC^*)$ we define an element $\omega(\alpha)$ of
${\mathscr U}(\fg_\BC)$ by
\begin{equation}
\omega(\alpha):=\sum \alpha_*(Y_{i_1},\cdots,Y_{i_r})\,X_{i_1}\cdots X_{i_r},
\end{equation}
where $(i_1,\cdots,i_r)$ runs over $\{1,\cdots,N\}^r.$ If $Y\in {\mathfrak p}_\BC$, then $Y^r$ as
an element of $\textrm{Pol}_r({\mathfrak p}_\BC^*)$ is defined by
$$ Y^r(u)=Y(u)^r \qquad   \textrm{for all} \ u\in {\mathfrak p}_\BC^*.$$
Hence $(Y^r)_*(u_1,\cdots,u_r)=\,Y(u_1)\cdots Y(u_r).$ According to (2.25), we see that if
$\alpha (\sum t_iY_i)=\,P(t_1,\cdots,t_N)$ for $t_i\in\BC$ with a polynomial $P$, then
\begin{equation}
\omega(\alpha)=\,\psi(P(X_1,\cdots,X_N)).
\end{equation}
Thus $\omega$ is a $\BC$-linear injection of $\textrm{Pol}({\mathfrak p}_\BC^*)$ into
${\mathscr U}(\fg_\BC)$ independent of the choice of a basis. We observe that
$\omega\big( \textrm{Pol}_r({\mathfrak p}_\BC^*)\big)=\,\psi(S_r({\mathfrak p}_\BC)).$ It is a well-known
fact that if $\alpha_1,\cdots,\alpha_m\in \textrm{Pol}_r({\mathfrak p}_\BC^*)$, then
\begin{equation}
\omega (\alpha_1\cdots\alpha_m)-\omega(\alpha_m)\cdots\omega(\alpha_1)\in {\mathscr U}^{r-1}(\fg_\BC).
\end{equation}

\vskip 0.2cm
We have a canonical pairing
\begin{equation}
\langle\, \ ,\ \,\rangle: \textrm{Pol}_r({\mathfrak p}_\BC^+) \times \textrm{Pol}_r({\mathfrak p}_\BC^-)
\lrt \BC
\end{equation}
defined by
\begin{equation}
\langle f ,g\rangle=\sum f_*({\widetilde X}_{i_1},\cdots,{\widetilde X}_{i_r}) g_*({\widetilde Y}_{i_1},\cdots,{\widetilde Y}_{i_r}),
\end{equation}
where $f_*$ (resp. $g_*$) are the unique symmetric elements of $ \textrm{Ml}_r({\mathfrak p}_\BC^+)$ (resp.\
$ \textrm{Ml}_r({\mathfrak p}_\BC^-))$, and
$\{ {\widetilde X}_1,\cdots,{\widetilde X}_{\widetilde N}\}$ and $\{ {\widetilde Y}_1,\cdots, {\widetilde Y}_{\widetilde N}\}$
are dual bases of ${\mathfrak p}_\BC^+$
and ${\mathfrak p}_\BC^-$ with respect to the Killing form $B(X,Y)=\,2(n+1)\,\textrm{tr}(XY)$, ${\widetilde N}= {{n(n+1)}\over 2},$
and $(i_1,\cdots,i_r)$ runs over $\big\{1,\cdots,{\widetilde N}\big\}^r.$

\vskip 0.3cm The adjoint representation of $K_\BC$ on ${\mathfrak p}_\BC^{\pm}$ induces the representation
of $K_\BC$ on $ \textrm{Pol}_r({\mathfrak p}_\BC^{\pm})$. Given a $K_\BC$-irreducible subspace
$Z$ of $\textrm{Pol}_r({\mathfrak p}_\BC^+),$ we can find a unique
$K_\BC$-irreducible subspace
$W$ of $\textrm{Pol}_r({\mathfrak p}_\BC^-)$ such that $\textrm{Pol}_r({\mathfrak p}_\BC^-)$ is the
direct sum of $W$ and the annihilator of $Z$. Then $Z$ and $W$ are dual with respect to the pairing (3.26).
Take bases $\{\zeta_1,\cdots,\zeta_\kappa\}$ of $Z$ and $\{ \xi_1,\cdots,\xi_\kappa\}$ of $W$ that
are dual to each other. We set
\begin{equation}
f_Z(x,y) =\,\sum_{\nu=1}^\kappa \zeta_\nu(x)\,\xi_\nu(y)\qquad (x\in {\mathfrak p}_\BC^+,
\ y\in {\mathfrak p}_\BC^-).
\end{equation}
It is easily seen that $f_Z$ belongs to $\textrm{Pol}_{2r}({\mathfrak p}_\BC)^K$ and is independent of the choice of
dual bases $\{ \zeta_\nu\}$ and $\{ \xi_\nu\}.$ Shimura \cite{Sh3} proved that there exists a
canonically defined set
$\{ Z_1,\cdots,Z_n\}$ with a $K_\BC$-irreducible subspace $Z_r$ of $\textrm{Pol}_r({\mathfrak p}_\BC^+)\
(1\leq r \leq n)$ such that $f_{Z_1},\cdots,f_{Z_n}$ are algebraically independent generators of
$\textrm{Pol}({\mathfrak p}_\BC)^K$. We can identify ${\mathfrak p}_\BC^+$ with $T_n$.
We recall that $T_n$ denotes the vector space of $n\times n$ symmetric complex matrices.
We can take $Z_r$ as the subspace
of $\textrm{Pol}_r(T_n)$ spanned by the functions $f_{a;r}(Z)=\det_r(\,{}^taZa)$ for all $a\in GL(n,\BC),$
where $\det_r(x)$ denotes the determinant of the upper left $r\times r$ submatrix of $x$. For every $f\in
\textrm{Pol}({\mathfrak p}_\BC)^K$, we let $\Omega(f)$ denote the element of $\BD (\BH_n)$ represented by
$\omega (f)$. Then $\BD (\BH_n)$ is the polynomial ring $\BC[\omega (f_{Z_1}),\cdots, \omega (f_{Z_n})]$
generated by $n$ algebraically independent elements $\omega (f_{Z_1}),\cdots, \omega (f_{Z_n}).$

\vskip 0.21cm
\newcommand\PE{ {{\partial}\over {\partial \eta}} }
\newcommand\PEB{ {{\partial}\over {\partial{\overline \eta}}} }

\newcommand\pdx{ {{\partial}\over{\partial x}} }
\newcommand\pdy{ {{\partial}\over{\partial y}} }
\newcommand\pdu{ {{\partial}\over{\partial u}} }
\newcommand\pdv{ {{\partial}\over{\partial v}} }
\newcommand\PW{ {{\partial}\over {\partial W}} }
\newcommand\PWB{ {{\partial}\over {\partial{\overline W}}} }
\renewcommand\th{\theta}
\renewcommand\l{\lambda}
\renewcommand\k{\kappa}

\vskip 0.3cm
Now we investigate differential operators on the Siegel-Jacobi space $\BH_{n,m}$ invariant under the action (1.2) of $G^J$.
The stabilizer $K^J$ of $G^J$ at $(iI_n,0)$ is given by
\begin{equation*}
K^J=\Big\{ \big(k,(0,0;\ka)\big)\,\big|\ k\in K,\ \ka=\,^t\ka\in
\BR^{(m,m)}\,\Big\}.
\end{equation*}

\noindent Therefore $\Hnm\cong G^J/K^J$ is a homogeneous space which is not symmetric. The Lie algebra $\fg^J$ of $G^J$ has a decomposition
\begin{equation*}
\fg^J=\fk^J+\fp^J,
\end{equation*}

\noindent where
\begin{equation*}
\fg^J=\Big\{ \big(Z,(P,Q,R)\big)\,\big|\ Z\in \fg,\ P,Q\in\BR^{(m,n)},\
R=\,^t\!R\in \BR^{(m,m)}\,\Big\},
\end{equation*}

\begin{equation*}
\fk^J=\Big\{ \big(X,(0,0,R)\big)\,\big|\ X\in \fk,\
R=\,^t\!R\in \BR^{(m,m)}\,\Big\},
\end{equation*}

\noindent
\begin{equation*}
\fp^J=\Big\{ \big(Y,(P,Q,0)\big)\,\big|\ Y\in \fp,\ P,Q\in
\BR^{(m,n)}\,\Big\}.
\end{equation*}

\noindent Thus the tangent space of the homogeneous space $\Hnm$
at $(iI_n,0)$ is identified with $\fp^J$.

\vskip 0.2cm
If $\alpha=\left( \begin{pmatrix} X_1 & \ Y_1 \\ Z_1 & -{}^t\!X_1 \end{pmatrix},
(P_1,Q_1,R_1)\right)$ and
$\beta=\left( \begin{pmatrix} X_2 & \ Y_2 \\ Z_2 & -{}^t\!X_2 \end{pmatrix},
(P_2,Q_2,R_2)\right)$ are elements of $\fg^J$, then the Lie bracket
$[\alpha,\beta]$ of $\alpha$ and $\beta$ is given by
\begin{equation}
[\alpha,\beta]=\left( \begin{pmatrix} X^* & \ Y^* \\ Z^* & -{}^t\!X^* \end{pmatrix},
(P^*,Q^*,R^*)\right),
\end{equation}
where
\begin{eqnarray*}
X^*&=& X_1X_2 - X_2X_1 + Y_1Z_2 - Y_2Z_1,\\
Y^*&=& X_1Y_2 - X_2Y_1 + Y_2\,^t\!X_1 - Y_1\,^t\!X_2,\\
Z^*&=& Z_1X_2 - Z_2X_1 + \,^t\!X_2 Z_1 - \,^t\!X_1 Z_2,\\
P^*&=& P_1X_2 - P_2X_1 + Q_1Z_2 - Q_2Z_1,\\
Q^*&=& P_1Y_2 - P_2Y_1 + Q_2\,^t\!X_1 - Q_1\,^t\!X_2,\\
R^*&=& P_1\,^t\!Q_2 - P_2\,^t\!Q_1 + Q_2\,^t\!P_1 - Q_1\,^t\!P_2
\end{eqnarray*}

\begin{lemma} \begin{equation*}
[\fk^J,\fk^J]\subset \fk^J,\quad [\fk^J,\fp^J]\subset \fp^J.
\end{equation*}
\end{lemma}

\begin{proof}
The proof follows immediately from Formula (3.29).
\end{proof}

\vskip 0.3cm

\begin{lemma}
Let
$$k^J = \left( \begin{pmatrix} A & -B \\ B & \ A \end{pmatrix},
(0,0,\kappa)\right)\in K^J $$
with $ \begin{pmatrix} A & -B \\ B & \ A \end{pmatrix}\in K,
\ \ka=\,^t\ka \in\BR^{(m,m)}$
and
$$\alpha=\left( \begin{pmatrix} X & \ Y\\ Y & -X \end{pmatrix},(P,Q,0)\right)
\in \fp^J$$
with $ X=\,^tX,\ Y=\,^tY\in \BR^{(n,n)},\ P,Q\in\BR^{(m,n)}.$
Then the adjoint action of $K^J$ on $\fp^J$ is given by
\begin{equation}
\textrm{Ad}(k^J)\alpha = \left( \begin{pmatrix} X_* & \ Y_*\\ Y_* & -X_* \end{pmatrix},(P_*,Q_*,0)\right),
\end{equation}
where
\begin{eqnarray}
X_* &=& AX\,^t\!A - \big( BX\,^t\!B + BY\,^t\!A + AY\,^t\!B\big),\\
Y_* &=& \big( AX\,^t\!B + AY\,^t\!A + BX\,^t\!A\big) -BY\,^t\!B,\\
P_* &=& P\,\,^t\!A - Q\,\,^t\!B,\\
Q_* &=& P\,\,^t\!B + Q\,\,^t\!A.
\end{eqnarray}
\end{lemma}

\begin{proof}
We leave the proof to the reader.
\end{proof}

\vskip 0.3cm
We recall that $T_n$ denotes the vector space of all $n\times n$ symmetric complex matrices.
For brevity, we put $T_{n,m}:=T_n\times \BC^{(m,n)}.$
We define the real linear isomorphism $\Phi:\fp^J\lrt T_{n,m}$ by
\begin{equation}
\Phi \left( \begin{pmatrix} X & \ Y\\ Y & -X \end{pmatrix},(P,Q,0)\right)=
\big( X\,+\,i\,Y,\,P\,+\,i\,Q \big),
\end{equation}
where $\begin{pmatrix} X & \ Y\\ Y & -X\end{pmatrix}\in \fp$ and $P,Q\in \BR^{(m,n)}.$

\vskip 0.3cm Let $S(m,\BR)$ denote the additive group consisting of all $m\times m$
real symmetric matrices. Now we define the isomorphism $\theta : K^J \lrt U(n)\times S(m,\BR)$
by
\begin{equation}
\theta (h,(0,0,\kappa))= (\delta(h), \kappa),\quad h\in K, \ \kappa\in S(m,\BR),
\end{equation}
where $\delta :K\lrt U(n)$ is the map defined by (3.3).
Identifying $\BR^{(m,n)}\times \BR^{(m,n)}$ with
$\BC^{(m,n)}$, we can identify $\fp^J$ with $T_n\times
\BC^{(m,n)}$.

\vskip 0.3cm
\begin{theorem}
The adjoint representation of $K^J$ on $\fp^J$ is compatible with the {\it natural\ action} of
$U(n)\times S(m,\BR)$ on $T_{n,m}$ defined by
\begin{equation}
(h,\kappa)\cdot (\omega,z):= (h\,\omega \,^th,\,z\,^th),\qquad
h\in U(n),\ \kappa\in S(m,\BR), \ (\omega,z)\in T_{n,m}
\end{equation}
through the maps $
\Phi$ and $\theta$. Precisely, if $k^J\in K^J$ and $\alpha\in \fp^J$, then we have the
following equality
\begin{equation}
\Phi \big( Ad\big( k^J\,\big)\alpha\Big) = \theta \big( k^J\, \big)\cdot \Phi (\alpha).
\end{equation}
Here we
regard the complex vector space $T_{n,m}$ as a real vector space.
\end{theorem}

\begin{proof}
Let
$$k^J = \left( \begin{pmatrix} A & -B \\ B & \ A \end{pmatrix},
(0,0,\kappa)\right)\in K^J $$
with $ \begin{pmatrix} A & -B \\ B & \ A \end{pmatrix}\in K,
\ \ka=\,^t\ka \in\BR^{(m,m)}$
and
$$\alpha=\left( \begin{pmatrix} X & \ Y\\ Y & -X \end{pmatrix},(P,Q,0)\right)
\in \fp^J$$
with $ X=\,^tX,\ Y=\,^tY\in \BR^{(n,n)},\ P,Q\in\BR^{(m,n)}.$
Then we have

\begin{eqnarray*}
\theta \big( k^J\, \big)\cdot \Phi (\alpha)
&=& \big( A\,+\,i\,B,\,\kappa\big)\cdot \big(  X\,+\,i\,Y,\,P\,+\,i\,Q\big)\\
&=& \big( (A+iB)(X+iY)\,^t(A+iB),\,(P+iQ)\,^t\!(A+iB)\big)\\
&=& \big( X_*\,+\,i\,Y_*,\,P_*\,+\,i\,Q_*\big)\\
&=& \Phi  \left( \begin{pmatrix} X_* & \ Y_*\\ Y_* & -X_* \end{pmatrix},(P_*,Q_*,0)\right)\\
&=& \Phi \big( Ad\big( k^J\,\big)\alpha\Big)\qquad (by \ Lemma\ 3.2),
\end{eqnarray*}
where $X_*,Y_*,Z_*$ and $Q_*$ are given by the formulas (3.31),\,(3.32),\,(3.33) and (3.34)
respectively.
\end{proof}

\newcommand\bw{d{\overline W}}
\newcommand\bz{d{\overline Z}}
\newcommand\bo{d{\overline \Omega}}

\vskip 0.3cm We now study the algebra $\BD(\BH_{n,m})$ of all
differential operators on $\BH_{n,m}$ invariant under the {\it natural action}
(1.2) of $G^J$. The action (3.37) induces the action of $U(n)$ on
the polynomial algebra $\text{Pol}_{n,m}:=\,\text{Pol}\,(T_{n,m}).$
We denote by $\text{Pol}_{n,m}^{U(n)}$ the subalgebra of
$\text{Pol}_{n,m}$ consisting of all $U(n)$-invariants. Similarly
the action (3.30) of $K$ induces the action of $K$ on the
polynomial algebra $ \textrm{Pol}\big(\fp^J\big)$. We see that
through the identification of $\fp^J$ with $T_{n,m}$, the algebra
$ \textrm{Pol}\big(\fp^J\big)$ is isomorphic to
$\text{Pol}_{n,m}.$ The following $U(n)$-invariant inner product
$(\,\,,\,)_*$ of the complex vector space $T_{n,m}$ defined by
\begin{equation*}
\big((\om,z),(\om',z')\big)_*= \textrm{tr}\big(\om{\overline
{\om'}}\,\big)+ \textrm{tr}\big(z\,^t{\overline {z'}}\,\big),\quad
(\om,z),\,(\om',z')\in T_{n,m}
\end{equation*}

\noindent gives a canonical isomorphism
\begin{equation*}
T_{n,m}\cong\,T_{n,m}^*,\quad (\om,z)\mapsto f_{\om,z},\quad
(\om,z)\in T_{n,m},
\end{equation*}

\noindent where $f_{\om,z}$ is the linear functional on $T_{n,m}$
defined by
\begin{equation*}
f_{\om,z}\big((\om',z'\,)\big)=\big((\om',z'),(\om,z)\big)_*,\quad
(\om',z'\,)\in T_{n,m}.
\end{equation*}

\noindent According to Helgason (\cite{He2}, p.\,287), one gets a canonical linear bijection
of $S(T_{n,m})^{U(n)}$ onto $\BD(\BH_{n,m})$. Identifying
$T_{n,m}$ with $T_{n,m}^*$ by the above isomorphism, one gets a
natural linear bijection
$$\Theta_{n,m}:\,\text{Pol}^{U(n)}_{n,m}\lrt \BD(\BH_{n,m})$$
of $\text{Pol}^{U(n)}_{n,m}$ onto $\BD(\BH_{n,m}).$ The map
$\Theta_{n,m}$ is described explicitly as follows. We put
$N_{\star}=n(n+1)+2mn$. Let $\big\{ \eta_{\alpha}\,|\ 1\leq \alpha
\leq N_{\star}\, \big\}$ be a basis of $\fp^J$. If $P\in
\textrm{Pol}\big(\fp^J\big)^K=\mathrm{Pol}_{n,m}^{U(n)}$, then
\begin{equation}
\Big(\Theta_{n,m} (P)f\Big)(gK^J)=\left[ P\left( {{\partial}\over
{\partial t_{\al}}}\right)f\left(g\,\text{exp}\,
\left(\sum_{\al=1}^{N_{\star}} t_{\al}\eta_{\al}\right)
K^J\right)\right]_{(t_{\al})=0},
\end{equation}

\noindent where $g\in G^J$ and $f\in C^{\infty}({\mathbb H}_{n,m})$. In general,
it is hard to express $\Theta_{n,m}(P)$ explicitly for a polynomial
$P\in \textrm{Pol}\big(\fp^J\big)^K$.

 \vskip 0.355cm We propose the following natural problems.

\vskip 0.2cm \noindent $ \textbf{Problem 1.}$ Find a complete list of explicit generators
of $\text{Pol}_{n,m}^{U(n)}$.

\vskip 0.32cm \noindent $ \textbf{Problem 2.}$ Find all the relations among a set of generators of $\text{Pol}_{n,m}^{U(n)}$.

\vskip 0.32cm \noindent $ \textbf{Problem 3.}$ Find an easy or effective way to
express the images of the above invariant polynomials or generators of $\text{Pol}_{n,m}^{U(n)}$ under the
Helgason map $\Theta_{n,m}$ explicitly.

%\vskip 0.32cm \noindent $ \textbf{Problem 4.}$ Decompose $\text{Pol}_{n,m}$ into
%$U(n)$-irreducibles.

\vskip 0.32cm \noindent $ \textbf{Problem 4.}$ Find a complete list of
 explicit generators of the algebra $\BD(\BH_{n,m})$. Or
construct explicit $G^J$-invariant differential operators on $\BH_{n,m}.$

\vskip 0.32cm \noindent $ \textbf{Problem 5.}$ Find all the relations among a set of generators of $\BD(\BH_{n,m})$.

\vskip 0.32cm \noindent $ \textbf{Problem 6.}$ Is $\text{Pol}_{n,m}^{U(n)}$ finitely generated ?

\vskip 0.32cm \noindent $ \textbf{Problem 7.}$ Is $\BD(\BH_{n,m})$ finitely generated ?

%\vskip 0.32cm \noindent $ \textbf{Problem 8.}$ Are there canonical ways to find
%generators of $\text{Pol}_{n,m}^{U(n)}$ like Shimura's way in the hermitian symmetric space case \cite{Sh} ?

\vskip 0.5cm
We will give answers to Problems 1, 2 and 6.

\vskip 0.25cm
We put $\varphi^{(2k)} = \operatorname{tr}((w\bar{w})^k)$.
Moreover, for $1\leq a,b\leq m$ and $k \geq 0$, we put
\begin{alignat*}{2}
   \psi^{(0,2k,0)}_{ba} &= (\bar{z}(w\bar{w})^k \, {}^t\!z)_{ba}, & \qquad
   \psi^{(1,2k,0)}_{ba} &= (z\bar{w}(w\bar{w})^k \, {}^t\!z)_{ba}, \\
   \psi^{(0,2k,1)}_{ba} &= (\bar{z}(w\bar{w})^k w\, {}^t\!\bar{z})_{ba}, & \qquad
   \psi^{(1,2k,1)}_{ba} &= (z\bar{w}(w\bar{w})^k w\, {}^t\!\bar{z})_{ba}.
\end{alignat*}

\noindent Then we have the following relations:
\begin{equation}\label{eq:relations_of_generators}
   \varphi^{(2k)} = \bar{\varphi}^{(2k)}, \
   \psi^{(1,2k,1)}_{ab} = \psi^{(0,2k+2,0)}_{ba}, \
   \psi^{(1,2k,0)}_{ab} = \psi^{(1,2k,0)}_{ba} =
   \bar{\psi}^{(0,2k,1)}_{ab} = \bar{\psi}^{(0,2k,1)}_{ba}.
\end{equation}

\noindent Then we have the following theorem:

\begin{theorem}\label{thm:FFT1}
   The algebra $\operatorname{Pol}^{U(n)}_{n,m}$ is generated by the following polynomials\/{\rm :}
   $$
      \varphi^{(2k+2)}, \qquad
      \operatorname{Re} \psi^{(0,2k,0)}_{ab}, \qquad
      \operatorname{Im} \psi^{(0,2k,0)}_{cd}, \qquad
      \operatorname{Re} \psi^{(1,2k,0)}_{ab}, \qquad
      \operatorname{Im} \psi^{(1,2k,0)}_{ab}.
   $$
   Here the indices run as follows:
   $$
      0 \leq k \leq n-1, \qquad
      1 \leq a \leq b \leq m, \qquad
      1 \leq c < d \leq m.
   $$
\end{theorem}

This is seen from the following theorem
by using (3.40):

\begin{theorem}\label{thm:FFT2}
   The algebra $\operatorname{Pol}_{n,m}^{U(n)}$ is generated by
   $\varphi^{(2k+2)}$,
   $\psi^{(0,2k,0)}_{ba}$,
   $\psi^{(0,2k,1)}_{ba}$,
   and
   $\psi^{(1,2k,0)}_{ba}$.
   Here the indices run as follows\/{\rm :}
   $$
      0 \leq k \leq n-1, \qquad
      1\leq a,b \leq m.
   $$
\end{theorem}
\vskip 2mm\noindent
{\it Proof.} See Theorem 3.3 in \cite{IOY}. \hfill $\square$

\vskip 0.5cm
Problem 2, that is, the second fundamental theorem for $\operatorname{Pol}_{n,m}^{U(n)}$ is stated as follows.
We consider indeterminates $\tilde{\omega}^{(2k+2)}$
and $\tilde{\psi}^{(\varepsilon,2k,\varepsilon')}_{ba}$
corresponding to $\omega^{(2k+2)}$ and $\psi^{(\varepsilon,2k,\varepsilon')}_{ba}$,
respectively.
For these, we assume the relations
$$
   \tilde{\psi}^{(1,2k,1)}_{ba} = \tilde{\psi}^{(0,2k+2,0)}_{ab}, \qquad
   \tilde{\psi}^{(1,2k,0)}_{ab} = \tilde{\psi}^{(1,2k,0)}_{ba}, \qquad
   \tilde{\psi}^{(0,2k,1)}_{ab} = \tilde{\psi}^{(0,2k,1)}_{ba}.
$$
We denote by $\tilde{\mathcal{Q}}$ the polynomial algebra in the following indeterminates:
$$
   \tilde{\omega}^{(2k+2)}, \quad
   \tilde{\psi}^{(0,2k,0)}_{ba}, \quad
   \tilde{\psi}^{(0,2k,1)}_{ba}, \quad
   \tilde{\psi}^{(1,2k,0)}_{ba}.
$$
Here the indices run as follows:
$$
   0 \leq k \leq n-1, \qquad
    1\leq a,b\leq m.
$$
The relations among the generators of $\operatorname{Pol}_{n,m}^{U(n)}$
are described as follows:

\begin{theorem}
   The kernel of the natural map from $\tilde{\mathcal{Q}}$ to $\operatorname{Pol}_{n,m}^{U(n)}$
   is generated by the entries of $A^{(q)}_{(c,\varepsilon),(c',\varepsilon')} B^{(q)}$
   with
   $$
      q \in \{ 2,3,\ldots,n+1 \}, \qquad\quad
      \varepsilon = (\varepsilon_1,\ldots,\varepsilon_q), \quad
      \varepsilon' = (\varepsilon'_1,\ldots,\varepsilon'_q) \in \{ 0,1 \}^q,
   $$
   $$
      c = (c_1,\ldots,c_q), \quad
      c' = (c'_1,\ldots,c'_q) \in \{1,\cdots,m\}^q.
   $$
\end{theorem}

Here the notation is as follows.
We put
$$
   \Delta^{(q),(\lambda_1,\lambda_2)}_{(c,\varepsilon),(c',\varepsilon')} =
   \sum_{l_1 + \cdots + l_q = \lambda_1 + \lambda_2}
   K_{(\lambda_1,\lambda_2),(l_1,\ldots,l_q)}
   \sum_{\sigma \in S_q} \operatorname{sgn}(\sigma)
   \tilde{\psi}^{(\varepsilon_1,2l_1,\varepsilon'_1)}_{c_1 c'_{\sigma(1)}}
   \cdots \tilde{\psi}^{(\varepsilon_q,2l_q,\varepsilon'_q)}_{c_q c'_{\sigma(q)}}.
$$
Here $K_{\lambda,\mu}$  means the Kostka number.
Namely, in general, we define $K_{\lambda,\mu}$ by
$$
   s_{\lambda}(u_1,\ldots,u_q)
   = \sum_{l_1,\ldots,l_q \geq 0} K_{\lambda,(l_1,\ldots,l_q)} u_1^{l_1} \cdots u_q^{l_q},
$$
where $s_{\lambda}$ is the Schur polynomial.
In other words, $\Delta^{(q),(\lambda_1,\lambda_2)}_{(c,\varepsilon),(c',\varepsilon')}$ is
the image of the Schur polynomial $s_{(\lambda_1,\lambda_2)}(u_1,\ldots,u_q)$
under the linear map
$$
   u_1^{l_1} \cdots u_q^{l_q} \mapsto
   \sum_{\sigma \in S_q} \operatorname{sgn}(\sigma)
   \tilde{\psi}^{(\varepsilon_1,2l_1,\varepsilon'_1)}_{c_1c'_{\sigma(1)}} \cdots
   \tilde{\psi}^{(\varepsilon_q,2l_q,\varepsilon'_q)}_{c_qc'_{\sigma(q)}}.
$$

Moreover we replace $\tilde{\psi}^{(1,2n-2,1)}_{ba}$
in $\Delta^{(q),(\lambda_1,\lambda_2)}_{(c,\varepsilon),(c',\varepsilon')}$ by
$$
   \sum_{k=1}^n (-1)^{k-1} \tilde{\omega}^{(2k)} \tilde{\psi}^{0,2n-2k,0}_{ab}.
$$
Finally $A^{(q)}_{(c,\varepsilon),(c',\varepsilon')}$ and $B^{(q)}$ are the following matrices
(an alternating matrix of size $q'+2$ and a $(q'+2) \times 1$ matrix):
\begin{align*}
   A^{(q)}_{(c,\varepsilon),(c',\varepsilon')} &=
   \begin{pmatrix}
   0                            & \Delta^{(q),(0,0)}_{(c,\varepsilon),(c',\varepsilon')}   & \Delta^{(q),(1,0)}_{(c,\varepsilon),(c',\varepsilon')}   & \hdots & \Delta^{(q),(q',0)}_{(c,\varepsilon),(c',\varepsilon')} \\[5pt]
   -\Delta^{(q),(0,0)}_{(c,\varepsilon),(c',\varepsilon')}  & 0                            & \Delta^{(q),(1,1)}_{(c,\varepsilon),(c',\varepsilon')}   & \hdots & \Delta^{(q),(q',1)}_{(c,\varepsilon),(c',\varepsilon')} \\[5pt]
   -\Delta^{(q),(1,0)}_{(c,\varepsilon),(c',\varepsilon')}  & -\Delta^{(r),(1,1)}_{(c,\varepsilon),(c',\varepsilon')}  & 0                            & \hdots & \Delta^{(q),(q',2)}_{(c,\varepsilon),(c',\varepsilon')} \\[5pt]
   \vdots                       & \vdots                       & \vdots                       & \ddots & \vdots \\[5pt]
   -\Delta^{(q),(q',0)}_{(c,\varepsilon),(c',\varepsilon')} & -\Delta^{(q),(q',1)}_{(c,\varepsilon),(c',\varepsilon')} & -\Delta^{(q),(c',2)}_{(c,\varepsilon),(c',\varepsilon')} & \hdots & 0
   \end{pmatrix}, \\[5pt]
   B^{(q)} &=
   \begin{pmatrix}
   (-1)^{q'}\tilde{\omega}^{(q')} \\
   \vdots \\
   \tilde{\omega}^{(2)} \\
   -\tilde{\omega}^{(1)} \\
   \tilde{\omega}^{(0)}
   \end{pmatrix}.
\end{align*}
Here we put $q' = n+1-q$.
%%Note that $A^{(q)}_{(a,\varepsilon),(a',\varepsilon')} = 0$,
%%if we have $a_i = a_j$ and $\varepsilon_i = \varepsilon_j$
%%(or $a'_i = a'_j$ and $\varepsilon'_i = \varepsilon'_j$)
%%for some $i \ne j$.

\vskip 0.3cm
The proof of Theorem 3.4 is complicated, but it is deduced from the second fundamental theorem
of invariant theory for vector invariants (this is quite parallel with the fact that Theorem 3.3
follows from the first fundamental theorem of invariant theory for vector invariants). The detail will
be given in the forthcoming paper.

\begin{remark}
Itoh, Ochiai and Yang \cite{IOY} solved all the problems (Problem 1--Problem 7) proposed in this section when $n=m=1.$
\end{remark}

\vskip 0.3cm We present some interesting $U(n)$-invariants.
For an $m\times m$ matrix $S$, we define the
following invariant polynomials in $\text{Pol}_{n,m}^{U(n)}$:
\begin{eqnarray*}
&&
m_{j;S}^{(1)}(\om,z)=\,\textrm{Re}\,\Big(\text{tr}\big(\om{\overline
\om}+ \,^tzS{\overline
z}\,\big)^j\,\Big),\quad 1\leq j\leq n,\\
&&
m_{j;S}^{(2)}(\om,z)=\,\textrm{Im}\,\Big(\text{tr}\big(\om{\overline
\om}+ \,^tzS{\overline
z}\,\big)^j\,\Big),\quad 1\leq j\leq n,\\
&&  q_{k;S}^{(1)}(\om,z)=\,\textrm{Re}\,\Big( \textrm{tr}\big( (
\,^tz\,S\,{\overline
z})^k\big) \Big),\quad 1\leq k\leq m,  \\
&&  q_{k;S}^{(2)}(\om,z)=\,\textrm{Im}\,\Big( \textrm{tr}\big( (
\,^tz\,S\,{\overline
z})^k\big) \Big),\quad 1\leq k\leq m,  \\
&& \theta_{i,k,j;S}^{(1)}(\om,z) =\,\textrm{Re}\,\Big(
\textrm{tr}\big( (\om {\overline \om})^i\,(\,^tz\,S\,{\overline
z})^k\,(\om {\overline \om}+\,^tz\,S\,{\overline
z}\,)^j\,\big)\Big),\\
&& \theta_{i,k,j;S}^{(2)}(\om,z) =\,\textrm{Im}\,\Big(
\textrm{tr}\big( (\om {\overline \om})^i\,(\,^tz\,S\,{\overline
z})^k\,(\om {\overline \om}+\,^tz\,S\,{\overline
z}\,)^j\,\big)\Big),
\end{eqnarray*}
\noindent where $1\leq i,j\leq n$ and $1\leq k\leq m$.

\vskip 0.2cm We define the following $U(n)$-invariant polynomials in
$\text{Pol}_{n,m}^{U(n)}$.
\begin{eqnarray*}
&&  r_{jk}^{(1)}(\om,z)= \,\textrm{Re}\,\Big( \textrm{det}\big(
(\om {\overline \om})^j\,(\,^tz{\overline z})^k\,\big)\Big),
\quad 1\leq j\leq n,\ 1\leq k\leq m,\\
&&  r_{jk}^{(2)}(\om,z)= \,\textrm{Im}\,\Big( \det\big(
(\om {\overline \om})^j\,(\,^tz{\overline z})^k\,\big)\Big), \quad
1\leq j\leq n,\ 1\leq k\leq m.
\end{eqnarray*}

\end{section}

\vskip 10mm

\begin{section}{{\bf The Partial Cayley Transform}}
\setcounter{equation}{0}

\newcommand\OW{\overline{W}}
\newcommand\OP{\overline{P}}
\newcommand\OQ{\overline{Q}}
\newcommand\OZ{\overline{Z}}
\newcommand\Dn{{\mathbb D}_n}
\newcommand\Dnm{{\mathbb D}_{n,m}}
\newcommand\Hn{{\mathbb H}_n}
\newcommand\SJ{{\mathbb H}_n\times {\mathbb C}^{(m,n)}}
\newcommand\DC{{\mathbb D}_n\times {\mathbb C}^{(m,n)}}
\newcommand\ot{\overline\eta}
\newcommand\PW{ {{\partial}\over {\partial W}} }
\newcommand\PZB{ {{\partial}\over {\partial{\overline Z}}} }
\newcommand\PWB{ {{\partial}\over {\partial{\overline W}}} }
\newcommand\PE{ {{\partial}\over {\partial \eta}} }
\newcommand\POB{ {{\partial}\over {\partial{\overline \Omega}}} }
\newcommand\PEB{ {{\partial}\over {\partial{\overline \eta}}} }

\vskip 0.21cm
Let
\begin{equation*}
\BD_n=\left\{ W\in \BC^{(n,n)}\,|\ W=\,{}^tW,\ I_n-\OW W >
0\,\right\}
\end{equation*}
be the generalized unit disk. We set
\begin{equation*}
G_*=T^{-1} Sp(n,\BR)\, T,\qquad T:={1\over {\sqrt 2}}\, \begin{pmatrix} I_n & \ I_n \\ i I_n & -iI_n \end{pmatrix}.
\end{equation*}
It is easily seen that
\begin{equation*}
G_*=\left\{\, \begin{pmatrix} P & Q \\ {\overline Q} & {\overline P} \end{pmatrix}\in \BC^{(2n,2n)}\,\Big|\ {}^tP {\overline P} - \,{}^t{\overline Q} Q=I_n,
\ {}^tP {\overline Q}= \,{}^t{\overline Q} P \ \right\}.
\end{equation*}
Then $G_*$ acts on $\BD_n$ transitively by
\begin{equation}
\begin{pmatrix} P & Q \\ {\overline Q} & {\overline P} \end{pmatrix}\cdot W =(PW+Q)({\overline Q}W+{\overline P})^{-1},\qquad
\begin{pmatrix} P & Q \\ {\overline Q} & {\overline P} \end{pmatrix}\in G_*,\ W\in \BD_n.
\end{equation}
It is well known that the action (1.1) is compatible with the action (4.1) through the Cayley transform $\Phi:\BD_n \lrt \BH_n$ given by
\begin{equation}
\Phi (W):=i\,(I_n+W)(I_n-W)^{-1},\qquad W\in \BD_n.
\end{equation}
In other words, if $M\in Sp(n,\BR)$ and $W\in \BD_n,$ then
\begin{equation}
M\!\cdot\! \Phi (W)=\Phi (M_*\!\cdot\! W),
\end{equation}
where $M_*=T^{-1} M T$. We refer to \cite{KW} for generalized Cayley transforms of bounded symmetric domains.

\vskip 0.2cm For brevity, we write $\Dnm:=\DC.$  This homogeneous space $\Dnm$ is called
the {\it Siegel-Jacobi disk} of degree $n$ and index $m$. For a coordinate
$(W,\eta)\in\Dnm$ with $W=(w_{\mu\nu})\in {\mathbb D}_n$ and
$\eta=(\eta_{kl})\in \Cmn,$ we put
\begin{eqnarray*}
dW\,&=&\,(dw_{\mu\nu}),\quad\ \ d{\overline W}\,=\,(d{\overline w}_{\mu\nu}),\\
d\eta\,&=&\,(d\eta_{kl}),\quad\ \
d{\overline\eta}\,=\,(d{\overline\eta}_{kl})
\end{eqnarray*}
and
\begin{eqnarray*}
\PW\,=\,\left(\, { {1+\delta_{\mu\nu}} \over 2}\, {
{\partial}\over {\partial w_{\mu\nu}} } \,\right),\quad
\PWB\,=\,\left(\, { {1+\delta_{\mu\nu}}\over 2} \, {
{\partial}\over {\partial {\overline w}_{\mu\nu} }  } \,\right),
\end{eqnarray*}

$$\PE=\begin{pmatrix} {\partial}\over{\partial \eta_{11}} & \hdots &
 {\partial}\over{\partial \eta_{m1}} \\
\vdots&\ddots&\vdots\\
 {\partial}\over{\partial \eta_{1n}} &\hdots & {\partial}\over
{\partial \eta_{mn}} \end{pmatrix},\quad \PEB=\begin{pmatrix}
{\partial}\over{\partial {\overline \eta}_{11} }   &
\hdots&{ {\partial}\over{\partial {\overline \eta}_{m1} }  }\\
\vdots&\ddots&\vdots\\
{ {\partial}\over{\partial{\overline \eta}_{1n} }  }&\hdots &
 {\partial}\over{\partial{\overline \eta}_{mn} }  \end{pmatrix}.$$

\vskip 0.2cm  We can identify an element $g=(M,(\la,\mu;\kappa))$
of $G^J,\ M=\begin{pmatrix} A&B\\
C&D\end{pmatrix}\in Sp(n,\BR)$ with the element
\begin{equation*}
\begin{pmatrix} A & 0 & B & A\,^t\mu-B\,^t\la  \\ \la & I_m & \mu
& \kappa \\ C & 0 & D & C\,^t\mu-D\,^t\la \\ 0 & 0 & 0 & I_m
\end{pmatrix}
\end{equation*}
of $Sp(m+n,\BR).$ \vskip 0.3cm We set
\begin{equation*}
T_*={1\over {\sqrt 2}}\,
\begin{pmatrix} I_{m+n} & I_{m+n}\\ iI_{m+n} & -iI_{m+n}
\end{pmatrix}.
\end{equation*}
We now consider the group $G_*^J$ defined by
\begin{equation*}
G_*^J:=T_*^{-1}G^JT_*.
\end{equation*}
If $g=(M,(\la,\mu;\kappa))\in G^J$ with $M=\begin{pmatrix} A&B\\
C&D\end{pmatrix}\in Sp(n,\BR)$, then $T_*^{-1}gT_*$ is given by
\begin{equation*}
T_*^{-1}gT_*=
\begin{pmatrix} P_* & Q_*\\ {\overline Q}_* & {\overline P}_*
\end{pmatrix},
\end{equation*}
where
\begin{equation*}
P_*=
\begin{pmatrix} P & {\frac 12} \left\{ Q\,\,{}^t(\la+i\mu)-P\,\,{}^t(\la-i\mu)\right\}\\
{\frac 12} (\la+i\mu) & I_h+i{\frac \kappa 2}
\end{pmatrix},
\end{equation*}

\begin{equation*}
Q_*=
\begin{pmatrix} Q & {\frac 12} \left\{ P\,\,{}^t(\la-i\mu)-Q\,\,{}^t(\la+i\mu)\right\}\\
{\frac 12} (\la-i\mu) & -i{\frac \kappa 2}
\end{pmatrix},
\end{equation*}
and $P,\,Q$ are given by the formulas
\begin{equation}
P= {\frac 12}\,\left\{ (A+D)+\,i\,(B-C)\right\}
\end{equation}
and
\begin{equation}
 Q={\frac
12}\,\left\{ (A-D)-\,i\,(B+C)\right\}.
\end{equation}
From now on, we write
\begin{equation*}
\left(\begin{pmatrix} P & Q\\ {\overline Q} & {\overline P}
\end{pmatrix},\left( {\frac 12}(\la+i\mu),\,{\frac 12}(\la-i\mu);\,-i{\kappa\over 2}\right)\right):=
\begin{pmatrix} P_* & Q_*\\ {\overline Q}_* & {\overline P}_*
\end{pmatrix}.
\end{equation*}
In other words, we have the relation
\begin{equation*}
T_*^{-1}\left( \begin{pmatrix} A&B\\
C&D\end{pmatrix},(\la,\mu;\kappa)
\right)T_*=  \left(\begin{pmatrix} P & Q\\
{\overline Q} & {\overline P}
\end{pmatrix},\left(
{\frac 12}(\la+i\mu),\,{\frac 12}(\la-i\mu);\,-i{\kappa\over 2}
\right)\right).
\end{equation*}
Let
\begin{equation*}
H_{\BC}^{(n,m)}:=\left\{ (\xi,\eta\,;\zeta)\,|\
\xi,\eta\in\BC^{(m,n)},\ \zeta\in\BC^{(m,m)},\
\zeta+\eta\,{}^t\xi\ \textrm{symmetric}\,\right\}
\end{equation*}
be the complex Heisenberg group endowed with the following
multiplication
\begin{equation*}
(\xi,\eta\,;\zeta)\circ
(\xi',\eta';\zeta'):=(\xi+\xi',\eta+\eta'\,;\zeta+\zeta'+
\xi\,{}^t\eta'-\eta\,{}^t\xi')).
\end{equation*}
We define the semidirect product
\begin{equation*}
%SLH(2g,\BC):=
SL(2n,\BC)\ltimes H_{\BC}^{(n,m)}
\end{equation*}
endowed with the following multiplication
\begin{eqnarray*}
& & \left( \begin{pmatrix} P & Q\\ R & S
\end{pmatrix}, (\xi,\eta\,;\zeta)\right)\cdot \left( \begin{pmatrix} P' & Q'\\
R' & S'
\end{pmatrix}, (\xi',\eta';\zeta')\right)\\
&=& \left( \begin{pmatrix} P & Q\\ R & S
\end{pmatrix}\,\begin{pmatrix} P' & Q'\\ R' & S'
\end{pmatrix},\,({\tilde \xi}+\xi',{\tilde
\eta}+\eta';\zeta+\zeta'+{\tilde \xi}\,{}^t\eta'-{\tilde
\eta}\,{}^t\xi')  \right),
\end{eqnarray*}
where ${\tilde\xi}=\xi P'+\eta R'$ and ${\tilde \eta}=\xi Q'+\eta
S'.$

\vskip 0.2cm If we identify $H_{\BR}^{(n,m)}$ with the subgroup
$$\left\{ (\xi,{\overline \xi};i\kappa)\,|\ \xi\in\BC^{(m,n)},\
\ka\in\BR^{(m,m)}\,\right\}$$ of $H_{\BC}^{(n,m)},$ we have the
following inclusion
$$G_*^J\subset SU(n,n)\ltimes H_{\BR}^{(n,m)}\subset SL(2n,\BC)\ltimes
H_{\BC}^{(n,m)}.$$ We define the mapping $\Theta:G^J\lrt G_*^J$ by
\begin{equation}\Theta\!
\left(\! \begin{pmatrix} A&B\\
C&D\end{pmatrix}\!,(\la,\mu;\kappa) \right)\!=\!\left( \!\begin{pmatrix} P
& Q\\ {\overline Q} & {\overline P}
\end{pmatrix},\left(
{\frac 12}(\la+i\mu),\,{\frac 12}(\la-i\mu);\,-i{\kappa\over 2}
\right)\!\right),\end{equation}
where $P$ and $Q$ are given by (4.4)
and (4.5). We can see that if $g_1,g_2\in G^J$, then
$\Theta(g_1g_2)=\Theta(g_1)\Theta(g_2).$

\vskip 0.2cm According to \cite[p.\,250]{YJH7}, $G_*^J$ is of the
Harish-Chandra type\,(cf.\,\cite[p.\,118]{Sa2}). Let
$$g_*=\left(\begin{pmatrix} P & Q\\
{\overline Q} & {\overline P}
\end{pmatrix},\left( \la, \mu;\,\kappa\right)\right)$$
be an element of $G_*^J.$ Since the Harish-Chandra decomposition
of an element $\begin{pmatrix} P & Q\\ R & S
\end{pmatrix}$ in $SU(n,n)$ is given by
\begin{equation*}
\begin{pmatrix} P & Q\\ R & S
\end{pmatrix}=\begin{pmatrix} I_n & QS^{-1}\\ 0 & I_n
\end{pmatrix} \begin{pmatrix} P-QS^{-1}R & 0\\ 0 & S
\end{pmatrix} \begin{pmatrix} I_n & 0\\ S^{-1}R & I_n
\end{pmatrix},
\end{equation*}
the $P_*^+$-component of the following element
$$g_*\cdot\left( \begin{pmatrix} I_n & W\\ 0 & I_n
\end{pmatrix}, (0,\eta;0)\right),\quad W\in \BD_n$$
of $SL(2n,\BC)\ltimes H_{\BC}^{(n,m)}$ is given by
\begin{equation}
\left( \! \begin{pmatrix} I_n & (PW+Q)(\OQ W+\OP)^{-1}
\\ 0 & I_n
\end{pmatrix}\!,\left(0,\,(\eta+\la W+\mu)(\OQ W+\OP)^{-1}\,;0\right)\!\right).
\end{equation}

\vskip 0.2cm We can identify $\Dnm$ with the subset
\begin{equation*}
\left\{ \left( \begin{pmatrix} I_n & W\\ 0 & I_n
\end{pmatrix}, (0,\eta;0)\right)\,\Big|\ W\in\BD_n,\
\eta\in\BC^{(m,n)}\,\right\}\end{equation*} of the
complexification of $G_*^J.$ Indeed, $\Dnm$ is embedded into
$P_*^+$ given by
\begin{equation*}
P_*^+=\left\{\,\left( \begin{pmatrix} I_n & W\\ 0 & I_n
\end{pmatrix}, (0,\eta;0)\right)\,\Big|\ W=\,{}^tW\in \BC^{(n,n)},\
\eta\in\BC^{(m,n)}\ \right\}.
\end{equation*}
This is a generalization of the Harish-Chandra
embedding\,(cf.\,\cite[p.\,119]{Sa2}). Then we get the {\it natural
transitive action} of $G_*^J$ on $\Dnm$ defined by
\begin{eqnarray}
& &\left(\begin{pmatrix} P & Q\\
{\overline Q} & {\overline P}
\end{pmatrix},\left( \xi, {\overline\xi};\,i\kappa\right)\right)\cdot
(W,\eta)\\
&=&\Big((PW+Q)(\OQ W+\OP)^{-1},(\eta+\xi
W+{\overline\xi})(\OQ W+\OP)^{-1}\Big),\nonumber
\end{eqnarray}

\noindent where $\begin{pmatrix} P & Q\\
{\overline Q} & {\overline P}
\end{pmatrix}\in G_*,\ \xi\in \BC^{(m,n)},\ \kappa\in\BR^{(m,m)}$ and
$(W,\eta)\in\Dnm.$

\vskip 0.2cm The author \cite{YJH10} proved that the action (1.2) of $G^J$
on $\Hnm$ is compatible with the action (4.8) of $G_*^J$ on $\Dnm$
through the {\it partial Cayley transform} $\Psi:\BD_{n,m}\lrt
\BH_{n,m}$ defined by
\begin{equation}
\Psi(W,\eta):=\Big(
i(I_n+W)(I_n-W)^{-1},\,2\,i\,\eta\,(I_n-W)^{-1}\Big).
\end{equation}
In other words, if $g_0\in G^J$ and $(W,\eta)\in\BD_{n,m}$,
\begin{equation}
g_0\cdot\Psi(W,\eta)=\Psi(g_*\cdot (W,\eta)),
\end{equation}
where $g_*=T_*^{-1}g_0 T_*$. $\Psi$ is a biholomorphic mapping of
$\Dnm$ onto $\Hnm$ which gives the partially bounded realization
of $\Hnm$ by $\Dnm$. The inverse of $\Psi$ is
\begin{equation}
\Psi^{-1}(\Omega,Z)=\Big(
(\Omega-iI_n)(\Omega+iI_n)^{-1},\,Z(\Omega+iI_n)^{-1}\Big).
\end{equation}

\end{section}

\newpage
\begin{section}{{\bf Invariant Metrics and Laplacians on the Siegel-Jacobi Disk}}
\setcounter{equation}{0}
\newcommand\Dnm{{\mathbb D}_{n,m}}
\newcommand\OW{\overline{W}}
\newcommand\OP{\overline{P}}
\newcommand\OQ{\overline{Q}}
\newcommand\OZ{\overline{Z}}
\newcommand\ot{\overline\eta}

\newcommand\bw{d{\overline W}}
\newcommand\bz{d{\overline Z}}
\newcommand\bo{d{\overline \Omega}}

\newcommand\PE{ {{\partial}\over {\partial \eta}} }
\newcommand\POB{ {{\partial}\over {\partial{\overline \Omega}}} }
\newcommand\PEB{ {{\partial}\over {\partial{\overline \eta}}} }

\newcommand\PW{ {{\partial}\over{\partial W}} }
\newcommand\PWB{ {{\partial}\over {\partial{\overline W}}} }
\newcommand\OVW{\overline W}

\vskip 0.2cm
For $W=(w_{ij})\in \BD_n,$ we write $dW=(dw_{ij})$ and
$d{\overline W}=(d{\overline{w}}_{ij})$. We put $$\PW=\,\left(\, {
{1+\delta_{ij}}\over 2}\, { {\partial}\over {\partial w_{ij} } }
\,\right) \qquad\text{and}\qquad \PWB=\,\left(\, {
{1+\delta_{ij}}\over 2}\, { {\partial}\over {\partial {\overline
{w}}_{ij} } } \,\right).$$

Using the Cayley transform $\Psi:\BD_n\lrt \BH_n$, Siegel \cite{Si1} showed that
\begin{equation}
ds_*^2=4\, \s \Big((I_n-W{\overline W})^{-1}dW\,(I_n-\OVW
W)^{-1}d\OVW\,\Big)\end{equation} is a $G_*$-invariant Riemannian
metric on $\BD_n$ and Maass \cite{M1} showed that its Laplacian is
given by
\begin{equation}
\Delta_*=\,\s \left( (I_n-W\OW)\,
{}^{{}^{{}^{{}^\text{\scriptsize $t$}}}}\!\!\!\left( (I_n-W\OVW)
\PWB\right)\PW\right).\end{equation}

\vskip 2mm
Yang \cite{YJH11} proved the following theorems.
\begin{theorem}
For any two positive real numbers $A$ and $B$,
the following metric $d{\tilde s}^2_{n,m;A,B}$ defined by
\begin{eqnarray*}
d{s}^2_{\Dnm;A,B}&=&4\,A\, \sigma  \Big( (I_n-W\OW)^{-1}dW(I_n-\OW W)^{-1}\bw\,\Big) \hskip 1cm\\
& &\,+\,4\,B\,\bigg\{ \sigma   \Big(
(I_n-W\OW)^{-1}\,{}^t(d\eta)\,\be\,\Big)\\
& & \quad\quad\quad
\,+\,\sigma   \Big(  (\eta\OW-{\overline\eta})(I_n-W\OW)^{-1}dW(I_n-\OW W)^{-1}\,{}^t(d\ot)\Big)\\
& & \quad\quad\quad  +\,\sigma   \Big( (\ot W-\eta)(I_n-\OW
W)^{-1}d\OW(I_n-W\OW)^{-1}\,{}^t(d\eta)\,\Big)    \\
& &\quad\quad\quad -\, \sigma   \Big( (I_n-W\OW)^{-1}\,{}^t\eta\,\eta\,
(I_n-\OW W)^{-1}\OW
dW (I_n-\OW W)^{-1}d\OW \, \Big)\\
& &\quad\quad\quad -\, \sigma   \Big( W(I_n-\OW W)^{-1}\,{}^t\ot\,\ot\,
(I_n-W\OW )^{-1}
dW (I_n-\OW W)^{-1}d\OW \,\Big)\\
& &\quad\quad\quad +\,\sigma    \Big( (I_n-W\OW)^{-1}{}^t\eta\,\ot \,(I_n-W\OW)^{-1} dW (I_n-\OW W)^{-1} d\OW\,\Big)\\
& &\quad\quad\quad +\,\sigma
\Big( (I_n-\OW)^{-1}\,{}^t\ot\,\eta\,\OW\,(I_n-W\OW)^{-1} dW (I_n-\OW W)^{-1} d\OW\,\Big)\\
& &\quad\quad\quad +\,\sigma \!  \Big( (I_n-\OW)^{-1}(I_n-W)(I_n-\OW
W)^{-1}\,{}^t\ot\,\eta\,(I_n-\OW W)^{-1}\\
& &\qquad\qquad\quad\quad \times\, (I_n-\OW)(I_n-W)^{-1}dW
(I_n-\OW W)^{-1}d\OW\,\Big)\\
& &\quad\quad\quad -\,\sigma
\Big( (I_n-W\OW)^{-1}(I_n-W)(I_n-\OW)^{-1}\,{}^t\ot\,\eta\,(I_n-W)^{-1}\\
& & \qquad\qquad\quad\quad \times\,dW (I_n-\OW
W)^{-1}d\OW\,\Big)\bigg\}
\end{eqnarray*}
\noindent is a Riemannian metric on $\Dnm$ which is invariant
under the action (4.8) of the Jacobi group $G^J_*$.
\end{theorem}
\vskip 2mm\noindent
{\it Proof.} See Theorem 1.3 in \cite{YJH11}.  \hfill $\Box$
%%%%%%%%%%%%%%%%%%%%%%%%%%%%%%%%%%%%%%%%%%%%%%%%%%%%%%%%%%%%%%%%%%%%%%%%%%%%%%%%%%%%%%%%%%%%%%%%%%%%%%%%%%%%%%%%%%%%%%%%%%%%%%%%%%%%%%%%%
%%%%%%%%%%%%%%%%%%%%%%%%%%%%%%%%%%%%%%%%%%%%%%%%%%%%%%%%%%%%%%%%%%%%%%%%%%%%%%%%%%%%%%%%%%%%%%%%%%%%%%%%%%%%%%%%%%%%%%%%%%%%%%%%%%%%%%%%%
%%%%%%%%%%%%%%%%%%%%%%%%%%%%%%%%%%%%%%%%%%%%%%%%%%%%%%%%%%%%%%%%%%%%%%%%%%%%%%%%%%%%%%%%%%%%%%%%%%%%%%%%%%%%%%%%%%%%%%%%%%%%%%%%%%%%%%%%%%
%%%%%%%%%%%%%%%%%%%%%%%%%%%%%%%%%%%%%%%%%%%%%%%%%%%%%%%%%%%%%%%%%%%%%%%%%%%%%%%%%%%%%%%%%%%%%%%%%%%%%%%%%%%%%%%%%%%%%%%%%%%%%%%%%%%%%%%%%%
%%%%%%%%%%%%%%%%%%%%%%%%%%%%%%%%%%%%%%%%%%%%%%%%%%%%%%%%%%%%%%%%%%%%%%%%%%%%%%%%%%%%%%%%%%%%%%%%%%%%%%%%%%%%%%%%%%%%%%%%%%%%%%%%%%%%%%%%%%
\newcommand\pdx{ {{\partial}\over{\partial x}} }
\newcommand\pdy{ {{\partial}\over{\partial y}} }
\newcommand\pdu{ {{\partial}\over{\partial u}} }
\newcommand\pdv{ {{\partial}\over{\partial v}} }
\newcommand\PZ{ {{\partial}\over {\partial Z}} }
\newcommand\PZB{ {{\partial}\over {\partial{\overline Z}}} }
\newcommand\PX{ {{\partial\ }\over{\partial X}} }
\newcommand\PY{ {{\partial\ }\over {\partial Y}} }
\newcommand\PU{ {{\partial\ }\over{\partial U}} }
\newcommand\PV{ {{\partial\ }\over{\partial V}} }
\renewcommand\th{\theta}
\renewcommand\l{\lambda}
\renewcommand\k{\kappa}

\begin{theorem}
The following
differential operators ${\mathbb S}_1$ and ${\mathbb S}_2$ on $\Dnm$ defined by
\begin{equation*}
{\mathbb S}_1=\,\s\left( (I_n-\OW W)\PE {}^{{}^{{}^{{}^\text{\scriptsize $t$}}}}\!\!\!\left(\PEB\right)\right)
\end{equation*}
\noindent and
\begin{eqnarray*}
{\mathbb S}_2&=& \, \sigma \left( (I_n-W\OW)\,{}^{{}^{{}^{{}^\text{\scriptsize $t$}}}}\!\!\!\left(
(I_n-W\OW)\PWB\right)\PW\right)\,\\
& &   +\,\sigma \left(\,{}^t(\eta-\ot\,W)\,{}^{{}^{{}^{{}^\text{\scriptsize $t$}}}}\!\!\!\left( \PEB\right)
(I_n-\OW W)\PW  \right)\,\\
& &  +\, \sigma \left( (\ot-\eta\,\OW)\,{}^{{}^{{}^{{}^\text{\scriptsize $t$}}}}\!\!\!\left(
(I_n-W\OW)\PWB\right)\PE\right)\\
& & -\, \sigma  \left( \eta \OW
(I_n-W\OW)^{-1}\,{}^t\eta\,{}^{{}^{{}^{{}^\text{\scriptsize $t$}}}}\!\!\!\left(\PEB\right)(I_n-\OW
W)\PE\right)\\
& & -\, \sigma  \left( \ot W (I_n-\OW W)^{-1}
\,{}^t\ot\,{}^{{}^{{}^{{}^\text{\scriptsize $t$}}}}\!\!\!\left(\PEB\right)(I_n-\OW
W)\PE\right)\\
& & +\, \sigma  \left( \ot (I_n-W\OW)^{-1}{}^t\eta\,{}^{{}^{{}^{{}^\text{\scriptsize $t$}}}}\!\!\!\left(
\PEB\right)
(I_n-\OW W)\PE \right)\\
& &  +\, \sigma  \left( \eta\,\OW W (I_n-\OW
W)^{-1}\,{}^t\ot\,{}^{{}^{{}^{{}^\text{\scriptsize $t$}}}}\!\!\!\left( \PEB\right) (I_n-\OW W)\PE
\right)
\end{eqnarray*}

\noindent are invariant under the action (4.8) of $G_*^J.$ The following differential operator
\begin{equation}
\Delta_{{\mathbb D}_{n,m};A,B}:=\,{\frac 1A}\,{\mathbb S}_2\,+\,{\frac 1B}\,{\mathbb S}_1
\end{equation}
is the Laplacian of the invariant metric $ds^2_{{\mathbb D}_{n,m};A,B}$ on $\Dnm$.
\end{theorem}
\vskip 2mm\noindent
{\it Proof.} See Theorem 1.4 in \cite{YJH11}.  \hfill $\Box$

\vskip 0.5cm
Itoh, Ochiai and Yang \cite{IOY} proved that the following differential operator on $\Dnm$ defined by
\begin{equation*}
{\mathbb S}_3=\,\det(I_n-{\overline W}W)\,\det\left( \PE {}^{{}^{{}^{{}^\text{\scriptsize $t$}}}}\!\!\!\left(
\PEB\right)\right)
\end{equation*}
\noindent
is invariant under the action (4.8) of $G^J_*$ on $\Dnm$.
Furthermore the authors \cite{IOY} proved that the following matrix-valued differential operator on $\Dnm$ defined by
\begin{equation*}
{\mathbb J}:=\,
{}^{{}^{{}^{{}^\text{\scriptsize $t$}}}}\!\!\!
\left( \PEB\right) (I_n-{\overline W}W) \PE
\end{equation*}
\noindent
and each $(k,l)$-entry ${\mathbb J}_{kl}$ of ${\mathbb J}$ given by
\begin{equation*}
{\mathbb J}_{kl}=\sum_{i,j=1}^n
\,\left( \delta_{ij}-\sum_{r=1}^n {\overline w}_{ir}\,w_{jr}\right)
\,{{\partial^2\ \ \ \ }\over{\partial {\overline \eta}_{ki}\partial
\eta_{lj}} },\quad 1\leq k,l\leq m
\end{equation*}
\noindent
are invariant under the action (4.8) of $G^J_*$ on $\Dnm$.

\vskip 0.3cm
\begin{equation*}
{\mathbb S}_*=\,[{\mathbb S}_1,{\mathbb S}_2]=\,{\mathbb S}_1{\mathbb S}_2-{\mathbb S}_2{\mathbb S}_1
\end{equation*}
\noindent
is an invariant differential operator of degree three on $\Dnm$ and
\begin{equation*}
{\mathbb Q}_{kl}=\,[{\mathbb S}_3,{\mathbb J}_{kl}]=
\,{\mathbb S}_3{\mathbb J}_{kl}-{\mathbb J}_{kl}{\mathbb S}_3,
\quad 1\leq k,l\leq m
\end{equation*}
\noindent
is an invariant differential operator of degree $2n+1$ on $\Dnm$.

\vskip 0.3cm
Indeed
it is very complicated and difficult at this moment to express the
generators of the algebra of all $G^J_{*}$-invariant differential
operators on $\Dnm$ explicitly.

\end{section}

\begin{section}{{\bf A Fundamental Domain for the Siegel-Jacobi Space}}
\setcounter{equation}{0}
Let
$$\mathscr P_n=\big\{ Y\in \BR^{(n,n)}\,|\ Y=\,^tY>0\ \big\}$$
be an open connected cone in $\BR^N$ with $N=n(n+1)/2.$ Then the general linear group $GL(n,\BR)$ acts on $\mathscr P_n$ transitively by
\begin{equation}
g\circ Y:=\,gY\,^tg,\quad g\in GL(n,\BR),\ Y\in \mathscr P_g.
\end{equation}
Thus $\mathscr P_n$ is a symmetric space diffeomorphic to $GL(n,\BR)/O(n).$

\vskip 2mm
The fundamental domain $\mathscr R_n$ for $GL(n,\BZ)\ba \mathscr P_n$ which was found by H. Minokwski \cite{Min} is defined as a subset of $\mathscr P_n$ consisting of
$Y=(y_{ij})\in \mathscr P_n$ satisfying the following conditions (M.1) and (M.2):
\vskip 1mm
(M.1) $aY\,^ta\geq y_{kk}$ for every $a=(a_i)\in\BZ^n$ in which $a_k,\cdots,a_n$ are relatively prime for
\vskip 1mm
$\quad\quad\ $ $k=1,2,\cdots,n.$
\vskip 1mm
(M.2) $y_{k,k+1}\geq 0$ for $k=1,\cdots,n-1.$

\vskip 1mm\noindent
We say that a point of $\mathscr R_n$ is {\it Minkowski reduced}.

\vskip 2mm
Let $\Gamma_n=Sp(n,\BZ)$ be the Siegel modular group of degree $n$. Siegel determined a fundamental domain $\mathscr F_n$ for $\Gamma_n \ba \BH_n.$
We say that $\Om=X+iY\in\BH_n$ with $X,\,Y$ real is {\it Siegel reduced} or {\it S-reduced} if it has the following three properties\,:
\vskip 1mm
(S.1) $\det ({\rm Im}(\gamma\cdot\Om))\leq \det({\rm Im}\,(\Om))\quad$ for all $\gamma\in \G_n$;
\vskip 1mm
(S.2) $Y={\rm Im}(\Om)$ is {\it Minkowski reduced}, that is, $Y\in \mathscr R_n$;
\vskip 1mm
(S.3) $|x_{ij}|\leq {1\over 2}$\ \ for $1\leq i,j\leq n$, where $X=(x_{ij}).$

\vskip 2.1mm
$\mathscr F_n$ is defined as the set of all Siegel reduced points in $\BH_n.$ Using the highest point method, Siegel proved the following (F1)-(F3):
\vskip 1mm
(F1) $\G_n\cdot \mathscr F_n=\BH_n$,\ i.e., $\BH_n= \bigcup_{\gamma\in \G_n} \gamma\!\cdot\!\mathscr F_n$ \,;
\vskip 1mm
(F2) $\mathscr F_n$ is closed in $\BH_n$\,;
\vskip 1mm
(F3) $\mathscr F_n$ is connected and the boundary of $\mathscr F_n$ consists of a finite number of hyperplanes.

\vskip 0.53cm
Let $E_{kj}$ be the $m\times n$ matrix with entry 1 where the $k$-th row and the $j$-the column meet, and all other entries 0. For an element $\Omega\in \BH_n$,
we set for brevity
$$ F_{kj}(\Omega):= E_{kj}\,\Omega,\qquad 1\leq k\leq m,\ 1\leq j\leq n.$$

\vskip 2mm
For each $\Omega\in \mathscr F_n$, we define the subset $P_\Omega$ of $\BC^{(m,n)}$ by

$$P_\Omega=\left\{ \sum_{k=1}^m\sum_{j=1}^n \lambda_{kj}E_{kj} +\sum_{k=1}^m\sum_{j=1}^n \mu_{kj}F_{kj}(\Omega)\,\Big|\ 0\leq \lambda_{kj},\mu_{kj}\leq 1\,\right\}.$$

\vskip 2mm\noindent
For each $\Omega\in \mathscr F_n$, we define the subset $D_\Omega$ of $\BH_{n,m}$ by
$$D_\Omega=\big\{ (\Omega,Z)\in \BH_{n,m}\,|\ Z \in P_\Omega\,\big\}.$$
Let
\begin{equation}
\G_{n,m}=\Gamma_n \ltimes H_\BZ^{(n,m)}
\end{equation}
be the Siegel-Jacobi (or simply Jacobi) modular group of degree $n$ and index $m$.

\vskip 2mm
Yang found a fundamental domain $\mathscr F_{n,m}$ for $\G_{n,m}\ba \BH_{n,m}$ using Siegel's fundamental domain $\mathscr F_n$ in \cite{YJH8}.

\begin{theorem}
 The set
$$ \mathscr F_{n,m}:=\bigcup_{\Omega\in\mathscr F_n} D_\Omega$$
is a fundamental domain for $\Gamma_{n,m}\backslash \BH_{n,m}.$
\end{theorem}
\vskip 1.2mm
\noindent
{\it Proof.} See Theorem 3.1 in \cite{YJH8}. \hfill$\square$

\end{section}

\vskip 1cm

\begin{section}{{\bf Jacobi Forms}}
\setcounter{equation}{0}

Let $\rho$ be a rational representation of
$GL(n,\mathbb{C})$ on a finite
dimensional complex vector space
$V_{\rho}.$ Let ${\mathcal M}\in \mathbb R^{(m,m)}$ be a
symmetric
half-integral semi-positive definite matrix of degree $m$.
The canonical automorphic factor
$$ J_{\rho,\mathcal M}: G^J\times \BH_{n,m}\lrt GL(V_\rho)$$
for $G^J$ on $\BH_{n,m}$ is given as follows\,:
\begin{eqnarray*}
 J_{\rho,\mathcal M}((g,(\lambda,\mu;\kappa)),(\Omega,Z))
&=&e^{2\,\pi\, i\,\sigma\left( {\mathcal M}(Z+\lambda\,
\Om+\mu)(C\Om+D)^{-1}C\,{}^t(Z+\lambda\,\Om\,+\,\mu)\right) }\\
& &\times\, e^{-2\,\pi\, i\,\sigma\left( {\mathcal M}(\lambda\,
\Om\,{}^t\!\lambda\,+\,2\,\lambda\,{}^t\!Z+\,\kappa+
\mu\,{}^t\!\lambda) \right)} \rho(C\,\Om+D),
\end{eqnarray*}
where $g=\left(\begin{matrix} A&B\\ C&D\end{matrix}\right)\in
Sp(n,\mathbb R),\ (\lambda,\mu;\kappa)\in H_{\mathbb R}^{(n,m)}$
and $(\Om,Z)\in \BH_{n,m}.$ We refer to \cite{YJH4} for a geometrical construction of $J_{\rho,\mathcal M}.$

\vskip 0.21cm
Let
$C^{\infty}(\BH_{n ,m},V_{\rho})$ be the algebra of all
$C^{\infty}$ functions on $\BH_{n,m}$
with values in $V_{\rho}.$
For $f\in C^{\infty}(\BH_{n,m}, V_{\rho}),$ we define
\begin{eqnarray}
 \left(f|_{\rho,{\mathcal M}}[(g,(\lambda,\mu;\kappa))]\right)(\Om,Z)
&= & J_{\rho,\mathcal M}((g,(\lambda,\mu;\kappa)),(\Omega,Z))^{-1}\\
& &\ f\left(g\!\cdot\!\Om,(Z+\lambda\,
\Om+\mu)(C\,\Om+D)^{-1}\right),\nonumber
\end{eqnarray}
where $g=\left(\begin{matrix} A&B\\ C&D\end{matrix}\right)\in
Sp(n,\mathbb R),\ (\lambda,\mu;\kappa)\in H_{\mathbb R}^{(n,m)}$
and $(\Om,Z)\in \BH_{n,m}.$

\begin{definition}
Let $\rho$ and $\mathcal M$
be as above. Let
$$H_{\mathbb Z}^{(n,m)}:= \left\{ (\lambda,\mu;\kappa)\in H_{\mathbb R}^{(n,m)}\, \vert
\,\ \lambda,\mu, \kappa\ {\rm integral}\ \right\}$$
be the discrete subgroup of $H_{\mathbb R}^{(n,m)}$.
A $\textsf{Jacobi\ form}$ of index $\mathcal M$
with
respect to $\rho$ on a subgroup $\Gamma$ of $\Gamma_n$ of finite
index is a
holomorphic function $f\in
C^{\infty}(\BH_{n,m},V_{\rho})$ satisfying the following
conditions (A) and (B):

\smallskip

\noindent (A) \,\ $f|_{\rho,{\mathcal M}}[\tilde{\gamma}] = f$ for
all $\tilde{\gamma}\in {\widetilde\Gamma}:= \Gamma \ltimes
H_{\mathbb Z}^{(n,m)}$.

\smallskip

\noindent (B) \,\ For each $M\in\Gamma_n$, $f|_{\rho,\CM}[M]$ has a
Fourier expansion of \\
\indent \ \ \ \ the following form :
$$(f|_{\rho,\CM}[M])(\Om,Z) = \sum\limits_{T=\,{}^tT\ge0\atop \text {half-integral}}
\sum\limits_{R\in \mathbb Z^{(n,m)}} c(T,R)\cdot e^{{ {2\pi
i}\over {\lambda_\G}}\,\sigma(T\Om)}\cdot e^{2\pi i\,\sigma(RZ)}$$

with $\lambda_\G(\neq 0)\in\BZ$ and
$c(T,R)\ne 0$ only if $\left(\begin{matrix} { 1\over {\lambda_\G}}T & \frac 12R\\
\frac 12\,^t\!R&{\mathcal M}\end{matrix}\right) \geq 0$.
\end{definition}

\medskip

\indent If $n\geq 2,$ the condition (B) is superfluous by K{\"
o}cher principle\,(\,cf.\,\cite{Zi} Lemma 1.6). We denote by
$J_{\rho,\mathcal M}(\Gamma)$ the vector space of all Jacobi forms
of index $\mathcal{M}$ with respect to $\rho$ on $\Gamma$.
Ziegler\,(\,cf.\,\cite{Zi} Theorem 1.8 or \cite{EZ} Theorem 1.1\,)
proves that the vector space $J_{\rho,\mathcal {M}}(\Gamma)$ is
finite dimensional. In the special case $\rho(A)=(\det(A))^k$ with
$A\in GL(n,\BC)$ and a fixed $k\in\BZ$, we write $J_{k,\CM}(\G)$
instead of $J_{\rho,\CM}(\G)$ and call $k$ the {\it weight} of the
corresponding Jacobi forms. For more results about Jacobi forms with
$n>1$ and $m>1$, we refer to \cite{YJH0}-\cite{YJH6} and \cite{Zi}. Jacobi forms play an
important role in lifting elliptic cusp forms to Siegel cusp forms of degree $2n$ (cf.\,\cite{Ik, Ik1}).

\vskip 0.21cm\noindent

Now we will make  brief historical remarks on
Jacobi forms. In 1985, the names Jacobi group and Jacobi forms got
kind of standard by the classic book \cite{EZ} by Eichler and
Zagier to remind of Jacobi's ``Fundamenta nova theoriae functionum
ellipticorum'', which appeared in 1829 (cf.\,\cite{J}). Before
\cite{EZ} these objects appeared more or less explicitly and under
different names in the work of many authors. In 1966
Pyatetski-Shapiro \cite{PS} discussed the Fourier-Jacobi expansion
of Siegel modular forms and the field of modular abelian
functions. He gave the dimension of this field in the higher
degree. About the same time Satake \cite{Sa1}-\cite{Sa2}
introduced the notion of ``groups of Harish-Chandra type'' which
are non reductive but still behave well enough so that he could
determine their canonical automorphic factors and kernel
functions. Shimura \cite{Sh1}-\cite{Sh2} gave a new foundation of
the theory of complex multiplication of abelian functions using
Jacobi theta functions. Kuznetsov \cite{Kuz} constructed functions
which are almost Jacobi forms from ordinary elliptic modular
functions. Starting 1981, Berndt \cite{Be1}-\cite{Be3} published
some papers which studied the field of arithmetic Jacobi
functions, ending up with a proof of Shimura reciprocity law for
the field of these functions with arbitrary level. Furthermore he
investigated the discrete series for the Jacobi group $G^J$
and developed the spectral theory for $L^2(\G_{n,m}\backslash
G^J)$ in the case $n=m=1$\,(cf.\,\cite{Be4}-\cite{Be6}). The
connection of Jacobi forms to modular forms was given by Maass,
Andrianov, Kohnen, Shimura, Eichler and Zagier. This connection is
pictured as follows. For $k$ even, we have the following
isomorphisms
\begin{equation}
M_k^*(\G_2)\,\cong\,J_{k,1}(\G_1)\,\cong\,M_{k-\frac12}^+ \left(\G_0^{(1)}(4)\right)\,
\cong\,M_{2k-2}(\G_1).
\end{equation}
Here $M_k^*(\G_2)$ denotes Maass's
Spezialschar or Maass space and $M_{k-\frac12}^+\left(\G_0^{(1)}(4)\right)$
denotes the Kohnen plus space. For a precise detail, we
refer to \cite{M3}-\cite{M5},\,\cite{A1},\,\cite{EZ},\,\cite{Ko1, Ko2} and \cite{Z}. In 1982 Tai \cite{Ta} gave asymptotic dimension
formulae for certain spaces of Jacobi forms for arbitrary $n$ and
$m=1$ and used these ones to show that the moduli ${\mathcal A}_n$
of principally polarized abelian varieties of dimension $n$ is
{\it of\ general\ type} for $n\geq 9.$ Feingold and Frenkel
\cite{FF} essentially discussed Jacobi forms in the context of
Kac-Moody Lie algebras generalizing the Maass correspondence to
higher level. Gritsenko \cite{Gri} studied Fourier-Jacobi
expansions and a non-commutative Hecke ring in connection with the
Jacobi group. After 1985 the theory of Jacobi forms for $n=m=1$
had been studied more or less systematically by the Zagier school.
A large part of the theory of Jacobi forms of higher degree was
investigated by Kramer \cite{Kr1, Kr2}, Runge \cite{Ru}, Yang
\cite{YJH0}-\cite{YJH4}and Ziegler \cite{Zi}. There were several
attempts to establish $L$-functions in the context of the Jacobi
group by Murase \cite{Mu1, Mu2} and Sugano \cite{MS} using
the so-called ``Whittaker-Shintani functions''. Kramer
\cite{Kr1, Kr2} developed an arithmetic theory of Jacobi
forms of higher degree. Runge \cite{Ru} discussed some part of the
geometry of Jacobi forms for arbitrary $n$ and $m=1.$ For a good
survey on some motivation and background for the study of Jacobi
forms, we refer to \cite{Be7}. The theory of Jacobi forms has been
extensively studied by many people until now and has many
applications in other areas like geometry and physics.

\end{section}

\vskip 1cm

\begin{section}{{\bf Singular Jacobi Forms}}
\setcounter{equation}{0}

\begin{definition}
A Jacobi form $f\in J_{\rho,\mathcal {M}}(\Gamma_n)$ is said to be
$\textsf{cuspidal}$
if $\begin{pmatrix} T & {\frac 12}R\\ {\frac 12}\,^t\!R & \mathcal
{M}\end{pmatrix} > 0$ for any $T,\,R$ with $c(T,R)\ne 0.$ A Jacobi
form $f\in J_{\rho,\mathcal{M}}(\Gamma_n)$ is said to be
$\textsf{singular}$
if it admits a Fourier expansion such that a Fourier
coefficient
$c(T,R)$ vanishes unless $\text{det}\begin{pmatrix} T &{\frac 12}R\\ {\frac 12}\,^t\!R & \mathcal
{M}\end{pmatrix}=0.$
\end{definition}

Let $\mathscr P_{n,m}=\mathscr P_n \times \BR^{(m,n)}$ be the Minkowski-Euclid space,
where $\mathscr P_n$ is the open cone consisting of positive symmetric
$n\times n$ real matrices.
For a variable $(Y,V)\in \mathscr P_{n,m}$ with $Y\in \mathscr P_n$ and $V\in \BR^{(m,n)}$,
we put
$$Y=(y_{\mu\nu})\ \text{with}\ y_{\mu\nu}=y_{\nu\mu},\ \
V=(v_{kl}),$$
$$\Yd\,=\left( { {1+\delta_{\mu\nu}}\over 2 } {
{\partial}\over {\partial y_{\mu\nu} } }\right),\ \ \ \Vd\,=\left({
{\partial}\over {\partial v_{kl}} } \right),$$
where $1\leq
\mu,\nu,\,l\leq n$ and $1\leq k\leq m.$

\vskip 0.25cm
We define the following differential operator
\begin{equation}
M_{n,m,\CM}:=\det (Y)\cdot \det \left( \Yd + {1\over {8\,\pi}}\, {}^{{}^{{}^{{}^\text{\scriptsize $t$}}}}\!\!\!\left(\Vd\right) \CM^{-1}\Vd\right).
\end{equation}
In \cite{YJH3}, Yang characterized singular Jacobi forms in the following way:
\begin{theorem}
Let $f\in J_{\rho,\mathcal M}(\G_n)$ be a Jacobi form of index $\mathcal M$ with respect to a rational representation $\rho$ of $GL(n,\BC)$. Then the
following conditions are equivalent:
\vskip 1mm
(Sing-1) $f$ is a singular Jacobi form.
\vskip 1mm
(Sing-2) $f$ satisfies the differential equation $M_{n,m,\mathcal M}f=0.$
\end{theorem}
\vskip 2mm\noindent
{\it Proof.} See Theorem 4.1 in \cite{YJH3}. \hfill $\Box$

\vskip2mm
\begin{theorem}
Let $2\mathcal M$ be a symmetric, positive definite, unimodular even matrix of degree $m$. Assume that $\rho$ is irreducible and satisfies the condition
$$ \rho(A)=\rho(-A)\quad {\rm for\ all}\ A\in GL(n,\BC).$$
Then a nonvanishing Jacobi form in $J_{\rho,\mathcal M}(\G_n)$ is singular if and only if $2\,k(\rho)< n+m$.
\end{theorem}
\vskip 2mm\noindent
{\it Proof.} See Theorem 4.5 in \cite{YJH3}. \hfill $\Box$

\vskip 0.32cm
\begin{remark}
We let
\begin{equation*}
GL_{n,m}:=\,GL(n,\BR)\ltimes \BR^{(m,n)}
\end{equation*}
be the semidirect product of $GL(n,\BR)$ and $\BR^{(m,n)}$ with multiplication law
\begin{equation*}
(A,a)\cdot (B,b):=\,(AB,a\,{}^tB^{-1}+b),\qquad A,B\in GL(n,\BR),\ \ a,b\in \BR^{(m,n)}.
\end{equation*}
Then we have the {\it natural action} of $GL_{n,m}$ on the Minkowski-Euclid space ${\mathscr P}_{n,m}$
defined by
\begin{equation}
(A,a)\cdot (Y,\zeta):=\,\big(AY\,{}^t\!A,\,(\zeta+a)\,{}^t\!A \big),
\end{equation}
where $(A,a)\in GL_{n,m},\ Y\in {\mathscr P}_n,\ \zeta\in \BR^{(m,n)}.$
Without difficulty we see that the differential operator $M_{n,m,\CM}$ is invariant under the action (8.2) of $GL_{n,m}.$
We refer to \cite{YJH15} for more detail about invariant differential operators on the Minkowski-Euclid space ${\mathscr P}_{n,m}$.
\end{remark}

\end{section}

\vskip 10mm

\begin{section}{{\bf The Siegel-Jacobi Operator}}
\setcounter{equation}{0}
Let $\rho$ be a rational representation of $GL(n,\BC)$ on a finite dimensional vector space $V_\rho$. For a positive integer $r<n$,
we let $\rho^{(r)}:GL(r,\BC)\lrt GL(V_\rho)$ be a rational representation of $GL(r,\BC)$ defined by
$$\rho^{(r)}(a)v:=\rho \left( \begin{pmatrix} a & 0 \\ 0 & iI_{n-r}\end{pmatrix} \right)v,\quad a\in GL(r,\BC),\ v\in V_\rho.$$
The \textsf{Siegel-Jacobi operator} $\Psi_{n,r}:J_{\rho,\CM}(\Gamma_n)\lrt J_{\rho^{(r)},\CM}(\Gamma_n)$ is defined by
$$\left( \Psi_{n,r}f \right)(\Om,Z):=\lim_{t\lrt \infty} f\left( \begin{pmatrix} \Om & 0 \\ 0 & i\,tI_{n-r}\end{pmatrix},(Z,0)\right),$$
where $f\in J_{\rho,\CM}(\Gamma_n),\ \Om\in \BH_r$ and $Z\in \BC^{(m,r)}.$

\vskip 2mm
In \cite{YJH0}, Yang investigated the injectivity, surjectivity and bijectivity of the Siegel-Jacobi operator.
\begin{theorem}
Let $2\mathcal M$ be a symmetric, positive definite, unimodular even matrix of degree $m$. Assume that $\rho$ is irreducible and satisfies the condition
$$ \rho(A)=\rho(-A)\quad {\rm for\ all}\ A\in GL(n,\BC).$$
If $2\,k(\rho)< n+{\rm rank} (\mathcal M)$, then the Siegel-Jacobi operator $\Psi_{n,n-1}$ is injective.
\end{theorem}
\vskip 2mm\noindent
{\it Proof.} See Theorem 3.5 in \cite{YJH0}. \hfill $\Box$

\begin{theorem}
Let $2\mathcal M$ be a symmetric, positive definite, unimodular even matrix of degree $m$. Assume that $\rho$ is irreducible and satisfies the condition
$$ \rho(A)=\rho(-A)\quad {\rm for\ all}\ A\in GL(n,\BC).$$
If $2\,k(\rho)+1 < n+{\rm rank} (\mathcal M)$, then the Siegel-Jacobi operator $\Psi_{n,n-1}$ is an isomorphism.
\end{theorem}
\vskip 2mm\noindent
{\it Proof.} See Theorem 3.6 in \cite{YJH0}. \hfill $\Box$

\begin{theorem}
Let $2\mathcal M$ be a symmetric, positive definite, unimodular even matrix of degree $m$.
Assume that $2\,k(\rho)> 4\,n+{\rm rank}(\mathcal M)$ and $k\equiv 0\ ({\rm mod}\ 2).$
Then the Siegel-Jacobi operator $\Psi_{n,n-1}$ is an isomorphism.
\end{theorem}
\vskip 2mm\noindent
{\it Proof.} See Theorem 3.7 in \cite{YJH0}. \hfill $\Box$

\vskip 3.5mm
Now we review the action of the Hecke operators on Jacobi forms. For a positive integer $\ell$, we define
$$O_n(\ell):=\big\{ M\in\BZ^{(2n,2n)}\ |\ ^tMJ_nM=\ell J_n\,\big).$$
Then $O_n(\ell)$ is decomposed into finitely many double cosets {\it mod} $\G_n$, that is,
$$O_n(\ell)=\bigcup_{j=1}^s \G_n g_j \G_n \qquad ({\rm disjoint \ union}).$$
We define
$$T(\ell):=\sum_{j=1}^s \G_n g_j \G_n \in \mathscr H^{(n)},\qquad {\rm the\ Hecke\ algebra}.$$
Let $M\in O_n(\ell).$ For a Jacobi form $f\in J_{\rho,\mathcal M}(\G_n)$, we define
\begin{equation}
f|_{\rho,\CM}(\G_n M\G_n):= \ell^{n k(\rho)-{{n(n+1)}\over 2}}\sum_{j=1}^s f|_{\rho,\CM}[(M_j,(0,0;0)))],
\end{equation}
where $\G_n M\G_n=\bigcup_{j=1}^s \G_n M_j$ (finite disjoint union) and $k(\rho)$ denotes the weight of $\rho$.
We see easily that if $M\in O_n(\ell)$ and $f\in J_{\rho,\mathcal M}(\G_n)$, then
$$f|_{\rho,\CM}(\G_n M\G_n)\in J_{\rho,\ell\mathcal M}(\G_n).$$
For a prime $p$, we define
$$O_{n,p}:=\bigcup_{l=0}^\infty O_n(p^l).$$
Let $\check{\mathscr L}_{n,p}$ be the $\BC$-module generated by all left cosets $\G_nM,\ M\in O_{n,p}$ and $\check{\mathscr H}_{n,p}$ the $\BC$-module generated by
all double cosets $\G_n M\G_n,\ M\in O_{n,p}.$ Then $\check{\mathscr H}_{n,p}$ is a commutative associative algebra. We associate to a double coset
$$ \G_n M\G_n=\bigcup_{i=1}^s \G_n M_i,\qquad M,M_i\in O_{n,p}\quad ({\rm disjoint\ union})$$
the element
$$ j(\G_n M\G_n)=\sum_{i=1}^s \G_n M_i \in \check{\mathscr L}_{n,p}.$$
We extend $j$ linearly to the Hecke algebra $\check{\mathscr H}_{n,p}$ and then we have a monomorphism $j:\check{\mathscr H}_{n,p}\lrt \check{\mathscr L}_{n,p}.$
We now define a bilinear mapping
$$ \check{\mathscr H}_{n,p} \times \check{\mathscr L}_{n,p} \lrt \check{\mathscr L}_{n,p}$$
by
$$ (\G_n M\G_n)\cdot (\G_n M_0)= \sum_{i=1}^s \G_n M_iM_0,\qquad {\rm where}\ \G_n M\G_n=\bigcup_{i=1}^s \G_n M_i.$$
This mapping is well defined because the definition does not depend on the choice of representatives.

\vskip 1.2mm
Let $f\in J_{\rho,\mathcal M}(\G_n)$ be a Jacobi form. For a left coset $L:=\G_n N$ with $N\in O_{n,p}$, we put
\begin{equation}
f|L:= f|_{\rho,\CM} [(N,(0,0;0))].
\end{equation}
We extend this operator (9.2) linearly to $\check{\mathscr L}_{n,p}.$ If $T\in \check{\mathscr H}_{n,p},$ we write
$$ f|T:=f|j(T).$$
Obviously we have
$$  (f|T)L= f|(TL),\qquad f\in J_{\rho,\mathcal M}(\G_n).$$
In a left coset $\G_nM,\ M\in O_{n,p},$ we can choose a representative $M$ of the form
\begin{equation*}
M=\begin{pmatrix} A & B \\ 0 & D \end{pmatrix},\quad ^t\!AD=p^{k_0}I_n,\ ^tBD=\,^tDB,
\end{equation*}
\begin{equation*}
A=\begin{pmatrix} a & {}^t\alpha  \\ 0 & A^* \end{pmatrix},\qquad B=\begin{pmatrix} b & {}^t\beta_1  \\ \beta_2 & B^* \end{pmatrix},\qquad
\D=\begin{pmatrix} d & 0  \\ \delta & D^* \end{pmatrix},
\end{equation*}
where $\alpha, \beta_1,\beta_2,\delta\in \BZ^{n-1}.$ Then we have
\begin{equation*}
M^*=\begin{pmatrix} A^* & B^* \\ 0 & D^* \end{pmatrix}\in O_{n-1,p}.
\end{equation*}
For an integer $r\in \BZ$, we define
\begin{equation*}
(\G_n M)^*:= {1\over {d^r}}\G_{n-1}M^*.
\end{equation*}
If $\G_n M\G_n=\bigcup_{j=1}^s \G_n M_j$ (disjoint union), $M,M_j\in O_{n,p},$ then we define in a natural way
\begin{equation}
(\G_n M\G_n)^*:= {1\over {d^r}}\,\sum_{j=1}^s \G_{n-1}M_j^*.
\end{equation}
We extend the above map (9.3) linearly on $\check{\mathscr H}_{n,p}$ and then we have an algebra homomorphism
\begin{equation}
\check{\mathscr H}_{n,p}\lrt \check{\mathscr H}_{n-1,p},\qquad T \longmapsto T^*.
\end{equation}
It is known that the above map (9.4) is a surjective map (\cite{Zh} Theorem 2).

\vskip 2mm
Let $\Psi_{n,r}^0:J_{\rho,\mathcal M}(\G_n) \lrt J_{\rho_0^{(r)},\mathcal M}(\G_r)$ be the {\it modified Siegel-Jacobi operator} defined by
\begin{equation*}
\left( \Psi_{n,r}^0 f \right)(\Om,Z):=\lim_{t\lrt \infty} f\left( \begin{pmatrix} it I_{n-r} & 0 \\ 0 & \Omega \end{pmatrix},(0,Z)\right),\quad (\Om,Z)\in\BH_{r,m},
\end{equation*}
where $\rho_0^{(r)}: GL(r,\BC) \lrt GL(V_\rho)$ is a finite dimensional representation of $GL(r,\BC)$ defined by
\begin{equation*}
\rho_0^{(r)}(A)=\begin{pmatrix} I_{n-r} & 0 \\ 0 & A \end{pmatrix},\quad A\in GL(r,\BC).
\end{equation*}

\vskip 2mm
In \cite{YJH0}, Yang proved that the action of the Hecke operators is compatible with that of the Siegel-Jacobi operator:
\begin{theorem}
Suppose we have
\vskip 2mm
(a) a rational finite dimensional representation
$$ \rho : GL(n,\BC) \lrt GL(V_\rho),$$
\vskip 1mm
(b) a rational finite dimensional representation
$$ \rho_0 : GL(n-1,\BC) \lrt GL(V_{\rho_0}),$$
\vskip 1mm
(c) a linear map $R:V_\rho \lrt V_{\rho_0},$
\vskip 2mm\noindent
satisfying the following properties (1) and (2):
\vskip 2mm
(1) $R\circ \rho \begin{pmatrix} 1 & 0 \\ 0 & A \end{pmatrix} =\rho_0 (A) \circ R$ \quad for\ all\ $A\in GL(n-1,\BC),$
\vskip 2mm
(2) $R\circ \rho \begin{pmatrix} a & 0 \\ 0 & I_{n-1} \end{pmatrix} =a^k R$ \ \ for\ some\ $k\in \BZ.$
\vskip 2mm\noindent
Then for any $f\in J_{\rho,\mathcal M}(\G_n)$ and $T\in \check{\mathscr H}_{n,p}$, we have
\begin{equation*}
\big( R\circ \Psi_{n,n-1}^0\big)(f|T)=R ( \Psi_{n,n-1}^0 f)|T^*.
\end{equation*}
\end{theorem}
\vskip 2mm\noindent
{\it Proof.} See Theorem 4.2 in \cite{YJH0}. \hfill $\Box$

\begin{remark}
Freitag \cite{Fr1} introduced the concept of stable modular forms using the Siegel operator and developed the theory of stable modular forms.
We can define the concept of stable Jacobi forms using the Siegel-Jacobi operator and develop the theory of stable Jacobi forms.
\end{remark}

\end{section}

\vskip 5mm

\begin{section}{{\bf Construction of Vector-Valued Modular Forms from Jacobi Forms}}
\setcounter{equation}{0}
Let $n$ and $m$ be two positive integers and let $\mathcal P_{m,n}:=\BC[z_{11},\cdots,z_{mn}]$ be the ring of complex valued polynomials on
$\BC^{(m,n)}.$ For any homogeneous polynomial $P\in \mathcal P_{m,n}$, we put
$$ P(\partial_Z):=P\left( { {\partial\ \ }\over {\partial z_{11}} },\cdots, { {\partial\ \ }\over {\partial z_{11}} } \right).$$
Let $S$ be a positive definite symmetric rational matrix of degree $m$. Let $T:=(t_{pq})$ be the inverse of $S$. For each $i,j$ with $1\leq i,j\leq n,$
we denote by $\Delta_{i,j}$ the following differential operator
\begin{equation*}
\Delta_{i,j}:=\sum_{p,q=1}^m t_{pq}\,{{\partial^2\ \ \ }\over {\partial z_{pi}\partial z_{qj}} },\qquad 1\leq i,j\leq n.
\end{equation*}
A polynomial $P$ on $\BC^{(m,n)}$ is said to be {\it harmonic} with respect to $S$ if
\begin{equation*}
\sum_{i=1}^n \Delta_{i,i}P=0.
\end{equation*}
A polynomial $P$ on $\BC^{(m,n)}$ is said to be {\it pluriharmonic} with respect to $S$ if
\begin{equation*}
\Delta_{i,j}P=0,\qquad 1\leq i,j\leq n.
\end{equation*}
If there is no confusion, we just write harmonic or pluriharmonic instead of  harmonic or pluriharmonic with respect to $S$. Obviously a pluriharmonic polynomial is harmonic.
We denote by $\mathscr H_{m,n}$ the space of all pluriharmonic polynomials on $\BC^{(m,n)}$. The ring $\mathcal P_{m,n}$ has a symmetric nondegenerate bilinear form
$\langle P,Q \rangle:=\big( P(\partial_Z)Q\big)(0)$ for $P,Q\in \mathcal P_{m,n}.$ It is easy to check that $\langle\ ,\ \rangle$ satisfies
\begin{equation*}
\langle P,QR \rangle = \langle Q(\partial_Z)P, R \rangle,\qquad P,Q,R\in \mathcal P_{m,n}.
\end{equation*}

\begin{lemma}
$\mathscr H_{m,n}$ is invariant under the action of $GL(n,\BC)\times O(S)$ given by
\begin{equation}
\big( (A,B),P(Z) \big) \longmapsto P(\,{}^tBZA),\qquad A\in GL(n,\BC),\ B\in O(S), \ P\in \mathscr H_{m,n}.
\end{equation}
Here $O(S):=\big\{ B\in GL(m,\BC)\ |\ {}^tBSB=S\,\big\}$ denotes the orthogonal group of the quadratic form $S$.
\end{lemma}
\vskip 2mm\noindent
{\it Proof.} See Corollary 9.11 in \cite{Mf3}. \hfill $\square$

\begin{remark}
In \cite{KV}, Kashiwara and Vergne investigated an irreducible decomposition of the space of complex pluriharmonic polynomials defined on $\BC^{(m,n)}$ under the action (10.1).
They showed that each irreducible component $\tau\otimes \lambda$ occurring in the decomposition of $\mathscr H_{m,n}$ under the action (10.1) has multiplicity one and the
irreducible representation $\tau$ of $GL(n,\BC)$ is determined uniquely by the irreducible representation of $O(S).$
\end{remark}

\vskip 3mm
Throughout this section we fix a rational representation $\rho$ of $GL(n,\BC)$ on a finite dimensional complex vector space $V_\rho$ and a positive definite symmetric,
half integral matrix $\mathcal M$ of degree $m$ once and for all.

\begin{definition}
A holomorphic function $f:\BH_n\lrt V_\rho$ is called a {\it modular form of type} $\rho$ on $\G_n$ if
$$  f(M\cdot\Omega)=f\big( (A\Om+B)(C\Om+D)^{-1}\big)=\rho (C\Omega+D) f(\Omega),\quad \Om\in\BH_n$$
for all $M=\begin{pmatrix} A & B \\ C & D \end{pmatrix}\in \Gamma_n.$ If $n=1,$ the additional cuspidal condition will be added. We denote by $[\G_n,\rho]$
the vector space of all modular forms of type $\rho$ on $\G_n$.
\end{definition}

\vskip 2mm
Let $\mathscr H_{m,n;\CM}$ be the vector space of of all pluriharmonic polynomials on $\BC^{(m,n)}$ with respect to $S:=(2\,\!\CM)^{-1}.$
According to Lemma 10.1, there exists an irreducible subspace $V_\tau(\neq 0)$ invariant under the action of $GL(n,\BC)$ given by (10.1).
We denote this representation by $\tau$. Then we have
\begin{equation*}
\big( \tau(A)P \big)(Z)=P(ZA),\qquad A\in GL(n,\BC),\ P\in V_\tau,\ Z\in \BC^{(m,n)}.
\end{equation*}
The action $\widehat\tau$ of $GL(n,\BC)$ on $V_\tau^*$ is defined by
\begin{equation*}
\big(  \widehat\tau(A)^{-1}\zeta\big)(P):= \zeta \big( \tau (\,{}^tA^{-1})P\big),
\end{equation*}
where $A\in GL(n,\BC),\ \zeta\in V_\tau^*$ and $P\in V_\tau.$

\begin{definition}
Let $f\in J_{\rho,\mathcal M}(\G_n)$ be a Jacobi form of index $\CM$ with respect to $\rho$ on $\G_n$. Let $P\in V_\tau$ be a homogeneous pluriharmonic polynomial.
We put
\begin{equation*}
f_P (\Omega):=P(\partial_Z) f(\Om,Z)|_{Z=0},\qquad \Om\in \BH_n, \ Z\in \BC^{(m,n)}.
\end{equation*}
Now we define the mapping
$$  f_\tau: \BH_n \lrt V_\tau^* \otimes V_\rho $$
by
\begin{equation}
\big( f_\tau (\Omega) \big)(P):=  f_P (\Omega),\qquad \Om\in \BH_n, \ P\in V_\tau.
\end{equation}
\end{definition}

Yang proved the following theorem in \cite{YJH4}.
\begin{theorem}
Let $\tau$ and $\widehat\tau$ be as before. Let $f\in J_{\rho,\mathcal M}(\G_n)$ be a Jacobi form of index $\CM$ with respect to $\rho$ on $\G_n$.
Then $f_\tau (\Omega)$ is a modular form of type $\widehat\tau \otimes \rho$, i.e., $f_\tau \in [\G_n, \widehat\tau \otimes \rho].$
\end{theorem}
\vskip 2mm\noindent
{\it Proof.} See  Main Theorem in \cite{YJH4}. \hfill $\square$

\vskip 3mm
We obtain an interesting and important identity by applying Theorem 10.1 to the Eisenstein series.
Let $\CM$ be a half integral positive symmetric matrix of degree $m$. We set
\begin{equation*}
\G_{n;[0]}:=\left\{ \begin{pmatrix} A & B \\ C & D \end{pmatrix}\in \G_n \ \Big|\ C=0\,\right\}.
\end{equation*}
Let $\mathscr R$ be a complete system of representatives of the cosets $\G_{n;[0]}\ba \G_n$ and $\Lambda$ be a complete system of representatives of the cosets
$\BZ^{(m,n)}/\big( {\rm Ker}(\CM) \cap \BZ^{(m,n)}\big)$, where ${\rm Ker}(\CM):=\left\{ \lambda\in \BR^{(m,n)}\,|\ \CM\cdot \lambda=0\,\right\}.$
Let $k\in\BZ^+$ be a positive integer. In \cite{Zi}, Ziegler defined the Eisenstein series $E_{k,\CM}^{(n)}(\Omega,Z)$ of Siegel type by
\begin{eqnarray*}
E_{k,\CM}^{(n)}(\Omega,Z):=&& \sum_{{\tiny\begin{pmatrix} A & B \\ C & D \end{pmatrix}}\in \mathscr R} \det (C\Omega+D)^{-k}\cdot e^{2\pi i\,\sigma(\CM Z(C\Om+D)^{-1}C\,^tZ)}\\
&& \cdot \sum_{\lambda\in \Lambda} e^{2\pi i\, \sigma \big(\CM ((A\Om+B)(C\Om+D)^{-1}\,{}^t\lambda + 2\,\lambda\,{}^t(C\Om+D)^{-1}\,{}^tZ)\big)},
\end{eqnarray*}
where $(\Om,Z)\in \BH_{n,m}.$ Now we assume that $k>n+m+1$ and $k$ is even. Then according to \cite{Zi}, Theorem 2.1, $E_{k,\CM}^{(n)}(\Omega,Z)$ is a nonvanishing Jacobi form
in $J_{k,\CM}(\G_n).$ By Theorem 10.1, $\big( E_{k,\CM}^{(n)} \big)_\tau$ is a ${\rm Hom}(V_\tau,\BC)$-valued modular form of type $\widehat\tau \otimes \det^k.$
We define the automorphic factor $j: Sp(n,\BR)\times \BH_n \lrt GL(n,\BC)$ by
\begin{equation*}
j(g,\Om):=C\Om +D,\qquad g=\begin{pmatrix} A & B \\ C & D \end{pmatrix}\in Sp(n,\BR),\ \Om\in \BH_n.
\end{equation*}
Then according to the relation occurring in the process of the proof of Theorem 10.1, for any homogeneous pluriharmonic polynomial $P$ with respect to $(2\,\CM)^{-1},$
we obtain the following identity
\begin{eqnarray}
&&\det j(M,\Om)^k  \sum_{\gamma\in {\mathscr R}} \sum_{\lambda\in \Lambda} \det j(\gamma,\Om)^{-k}\cdot P\big( 4\pi i\,\CM \lambda\,{}^t\! j(\gamma,\Om)^{-1} \big)\cdot
e^{2\pi i\,\sigma \big( \CM (\gamma\cdot \Om)\,{}^t\lambda \big)}\\
&& \quad = \sum_{\gamma\in {\mathscr R}} \sum_{\lambda\in \Lambda} \det j(\gamma,M\!\cdot\!\Om)^{-k}\cdot P\big( 4\pi i\,\CM \lambda\,{}^t\! j(\gamma M,\Om)^{-1} \big)
\cdot e^{2\pi i\,\sigma \big( \CM ((\gamma M)\cdot \Om)\,{}^t\lambda \big)} \nonumber
\end{eqnarray}
for all $M\in \G_n$ and $\Om\in \BH_n.$

\vskip 2mm
For any homogeneous pluriharmonic polynomial $P$ with respect to $(2\,\CM)^{-1},$ we define the function $G_P: \Gamma_n \times \BH_n \lrt \BC$ by
\begin{equation}
G_P(M,\Om):= \sum_{\gamma\in {\mathscr R}} \sum_{\lambda\in \Lambda} \det j(\gamma M, \Om)^{-k}\, P\big( 4\pi i\,\CM \lambda\,{}^t\! j(\gamma M,\Om)^{-1} \big)
\, e^{2\pi i\,\sigma \big( \CM ((\gamma M)\cdot \Om)\,{}^t\lambda \big)},
\end{equation}
where $M\in \G_n$ and $\Om\in \BH_n.$ Then according to Formula (10.3), we obtain the following relation
\begin{equation}
G_P(M,\Om)=G_P(I_{2n},\Om)\qquad \textrm{for all} \ M\in \G_n\ \textrm{and}\ \Om\in \BH_n.
\end{equation}

If $P=c$ is a constant, we see from (10.3) and (10.5) that $G_c:=G_P$ satisfies the following relation
\begin{equation}
G_c(M,N\!\cdot\!\Om)=G_c(I_{2n},N\!\cdot\! \Om)=\det j(N,\Om)^k\, G_c(M,\Om)
\end{equation}
for all $M,N\in \G_n$ and $\Om\in \BH_n$. Therefore for any $M\in \G_n$, the function $G_c (M,\cdot):\BH_n\lrt \BC$ is a Siegel modular form of weight $k$.

\end{section}

\vskip 10mm

\begin{section}{{\bf Maass-Jacobi Forms}}
\setcounter{equation}{0}

\newcommand\ddx{{{\partial^2}\over{\partial x^2}}}
\newcommand\ddy{{{\partial^2}\over{\partial y^2}}}
\newcommand\ddu{{{\partial^2}\over{\partial u^2}}}
\newcommand\ddv{{{\partial^2}\over{\partial v^2}}}
\newcommand\px{{{\partial}\over{\partial x}}}
\newcommand\py{{{\partial}\over{\partial y}}}
\newcommand\pu{{{\partial}\over{\partial u}}}
\newcommand\pv{{{\partial}\over{\partial v}}}
\newcommand\pxu{{{\partial^2}\over{\partial x\partial u}}}
\newcommand\pyv{{{\partial^2}\over{\partial y\partial v}}}

\vskip 0.3cm Using $G^J$-invariant differential operators on the Siegel-Jacobi
space, we introduce a notion of Maass-Jacobi forms.

\vskip 0.2cm
\begin{definition} Let
$$\Gamma_{n,m}:=Sp(n,{\mathbb Z})\ltimes H_{\mathbb Z}^{(n,m)}$$
be the discrete subgroup of $G^J$, where
$$H_{\BZ}^{(n,m)}=\left\{ (\lambda,\mu;\kappa)\in
H_{\BR}^{(n,m)}\,|\ \lambda,\mu,\kappa \ \textrm{are integral}\
\right\}.$$ A smooth function $f:\Hnm\lrt \BC$ is called a
$\textsf{Maass}$-$\textsf{Jacobi form}$ on $\Hnm$ if $f$ satisfies
the following conditions (MJ1)-(MJ3)\,:\vskip 0.1cm (MJ1)\ \ \ $f$
is invariant under $\G_{n,m}.$\par (MJ2)\ \ \ $f$ is an
eigenfunction of the Laplacian $\Delta_{n,m;A,B}$ (cf. Formula (2.4)).\par (MJ3)\ \ \ $f$
has a polynomial growth, that is, there exist a constant $C>0$
\par \ \ \ \ \ \ \ \ \ \ \ and a positive integer $N$ such that
\begin{equation*}
|f(X+i\,Y,Z)|\leq C\,|p(Y)|^N\quad \textrm{as}\ \det
Y\longrightarrow \infty,
\end{equation*}

\ \ \ \ \ \ \ \ \ \ \ where $p(Y)$ is a polynomial in
$Y=(y_{ij}).$
\end{definition}

\begin{remark}
We also may define the notion of Maass-Jacobi forms as follows.
Let $\mathbb{D}_*$ be a commutative subalgebra of $\mathbb{D}(\Hnm)$ containing the Laplacian
$\Delta_{n,m;A,B}$.
We say that a smooth function $f:\Hnm\lrt \BC$ is a Maass-Jacobi form with respect to $\mathbb{D}_*$
if $f$ satisfies the conditions $(MJ1),\ (MJ2)_*$ and $(MJ3)$\,: the condition $(MJ2)_*$ is given by
\vskip 0.3cm\noindent $(MJ2)_*$\  $f$ is an eigenfunction of any invariant differential
operator in $\BD_*$.
\end{remark}

\begin{remark}
Erik Balslev \cite{B} developed the spectral theory of $\Delta_{1,1;1,1}$ on $\Bbb H_{1,1}$ to
prove that the set of all eigenvalues of $\Delta_{1,1;1,1}$ satisfies the Weyl law.
\end{remark}

\vskip 0.3cm It is natural to propose the following problems.

\vskip 0.3cm\noindent {\bf {Problem\ A}\,:} Find all the
eigenfunctions of $\Delta_{n,m;A,B}.$

\vskip 0.3cm\noindent {\bf {Problem\ B}\,:} Construct Maass-Jacobi
forms.

\vskip 0.5cm If we find a {\it nice} eigenfunction $\phi$ of the Laplacian $\Delta_{n,m;A,B}$, we can construct
a Maass-Jacobi form $f_\phi$ on $\Hnm$ in the usual way defined by
\begin{equation*}
f_\phi(\Omega,Z):=\,\sum_{\gamma\in \Gamma_{n,m}^\infty\backslash \Gamma_{n,m}} \phi\big( \gamma\cdot (\Omega,Z)\big),
\end{equation*}
where
\begin{equation*}
\Gamma_{n,m}^\infty=\left\{ \left( \begin{pmatrix} A&B\\
C&D\end{pmatrix},(\lambda,\mu;\kappa)\right)\in \Gamma_{n,m}\,\Big|\ C=0\,\right\}
\end{equation*}
\noindent is a subgroup of $\Gamma_{n,m}.$

\vskip 0.3cm  We consider the simple case when $n=m=1$ and $A=B=1$. A metric
$ds_{1,1;1,1}^2$ on $\BH_{1,1}$ given by
\begin{align*} ds^2_{1,1;1,1}\,=\,&{{y\,+\,v^2}\over
{y^3}}\,(\,dx^2\,+\,dy^2\,)\,+\, {\frac 1y}\,(\,du^2\,+\,dv^2\,)\\
&\ \ -\,{{2v}\over {y^2}}\, (\,dx\,du\,+\,dy\,dv\,)\end{align*} is
a $G^J$-invariant K{\"a}hler metric on $\BH_{1,1}$.
Its Laplacian $\Delta_{1,1;1,1}$ is given by
\begin{align*} \Delta_{1,1;1,1}\,=\,& \,y^2\,\left(\,\ddx\,+\,\ddy\,\right)\,\\
&+\, (\,y\,+\,v^2\,)\,\left(\,\ddu\,+\,\ddv\,\right)\\ &\ \
+\,2\,y\,v\,\left(\,\pxu\,+\,\pyv\,\right).\end{align*}

\vskip 0.2cm We provide some examples of eigenfunctions of
$\Delta_{1,1;1,1}$. \vskip 0.2cm $(a)\ h(x,y)=y^{1\over
2}K_{s-{\frac12}}(2\pi |a|y)\,e^{2\pi iax} \ (s\in \BC,$
$a\not=0\,)$ with eigenvalue \par
\ \ \ \ $s(s-1).$ Here
$$K_s(z):={\frac12}\int^{\infty}_0 \exp\left\{-{z\over 2}(t+t^{-1})\right\}\,t^{s-1}\,dt,$$ \indent \ \ \ \ where
$\mathrm{Re}\,z
> 0.$ \par $(b)\ y^s,\ y^s x,\ y^s u\ (s\in\BC)$ with eigenvalue
$s(s-1).$
\par
$(c) \ y^s v,\ y^s uv,\ y^s xv$ with eigenvalue $s(s+1).$
\par
$(d)\ x,\,y,\,u,\,v,\,xv,\,uv$ with eigenvalue $0$.
\par
$(e)$\ All Maass wave forms.

\vskip 0.7cm
Let $\rho$ be a rational representation of $GL(n,\BC)$ on a finite dimensional complex vector space $V_\rho$.
Let $\mathcal M\in \BR^{(m,m)}$ be a symmetric half-integral semi-positive definite matrix of degree $m$. Let
$C^\infty(\Hnm,V_\rho)$ be the algebra of all $C^\infty$ functions on $\Hnm$ with values in $V_\rho$. We define the
$|_{\rho,\mathcal M}$-slash action of $G^J$ on $C^\infty(\Hnm,V_\rho)$ as follows: If $f\in C^\infty(\Hnm,V_\rho)$,
\begin{eqnarray*}
& & f|_{\rho,\mathcal M}[(M,(\lambda,\mu;\kappa))](\Om,Z) \nonumber\\
&:=&\,e^{-2\pi i \,\sigma (\mathcal M [Z+\lambda \Omega+\mu](C\Omega+D)^{-1}C)}\cdot
e^{2\pi i \,\sigma (\mathcal M (\lambda \Om \,^t\!\lambda\,+\,2\lambda\,^t\!Z\,+\,\kappa\,+\,\mu\,^t\!\lambda))}\\
& & \ \times\, \rho(C\Om+D)^{-1} f(M\cdot\Om,(Z+\lambda\Om+\mu)(C\Om+D)^{-1}),\nonumber
\end{eqnarray*}
where $M=\begin{pmatrix} A&B\\C&D\end{pmatrix}\in Sp(n,\BR)$ and $(\lambda,\mu;\kappa)\in H_\BR^{(n,m)}$. We recall the Siegel's notation
$\alpha[\beta]=\,^t\beta\alpha \beta$ for suitable matrices $\alpha$ and $\beta$. We define $\BD_{\rho,\mathcal M}$ to be
the algebra of all differential operators $D$ on $\Hnm$ satisfying the following condition
\begin{equation*}
(Df)|_{\rho,\mathcal M}[g]=\,D(f|_{\rho,\mathcal M}[g])
\end{equation*}
for all $f\in C^\infty(\Hnm,V_\rho)$ and for all $g\in G^J.$ We denote by ${\mathcal Z}_{\rho,\mathcal M}$ the center of
$\BD_{\rho,\mathcal M}$.

\vskip 0.3cm We define another notion of Maass-Jacobi forms as follows.
\vskip 0.31cm
\begin{definition} A vector-valued smooth function $\phi:\Hnm\lrt V_\rho$ is called a Maass-Jacobi form on $\Hnm$ of type $\rho$
and index $\mathcal M$ if it satisfies the following conditions $(MJ1)_{\rho,\mathcal M},\ (MJ2)_{\rho,\mathcal M}$ and
$(MJ3)_{\rho,\mathcal M}$\,:
\vskip 0.1cm $(MJ1)_{\rho,\mathcal M}$\ \ \ $\phi|_{\rho,\mathcal M}[\gamma]=\phi$\ \ for all $\gamma\in\G_{n,m}.$\par
$(MJ2)_{\rho,\mathcal M}$\ \ \ $f$ is an
eigenfunction of all differential operators in the center ${\mathcal Z}_{\rho,\mathcal M}$ of $\BD_{\rho,\mathcal M}$.\par
$(MJ3)_{\rho,\mathcal M}$\ \ \ $f$
has a growth condition
$$\phi(\Om,Z)=O\Big( e^{a\det Y}\cdot e^{2\pi \textrm{tr}(\mathcal M [V]Y^{-1})} \Big)$$
\qquad \qquad \quad \ \quad as $\det Y\lrt \infty$ for some $a>0.$
\end{definition}

\begin{remark}
In the sense of Definition 11.2, Pitale \cite{Pit} studied Maass-Jacobi forms on the Siegel-Jacobi space $\BH_{1,1}.$
We refer to \cite{YJH12, YJH13} for more details on Maass-Jacobi forms.
\end{remark}

\end{section}

\vskip 10mm

\newcommand\tg{\widetilde\gamma}
\newcommand\wmo{{\mathscr W}_{\mathcal M,\Omega}}
\newcommand\hrnm{H_\BR^{(n,m)}}
\newcommand\rmn{\BR^{(m,n)}}

\begin{section}{{\bf The Schr{\"o}dinger-Weil Representation}}
\setcounter{equation}{0}

\vskip 0.2cm Throughout this section we assume that $\CM$ is a positive definite
symmetric real $m\times m$ matrix. We
consider the Schr{\"o}dinger representation ${\mathscr W}_\CM$ of
the Heisenberg group $\hrnm$ with the central character ${\mathscr
W}_\CM((0,0;\kappa))=\chi_\CM((0,0;\kappa))=e^{\pi
i\,\s(\CM\kappa)},\ \kappa\in S(m,\BR)$. Then ${\mathscr W}_\CM$ is expresses explicitly as follows:
\begin{equation}
\left[ {\mathscr W}_\CM (h_0)f\right](\lambda)=e^{\pi
i\sigma\{\CM(\kappa_0+\mu_0\,^t\!\lambda_0+
2\lambda\,^t\!\mu_0)\}}\,f(\lambda+\lambda_0),
\end{equation}
\noindent where $h_0=(\lambda_0,\mu_0;\kappa_0)\in \hrnm$ and $\lambda\in\BR^{(m,n)}.$ For the construction of ${\mathscr W}_\CM$ we refer to \cite{YJH16}.
We note that the symplectic group $Sp(n,\BR)$ acts on $\hrnm$ by
conjugation inside $G^J$. For a fixed element $g\in Sp(n,\BR)$,
the irreducible unitary representation ${\mathscr W}_\CM^g$ of
$\hrnm$ defined by
\begin{equation}
{\mathscr W}_\CM^g(h)={\mathscr W}_\CM(ghg^{-1}),\quad h\in\hrnm
\end{equation}
has the property that
\begin{equation*}
{\mathscr W}_\CM^g((0,0;\kappa))={\mathscr W}_\CM((0,0;\kappa))=e^{\pi
i\,\s(\CM \kappa)}\,\textrm{Id}_{H(\chi_\CM)},\quad \kappa\in
S(m,\BR).
\end{equation*}
Here $\textrm{Id}_{H(\chi_\CM)}$ denotes the identity operator on
the Hilbert space $H(\chi_\CM).$ According to Stone-von Neumann
theorem, there exists a unitary operator $R_\CM(g)$ on
$H(\chi_\CM)$  with $R_\CM (I_{2n})=\textrm{Id}_{H(\chi_\CM)}$ such that
\begin{equation}
R_\CM(g){\mathscr W}_\CM(h)={\mathscr
W}_\CM^g(h) R_\CM(g)\qquad {\rm for\ all}\ h\in\hrnm.
\end{equation}
We observe that
$R_\CM(g)$ is determined uniquely up to a scalar of modulus one.

\vskip 0.35cm
From now on, for brevity, we put $G=Sp(n,\BR).$ According to
Schur's lemma, we have a map $c_\CM:G\times G\lrt T$ satisfying
the relation
\begin{equation}
R_\CM(g_1g_2)=c_\CM(g_1,g_2)R_\CM(g_1)R_\CM(g_2)\quad \textrm{for
all }\ g_1,g_2\in G.
\end{equation}
We recall that $T$ denotes the multiplicative group of complex numbers of modulus one.
Therefore $R_\CM$ is a projective representation of $G$ on
$H(\chi_\CM)$ and $c_\CM$ defines the cocycle class in $H^2(G,T).$
The cocycle $c_\CM$ yields the central extension $G_\CM$ of $G$ by
$T$. The group $G_\CM$ is a set $G\times T$ equipped with the
following multiplication

\begin{equation}
(g_1,t_1)\cdot (g_2,t_2)=\big(g_1g_2,t_1t_2\,
c_\CM(g_1,g_2)^{-1}\,\big),\quad g_1,g_2\in G,\ t_1,t_2\in T.
\end{equation}
We see immediately that the map ${\widetilde R}_\CM:G_\CM\lrt
GL(H(\chi_\CM))$ defined by

\begin{equation}
{\widetilde R}_\CM(g,t)=t\,R_\CM(g) \quad \textrm{for all}\
(g,t)\in G_\CM
\end{equation}
is a {\it true} representation of $G_\CM.$ As in Section 1.7 in
\cite{LV}, we can define the map $s_\CM:G\lrt T$ satisfying the
relation
\begin{equation*}
c_\CM(g_1,g_2)^2=s_\CM(g_1)^{-1}s_\CM(g_2)^{-1}s_\CM(g_1g_2)\quad
\textrm{for all}\ g_1,g_2\in G.
\end{equation*}
Thus we see that
\begin{equation}
G_{2,\CM}=\left\{\, (g,t)\in G_\CM\,|\ t^2=s_\CM(g)^{-1}\,\right\}
\end{equation}

\noindent is the metaplectic group associated with $\CM$ that is a
two-fold covering group of $G$. The restriction $R_{2,\CM}$ of
${\widetilde R}_\CM$ to $G_{2,\CM}$ is the $\textsf{Weil representation}$ of
$G$ associated with $\CM$.
\par
If we identify $h=(\lambda,\mu;\kappa)\in \hrnm$ (resp. $g\in Sp(n, \BR)$) with
$(I_{2n},(\lambda,\mu;\kappa))\in G^J$ (resp. $(g,(0,0;0))\in G^J),$
every element $\tilde g$ of $G^J$ can be written as $\tilde g =hg$ with $h\in \hrnm$
and $g\in Sp(n, \BR)$. In fact,
\begin{equation*}
(g,(\la,\mu;\kappa))=(I_{2n},((\la,\mu)g^{-1};\kappa))\,(g,(0,0;0))=((\la,\mu)g^{-1};\kappa)\cdot g.
\end{equation*}
Therefore we define the {\it projective} representation $\pi_\CM$ of the Jacobi group
$G^J$ with cocycle $c_\CM (g_1,g_2)$ by
\begin{equation}
\pi_\CM(hg)={\mathscr W}_\CM(h)\,R_\CM(g),\quad h\in\hrnm,\ g\in
G.
\end{equation}
\noindent
Indeed, since $\hrnm$ is a normal subgroup of $G^J$, for any $h_1,h_2\in \hrnm$ and $g_1,g_2\in G$,
\begin{eqnarray*}
\pi_\CM (h_1g_1h_2g_2)&=& \pi_\CM (h_1g_1h_2g_1^{-1}g_1g_2)\\
&=& {\mathscr W}_\CM \big(h_1(g_1h_2g_1^{-1})\big) R_\CM (g_1g_2)\\
&=& c_\CM (g_1,g_2) {\mathscr W}_\CM (h_1) {\mathscr W}_\CM^{g_1}(h_2) R_\CM (g_1)R_\CM (g_2)\\
&=& c_\CM (g_1,g_2) {\mathscr W}_\CM (h_1)  R_\CM (g_1) {\mathscr W}_\CM (h_2) R_\CM (g_2)\\
&=& c_\CM (g_1,g_2) \pi_\CM (h_1g_1) \pi_\CM (h_2g_2).
\end{eqnarray*}
\par
We let
$$G_{\CM}^J\!=G_{\CM}\ltimes \hrnm$$
be the semidirect product of $G_{\CM}$ and $\hrnm$ with the multiplication law
\begin{eqnarray*}
&&\big( (g_1,t_1),(\la_1,\mu_1;\kappa_1)\big)\cdot \big( (g_2,t_2),(\la_2,\mu_2\,;\kappa_2)\big)\\
&=&\big( (g_1,t_1)(g_2,t_2),(\tilde\la+\la_2,\tilde\mu+\mu_2\,;\kappa_1+\kappa_2+ \tilde\la\,^t\!\mu_2-\tilde\mu\,^t\!\la_2)\big),
\end{eqnarray*}
where $(g_1,t_1), (g_2,t_2)\in G_{\CM},\ (\la_1,\mu_1;\kappa_1), (\la_2,\mu_2\,;\kappa_2)\in\hrnm$ and
$(\tilde\la,\tilde\mu)=(\la,\mu)g_2.$
If we identify $h=(\lambda,\mu\,;\kappa)\in \hrnm$ (resp. $(g,t)\in G_{\CM})$ with
$((I_{2n},1),(\lambda,\mu\,;\kappa))\in G^J_{\CM}$ (resp. $((g,t),(0,0;0))\in G^J_{\CM}),$
we see easily that every element $\big( (g,t),(\la,\mu\,;\kappa)\big)$ of $G_{\CM}^J$ can be expressed as
$$\big( (g,t),(\la,\mu\,;\kappa)\big)=\big( (I_{2n},1),((\la,\mu)g^{-1};\kappa)\big) \big( (g,t),(0,0;0)\big)=((\la,\mu)g^{-1};\kappa)(g,t).  $$
Now we can define the {\it true} representation $\widetilde\om_\CM$ of $G_{\CM}^J$ by
\begin{equation}
\widetilde\omega_\CM(h\!\cdot\!(g,t))=t\,\pi_\CM(hg)=t\, {\mathscr
W}_\CM(h)\,R_\CM(g),\quad h\in\hrnm,\ (g,t)\in G_{\CM}.
\end{equation}
Indeed, since $\hrnm$ is a normal subgroup of $G_{\CM}^J$,
\begin{eqnarray*}
&& \widetilde\omega_\CM \big( h_1 (g_1,t_1) h_2 (g_2,t_2)\big)\\
&=& \widetilde\omega_\CM \big( h_1 (g_1,t_1) h_2 (g_1,t_1)^{-1}(g_1,t_1) (g_2,t_2)\big)\\
&=& \widetilde\omega_\CM \big( h_1 (g_1,t_1) h_2 (g_1,t_1)^{-1}\big(g_1g_2,t_1t_2\, c_\CM(g_1,g_2)^{-1}\big) \big)\\
&=& t_1t_2\, c_\CM(g_1,g_2)^{-1}\, {\mathscr W}_\CM \big(   h_1 (g_1,t_1) h_2 (g_1,t_1)^{-1} \big) R_\CM (g_1g_2)\\
&=& t_1t_2\, {\mathscr W}_\CM (h_1)  {\mathscr W}_\CM  \big(  (g_1,t_1) h_2 (g_1,t_1)^{-1} \big)\,R_\CM(g_1)R_\CM(g_2)\\
&=& t_1t_2\, {\mathscr W}_\CM (h_1)  {\mathscr W}_\CM  \big(  g_1 h_2 g_1^{-1} \big)\,R_\CM(g_1)R_\CM(g_2)\\
&=& t_1t_2\, {\mathscr W}_\CM (h_1)  R_\CM(g_1)\, {\mathscr W}_\CM  (h_2) R_\CM(g_2)\\
&=& \left\{ t_1 \,\pi_\CM (h_1g_1)\right\} \left\{ t_2 \,\pi_\CM (h_2g_2)\right\}\\
&=& \widetilde\omega_\CM \big( h_1 (g_1,t_1) \big)\, \tilde\omega_\CM \big( h_2 (g_2,t_2) \big).
\end{eqnarray*}
Here we used the fact that $(g_1,t_1) h_2 (g_1,t_1)^{-1}=g_1 h_2 g_1^{-1}.$

\vskip0.2cm
We recall that the following matrices
\begin{eqnarray*}
t(b)&=&\begin{pmatrix} I_n& b\\
                   0& I_n\end{pmatrix}\ \textrm{with any}\
                   b=\,{}^tb\in \BR^{(n,n)},\\
g(\alpha)&=&\begin{pmatrix} {}^t\alpha & 0\\
                   0& \alpha^{-1}  \end{pmatrix}\ \textrm{with
                   any}\ \alpha\in GL(n,\BR),\\
\s_n&=&\begin{pmatrix} 0& -I_n\\
                   I_n&\ 0\end{pmatrix}
\end{eqnarray*}
\noindent generate the symplectic group $G=Sp(n,\BR)$
(cf.\,\cite[p.\,326]{Fr2},\,\cite[p.\,210]{Mf2}). Therefore the
following elements $h_t(\lambda,\mu\,;\kappa),\
t(b\,;t),\,g(\alpha\,;t)$ and $\s_{n\,;t}$ of $G_\CM\ltimes \hrnm$
defined by
\begin{eqnarray*}
&& h_t(\la,\mu\,;\kappa)=\big( (I_{2n},t),(\la,\mu;\kappa)\big)\
\textrm{with}\ t\in T,\ \la,\mu\in
\BR^{(m,n)}\ \textrm{and}\ \kappa\in\BR^{(m,m)} ,\\
&&t(b\,;t)=\big( (t(b),t),(0,0;0) \big)\ \textrm{with any}\
                   b=\,{}^tb\in \BR^{(n,n)},\ t\in T,\\
&& g(\alpha\,;t)=\left(
\big(g(\alpha),t),(0,0;0)\right)\
\textrm{with any}\ \alpha\in GL(n,\BR)\ \textrm{and}\ t\in T,\\
 &&\s_{n\,;\,t}=\left( (\s_n,t),(0,0;0)\right)\
 \textrm{with}\ t\in T
\end{eqnarray*}
generate the group $G_\CM\ltimes\hrnm.$ We can show that the
representation ${\widetilde \om}_\CM$ is realized on the
representation $H(\chi_\CM)=L^2\big(\rmn\big)$ as follows: for
each $f\in L^2\big(\rmn\big)$ and $x\in \rmn,$ the actions of
${\widetilde \om}_\CM$ on the generators are given by

\begin{eqnarray}
%\begin{equation}
\left[ {\widetilde \om}_\CM
\big(h_t(\lambda,\mu\,;\kappa)\big)f\right](x)&=&\,t\,e^{\pi
i\,\s\{\CM(\kappa+\mu\,{}^t\!\lambda+2\,x\,{}^t\mu)\}}\,f(x+\lambda),\\
\left[ {\widetilde \om}_\CM\big(t(b\,;t)\big)f\right](x)&=& t\,e^{\pi i\,\s(\CM\, x\,b\,{}^tx)}f(x),\\
\left[ {\widetilde \om}_\CM\big(g(\alpha\,;t)\big)f\right](x)&=& t\,| \det \alpha|^{\frac m2}\,f(x\,{}^t\alpha),
\end{eqnarray}
\begin{equation}
\left[ {\widetilde \om}_\CM\big(\s_{n\,;\,t}\big)f\right](x)=t\,
\big( \det \CM\big)^{\frac n2}\,\int_{\rmn}f(y)\,e^{-2\,\pi i\,\s(\CM\,y\,{}^tx)}\,dy.
\end{equation}
\par
Let
$$G_{2,\CM}^J\!=G_{2,\CM}\ltimes \hrnm$$
be the semidirect product of $G_{2,\CM}$ and $\hrnm$. Then $G_{2,\CM}^J$ is a subgroup of $G_\CM^J$ which is a two-fold covering group of the Jacobi group $G^J.$
The restriction $\om_\CM$ of $\widetilde \om_\CM$ to $G_{2,\CM}^J$ is called the $\textsf{Schr{\"o}dinger-Weil}$
$\textsf{representation}$ of $G^J$ associated with $\CM$.

\vskip 0.251cm
We denote by $L^2_+\big(\rmn\big)$\,$\big(
\textrm{resp.}\,\,L^2_-\big(\rmn\big)\big)$ the subspace of
$L^2\big(\rmn\big)$ consisting of even (resp.\,odd) functions in
$L^2\big(\rmn\big)$. According to Formulas (12.11)--(12.13),
$R_{2,\CM}$ is decomposed into representations of $R_{2,\CM}^\pm$

\begin{equation*}
R_{2,\CM}=R_{2,\CM}^+\oplus R_{2,\CM}^-,
\end{equation*}
where $R_{2,\CM}^+$ and $R_{2,\CM}^-$ are the even Weil
representation and the odd Weil representation of $G$ that are
realized on $L^2_+\big(\rmn\big)$ and $L^2_-\big(\rmn\big)$
respectively. Obviously the center ${\mathscr Z}^J_{2,\CM}$ of
$G_{2,\CM}^J$ is given by
\begin{equation*}
{\mathscr Z}_{2,\CM}^J=\big\{ \big( (I_{2n},1),(0,0;\kappa)\big)\in
G_{2,\CM}^J\,\big\} \cong S(m,\BR).
\end{equation*}
We note that the restriction of $\omega_\CM$ to $G_{2,\CM}$
coincides with $R_{2,\CM}$ and $\omega_\CM(h)={\mathscr W}_\CM(h)$
for all $h\in \hrnm.$

\vskip 0.2cm\noindent
\begin{remark}
In the case $n=m=1,\
\omega_\CM$ is dealt in \cite{BS} and \cite{Ma}. We refer to
\cite{G} and \cite{KV} for more details about the Weil
representation $R_{2,\CM}$.
\end{remark}

\begin{remark}
The Schr{\"o}dinger-Weil representation is applied to the theory of Maass-Jacobi forms \cite{Pit}.
\end{remark}

\vskip 0.35cm
Let $\mathcal M$ be a positive definite symmetric real matrix of degree $m$. We recall the Schr{\"o}dinger representation ${\mathscr W}_\CM$ of the Heisenberg
group $\hrnm$ associate with $\CM$ given by Formula (12.1). We note that for an element $(\la,\mu;\kappa)$ of $\hrnm$, we have the
decomposition
\begin{equation*}
(\la,\mu;\kappa)=(\la,0;0)\circ (0,\mu;0)\circ (0,0;\kappa\!-\!\la\, {}^t\!\mu).
\end{equation*}

We consider the embedding $\Phi_n :SL(2,\BR)\lrt Sp(n,\BR)$ defined by
\begin{equation}
\Phi_n \left( \begin{pmatrix} a & b\\ c & d \end{pmatrix}\right):=
\begin{pmatrix} aI_n & bI_n\\ cI_n & d I_n\end{pmatrix},\qquad \begin{pmatrix} a & b\\ c & d \end{pmatrix}\in SL(2,\BR).
\end{equation}

\vskip 0.35cm
For $x,y\in \BR^{(m,n)},$ we put
$$ (x,y)_\CM:=\s (\,{}^tx \CM y)\qquad \textrm{and}\qquad \| x\|_\CM :=\sqrt{(x,x)_\CM}.$$
According to Formulas (12.11)-(12.13), for any $ M=\begin{pmatrix} a & b\\ c & d \end{pmatrix}\in SL(2,\BR)\hookrightarrow Sp(n,\BR)$ and
$f\in L^2\left(\BR^{(m,n)}\right)$, we have the following explicit representation
\begin{equation}
[R_\CM (M)f](x)= \begin{cases} |a|^{\frac{mn}2} e^{ab \|x\|_\CM^2 \pi i} f(ax)  & \text{if $c=0$,}\\
(\det \CM)^{\frac n2}\, |c|^{-{\frac{mn}2}} \int_{\BR^{(m,n)}} e^{{\frac{\alpha(M,x,y,\CM)}c} \pi i} f(y) dy
& \text{if $c\neq 0$,}\end{cases}
\end{equation}
where
$$ \alpha(M,x,y,\CM)= a\,\|x\|_\CM^2 + d\, \|y\|_\CM^2 - 2 (x,y)_\CM.$$
Indeed, if $a=0$ and $c\neq 0$, using the decomposition
\begin{equation*}
M=\begin{pmatrix} 0 & -c^{-1}\\ c & d \end{pmatrix}=
\begin{pmatrix} 0 & -1\\ 1 & \ 0 \end{pmatrix} \begin{pmatrix} c & d\\ 0 & c^{-1} \end{pmatrix}
\end{equation*}
and
if $a\neq 0$ and $c\neq 0$, using the decomposition
\begin{equation*}
M=\begin{pmatrix} a & b\\ c & d \end{pmatrix}=\begin{pmatrix} a & c^{-1} \\ 0 & a^{-1} \end{pmatrix}
\begin{pmatrix} 0 & -1\\ 1 & \ 0 \end{pmatrix} \begin{pmatrix} ac & ad\\ 0 & (ac)^{-1} \end{pmatrix},
\end{equation*}
we obtain Formula (12.15).

\vskip 0.35cm
If
\begin{equation*}
M_1=\begin{pmatrix} a_1 & b_1\\ c_1 & d_1 \end{pmatrix},\quad M_2=\begin{pmatrix} a_2 & b_2\\ c_2 & d_2 \end{pmatrix}\quad
\textrm{and}\quad M_3=\begin{pmatrix} a_3 & b_3\\ c_3 & d_3 \end{pmatrix}\in SL(2,\BR)
\end{equation*}
with $M_3=M_1M_2$, the corresponding cocycle is given by
\begin{equation}
c_\CM (M_1,M_2)=e^{-i\, \pi\, mn\,\textrm{sign}(c_1c_2c_3)/4},
\end{equation}
where
\begin{equation*}
\textrm{sign}(x)= \begin{cases}  -1 \qquad &(x<0)\\
\ \ 0 \qquad &(x=0)\\
\ \ 1 \qquad &(x>0). \end{cases}
\end{equation*}
In the special case when
\begin{equation*}
M_1=\begin{pmatrix} \cos \phi_1 & -\sin \phi_1 \\ \sin \phi_1 & \ \ \cos \phi_1 \end{pmatrix}\quad \textrm{and}\quad
M_2=\begin{pmatrix} \cos \phi_2 & -\sin \phi_2 \\ \sin \phi_2 & \ \ \cos \phi_2 \end{pmatrix},
\end{equation*}
we find
\begin{equation*}
c_\CM (M_1,M_2)=e^{-i\, \pi\, mn\,(\sigma_{\phi_1}+\sigma_{\phi_2}-\sigma_{\phi_1+\phi_2})/4},
\end{equation*}
where
\begin{equation*}
\sigma_\phi= \begin{cases}  2\nu \qquad & \text{if $\phi=\nu\pi$}\\
2\nu+1 \qquad & \text{if $\nu\pi <\phi< (\nu+1)\pi.$} \end{cases}
\end{equation*}
It is well known that every $M\in SL(2,\BR)$ admits the unique Iwasawa decomposition
\begin{equation}
M=\begin{pmatrix} 1 & u\\ 0 & 1 \end{pmatrix} \begin{pmatrix} v^{1/2} & 0 \\ 0 & v^{-1/2} \end{pmatrix}
\begin{pmatrix} \cos \phi & -\sin \phi \\\sin \phi & \ \ \cos\phi \end{pmatrix},
\end{equation}
where $\tau=u+iv \in \BH_1$ and $\phi\in [0,2\pi).$ This parametrization $M=(\tau,\phi)$ in $SL(2,\BR)$ leads to the natural action of
$SL(2,\BR)$ on $\BH_1\times [0,2\pi)$ defined by
\begin{equation}
\begin{pmatrix} a & b\\ c & d \end{pmatrix}(\tau,\phi):=\left( \frac{a\tau+b}{c\tau+d},\, \phi + \textrm{arg} (c\tau+d)\ \textrm{mod}\ 2\pi \right).
\end{equation}

\begin{lemma}
For two elements $g_1$ and $g_2$ in $SL(2,\BR)$, we let
\begin{equation*}
g_1=\begin{pmatrix} 1 & u_1\\ 0 & 1 \end{pmatrix} \begin{pmatrix} v_1^{1/2} & 0 \\ 0 & v_1^{-1/2} \end{pmatrix}
\begin{pmatrix} \cos \phi_1 & -\sin \phi_1 \\\sin \phi_1 & \ \ \cos\phi_1 \end{pmatrix}
\end{equation*}
and
\begin{equation*}
g_2=\begin{pmatrix} 1 & u_2\\ 0 & 1 \end{pmatrix} \begin{pmatrix} v_2^{1/2} & 0 \\ 0 & v_2^{-1/2} \end{pmatrix}
\begin{pmatrix} \cos \phi_2 & -\sin \phi_2 \\\sin \phi_2 & \ \ \cos\phi_2 \end{pmatrix}
\end{equation*}
be the Iwasawa decompositions of $g_1$ and $g_2$ respectively, where $u_1,u_2\in\BR,\ v_1>0,
\,v_2>0$ and $0\leq \phi_1,\phi_2 < 2\pi.$ Let
\begin{equation*}
g_3=g_1g_2=\begin{pmatrix} 1 & u_3\\ 0 & 1 \end{pmatrix} \begin{pmatrix} v_3^{1/2} & 0 \\ 0 & v_3^{-1/2} \end{pmatrix}
\begin{pmatrix} \cos \phi_3 & -\sin \phi_3 \\\sin \phi_3 & \ \ \cos\phi_3 \end{pmatrix}
\end{equation*}
be the Iwasawa decomposition of $g_3=g_1g_2.$ Then we have
\begin{eqnarray*}
u_3&=& \frac{A}{(u_2\sin \phi_1+\cos \phi_1)^2+(v_2\sin\phi_1)^2},\\
v_3&=& \frac{v_1v_2}{(u_2\sin \phi_1+\cos \phi_1)^2+(v_2\sin\phi_1)^2}
\end{eqnarray*}
and
\begin{equation*}
\phi_3=tan^{-1} \left[
{ {(v_2\cos\phi_2+u_2 \sin\phi_2)\tan\phi_1 +\sin\phi_2}\over {(-v_2\sin\phi_2+u_2 \cos\phi_2)\tan\phi_1 +\cos\phi_2} }\right],
\end{equation*}
where
\begin{eqnarray*}
A&=&u_1 (u_2\sin \phi_1+\cos \phi_1)^2+ (u_1v_2-v_1u_2)\sin^2 \phi_1 \\
 && \ +\,v_1u_2\cos^2\phi_1+ v_1(u_2^2+v_2^2-1)\sin\phi_1\cos\phi_1.
\end{eqnarray*}
\end{lemma}

\vskip 0.26cm \noindent
{\it Proof.} If $g\in SL(2,\BR)$ has the unique Iwasawa decomposition (12.17), then we get the following
\begin{eqnarray*}
a&=& v^{1/2}\cos\phi +uv^{-1/2}\sin\phi,\\
b&=& -v^{1/2}\sin\phi +uv^{-1/2}\cos\phi,\\
c&=& v^{-1/2}\sin\phi, \quad   d=v^{-1/2}\cos\phi,\\
u&=&(ac+bd)\left(c^2+d^2\right)^{-1},\quad v=\left(c^2+d^2\right)^{-1},\quad \tan\phi={c\over d}\,    .
\end{eqnarray*}
We set
\begin{equation*}
g_3=g_1g_2=\begin{pmatrix} a_3 & b_3\\ c_3 & d_3 \end{pmatrix}.
\end{equation*}
Since
\begin{equation*}
u_3=(a_3c_3+b_3 d_3)\left(c_3^2+d_3^2\right)^{-1},\quad v=\left(c_3^2+d_3^2\right)^{-1},\quad \tan\phi_3={c_3\over d_3},
\end{equation*}
by an easy computation, we obtain the desired results.
\hfill $\square$

\vskip 0.53cm
Now we use the new coordinates $(\tau=u+iv,\phi)$ with $\tau\in\BH_1$ and $\phi\in [0,2\pi)$ in $SL(2,\BR).$ According to Formulas
(12.11)-(12.13), the projective representation $R_\CM$ of $SL(2,\BR)\hookrightarrow Sp(n,\BR)$ reads in these coordinates $(\tau=u+iv,\phi)$ as
follows:
\begin{equation}
\left[R_\CM (\tau,\phi)f\right](x)=v^{\frac{mn}4}\,e^{u \|x\|_\CM^2 \pi\, i} \left[R_\CM(i,\phi)f\right]\big(v^{1/2}x \big),
\end{equation}
where $f\in L^2\left( \BR^{(m,n)}\right),\ x\in \BR^{(m,n)}$ and
\begin{eqnarray}
 &\left[R_\CM(i,\phi)f\right](x) \hskip 9cm\nonumber\\
 =&\begin{cases}
 f(x) & \text{if $\phi\equiv 0$ mod $2\pi$,}\\
 f(-x) & \text{if $\phi\equiv \pi$ mod $2\pi$,}\\
 (\det\CM)^{\frac n2}\,|\sin\phi|^{-{{mn}\over 2}}\,\int_{\BR^{(m,n)}}e^{B(x,y,\phi,\CM)\pi i}\,f(y)dy
 & \text{if $\phi\not\equiv 0$ mod $\pi$}.
 \end{cases}
\end{eqnarray}
Here $$ B(x,y,\phi,\CM)= { {\left( \|x\|_\CM^2 +  \|y\|_\CM^2\right) \cos\phi - 2(x,y)_\CM} \over {\sin\phi}}. $$
Now we set
$$ S=\begin{pmatrix} 0 & -1\\ 1 & \ \ 0 \end{pmatrix}.$$
We note that
\begin{equation}
\left[ R_\CM \left( i, {\pi\over 2}\right)f\right](x)=\left[ R_\CM(S)f\right](x)=(\det\CM)^{\frac n2}\,\int_{\BR^{(m,n)}} f(y\,)\,e^{-2\, (x,\,y)_\CM\,\pi\,i}\,dy
\end{equation}
for $f\in L^2\left( \BR^{(m,n)}\right).$

\begin{remark}
For Schwartz functions $f\in \mathscr{S} \left(\BR^{(m,n)}\right),$ we have
\begin{equation*}
\lim_{\phi\lrt 0\pm} |\sin\phi|^{-{{mn}\over 2}}\, \int_{\BR^{(m,n)}}e^{B(x,y,\phi,\CM)\,\pi\, i}\,f(y)dy= e^{\pm i\,\pi\, mn/4}f(x)\neq f(x).
\end{equation*}
\noindent
Therefore the projective representation $R_\CM$ is not continuous at $\phi=\nu \pi\,(\nu\in\BZ)$ in general.
If we set
\begin{equation*}
\tilde{R}_\CM (\tau,\phi)= e^{-i\,\pi\, mn\sigma_\phi/4} R_\CM (\tau,\phi),
\end{equation*}
$\tilde{R}_\CM$ corresponds to a unitary representation of the double cover of $SL(2,\BR)$ (cf. (3.5) and \cite{LV}).
This means in particular that
\begin{equation*}
\tilde{R}_\CM (i,\phi)\tilde{R}_\CM (i,\phi')=\tilde{R}_\CM (i,\phi+\phi'),
\end{equation*}
where $\phi\in [0,4\pi)$ parametrises the double cover of $SO(2)\subset SL(2,\BR).$
\end{remark}

\vskip 0.53cm
We observe that for any element $(g,(\la,\mu;\kappa))\in G^J$ with $g\in Sp(n,\BR)$ and $(\la,\mu;\kappa)\in \hrnm$, we have the following decomposition
\begin{equation*}
(g,(\la,\mu;\kappa))=(I_{2n},((\la,\mu)g^{-1};\kappa))\,(g,(0,0;0))=((\la,\mu)g^{-1};\kappa)\cdot g.
\end{equation*}
Thus $Sp(n,\BR)$ acts on $\hrnm$ naturally by
\begin{equation*}
g\cdot (\la,\mu;\kappa)=\left( (\la,\mu)g^{-1};\kappa\right),\qquad g\in Sp(n,\BR), \ (\la,\mu;\kappa)\in \hrnm.
\end{equation*}

\begin{definition}
For any Schwartz function $f\in \mathscr{S} \left(\BR^{(m,n)}\right),$ we define the function $\Theta_f^{[\CM]}$ on the Jacobi group
$SL(2,\BR)\ltimes \hrnm\hookrightarrow G^J$ by
\begin{equation}
\Theta_f^{[\CM]}(\tau,\phi\,;\la,\mu,\kappa):=\sum_{\om\in\BZ^{(m,n)}} \left[ \pi_\CM \left( (\la,\mu;\kappa)(\tau,\phi)\right)f\right] (\omega),
\end{equation}
where $(\tau,\phi)\in SL(2,\BR)$ and $(\la,\mu\,;\kappa)\in \hrnm$. The projective representation $\pi_\CM$ of the Jacobi group $G^J$ was already
defined by Formula (12.8).
More precisely, for $\tau=u+iv\in\BH_1$ and $(\la,\mu;\kappa)\in \hrnm$,
we have
\begin{eqnarray*}
&&\Theta_f^{[\CM]}(\tau,\phi\,;\la,\mu,\kappa)= v^{\frac{mn}4}\,\,e^{2\,\pi\,i\,\s(\CM(\kappa+\mu {}^t\la))}\\
&& \quad\times \sum_{\om\in\BZ^{(m,n)}}\, e^{\pi\,i\,\left\{ u \|\om+\la\|_\CM^2\,+\,2 (\om,\,\mu)_\CM \right\}}\,
\left[  R_\CM (i,\phi)f\right] \left( v^{1/2}(\omega+\l)\right).
%\nonumber
\end{eqnarray*}
\end{definition}

\begin{lemma}
We set $f_\phi:=\tilde{R}_\CM (i,\phi)f$ for $f\in \mathscr{S} \left(\BR^{(m,n)}\right)$. Then for any $R>1$, there exists a constant $C_R$ such that
for all $x\in \rmn$ and $\phi\in\BR,$
$$ |f_\phi(x)| \leq C_R \,\left( 1+ \|x\|_\CM\right)^{-R}.$$
\end{lemma}

\vskip 0.25cm\noindent
{\it Proof.} Following the arguments in the proof of Lemma 4.3 in \cite{Ma}, pp.\,428-429, we get the desired result.
\hfill $\square$

\begin{theorem}[Jacobi 1]
Let $\CM$ be a positive definite symmetric integral matrix of degree $m$ such that $\CM \BZ^{(m,n)}=\BZ^{(m,n)}.$ Then
for any Schwartz function $f\in \mathscr{S} \left(\BR^{(m,n)}\right),$ we have
$$\Theta_f^{[\CM]}\left( -{1\over {\tau}}, \,\phi+\textrm{arg}\,\tau\,;-\mu,\la,\kappa \right)=\big(\det\CM\big)^{-{\frac n2}}\,  c_\CM(S,(\tau,\phi)) \,
\Theta_f^{[\CM]}(\tau,\phi\,;\la,\mu,\kappa),$$
where
$$c_\CM(S,(\tau,\phi)):=e^{i\,\pi mn \,\textrm{sign}(\sin\phi\,\sin (\phi+\arg \tau))}.$$
\end{theorem}
\vskip 1mm\noindent
{\it Proof.} See Theorem 6.1 in \cite{YJH16}.   \hfill $\Box$

\vskip 0.53cm
\begin{theorem}[Jacobi 2]
Let $\CM=(\CM_{kl})$ be a positive definite symmetric integral $m\times m$ matrix and
let $s=(s_{kj})\in \BZ^{(m,n)}$ be integral. Then we have
\begin{equation*}
\Theta_f^{[\CM]}(\tau+2,\phi\,;\la,s-2\,\la+\mu,\kappa-s\,^t\la)=\Theta_f^{[\CM]}(\tau,\phi\,;\la,\mu,\kappa)
\end{equation*}
for all $(\tau,\phi)\in SL(2,\BR)$ and $(\la,\mu;\kappa)\in \hrnm$.
\end{theorem}
\vskip 1mm\noindent
{\it Proof.} See Theorem 6.2 in \cite{YJH16}.   \hfill $\Box$

\vskip 0.53cm
\begin{theorem}[Jacobi 3]
Let $\CM=(\CM_{kl})$ be a positive definite symmetric integral $m\times m$ matrix and let $(\la_0,\mu_0;\kappa_0)\in H_\BZ^{(m,n)}$
be an integral element of $\hrnm.$ Then we have
\begin{eqnarray*}
&&\Theta_f^{[\CM]}(\tau,\phi\,;\la+\la_0,\mu+\mu_0,\kappa+\kappa_0+\la_0\,{}^t\mu-\mu_0\,^t\la)\\
&=& e^{\pi\,i\,\s(\CM (\kappa_0+\mu_0\,{}^t\la_0))}
\Theta_f^{[\CM]}(\tau,\phi\,;\la,\mu,\kappa)
\end{eqnarray*}
for all $(\tau,\phi)\in SL(2,\BR)$ and $(\la,\mu;\kappa)\in \hrnm$.
\end{theorem}
\vskip 1mm\noindent
{\it Proof.} See Theorem 6.3 in \cite{YJH16}.   \hfill $\Box$

\vskip 0.35cm
We put $V(m,n)=\BR^{(m,n)}\times \BR^{(m,n)}$. Let
\begin{equation*}
G^{(m,n)}:=SL(2,\BR) \ltimes V(m,n)
\end{equation*}
be the group with the following multiplication law
\begin{equation}
(g_1,(\la_1,\mu_1))\cdot (g_2,(\la_2,\mu_2))=(g_1g_2,(\la_1,\mu_1)g_2+ (\la_2,\mu_2)),
\end{equation}
where $g_1,g_2\in SL(2,\BR)$ and $\la_1,\la_2,\mu_1,\mu_2\in \BR^{(m,n)}$.

\vskip 0.25cm\noindent
We define
$$  \G^{(m,n)}:= SL(2,\BZ)\times H_\BZ^{(n,m)}.$$
Then $\G^{(m,n)}$ acts on $G^{(m,n)}$ naturally through the multiplication law (12.23).
\begin{lemma}
$\G^{(m,n)}$ is generated by the elements
$$ ( S, (0,0)),\quad (T_\flat,(0,s)) \quad \textrm{and}\quad (I_2,(\la_0,\mu_0)),$$
where
$$  S=\begin{pmatrix} 0 & -1 \\ 1 & \ \ 0 \end{pmatrix},\quad T_\flat=\begin{pmatrix} 1 & 1 \\ 0 & 1 \end{pmatrix}\quad \textrm{and}\  s,\la_0,\mu_0\in \BZ^{(m,n)}.$$
\end{lemma}
\vskip 0.251cm\noindent
{\it Proof.}
Since $SL(2,\BZ)$ is generated by $S$ and $T_\flat$, we get the desired result.
\hfill $\square$

\vskip 0.53cm
\noindent We define
\begin{eqnarray*}
&&\Theta_f^{[\CM]}(\tau,\phi;\la,\mu)\\
&=&  v^{\frac{mn}4}\,\sum_{\om\in\BZ^{(m,n)}}\, e^{\pi\,i\,\left\{ u \|\om+\la\|_\CM^2\,+\,2 (\om,\,\mu)_\CM \right\}}\,
\left[ R_\CM (i,\phi)f\right] \left( v^{1/2}(\omega+\l)\right).\nonumber
\end{eqnarray*}

\begin{theorem}
Let $\G^{(m,n)}_{[2]}$ be the subgroup of $\G^{(m,n)}$ generated by the elements
$$ ( S, (0,0)),\quad (T_*,(0,s)) \quad \textrm{and}\quad (I_2,(\la_0,\mu_0)),$$
where
$$  T_*=\begin{pmatrix} 1 & 2 \\ 0 & 1 \end{pmatrix}\quad \textrm{and}\  s,\la_0,\mu_0\in \BZ^{(m,n)}.$$
Let $\CM=(\CM_{kl})$ be a positive definite symmetric unimodular integral $m\times m$ matrix such that $\CM \BZ^{(m,n)}=\BZ^{(m,n)}.$
Then for $f,g\in \mathscr{S} \left( \BR^{(m,n)}\right),$ the function
\begin{equation*}
\Theta_f^{[\CM]}(\tau,\phi;\la,\mu)\,\overline{\Theta_g^{[\CM]}(\tau,\phi;\la,\mu)}
\end{equation*}
is invariant under the action of $\G^{(m,n)}_{[2]}$ on $G^{(m,n)}$.
\end{theorem}
\vskip 1mm\noindent
{\it Proof.} See Theorem 6.4 in \cite{YJH16}.   \hfill $\Box$

\end{section}

\vskip 1cm

\begin{section}{{\bf Final Remarks and Open Problems}}
\setcounter{equation}{0}

The Siegel-Jacobi space $\BH_{n,m}$ is a non-symmetric homogeneous space that is important geometrically and arithmetically. As we see in the formula (7.2), the theory of Jacobi forms
is applied in the study of modular forms. The theory of Jacobi forms reduces to that of Siegel modular forms if the index $\CM$ is zero. Unfortunately the theory of
the geometry and the arithmetic of the Siegel-Jacobi space has not been well developed so far.

\vskip 2mm Now we propose open problems related to the geometry and the arithmetic of the Siegel-Jacobi space.

\vskip 2mm\noindent
{\bf Problem 1.} Find the analogue of the Hirzebruch-Mumford Proportionality Theorem.
\vskip 1mm
Let us give some remarks for this problem.
 Before we describe the proportionality
theorem for the Siegel modular variety, first of all we review the
compact dual of the Siegel upper half plane $\BH_n$. We note that
$\BH_n$ is biholomorphic to the generalized unit disk $\BD_n$ of
degree $n$ through the Cayley transform. We suppose that
$\Lambda=(\BZ^{2n},\langle\ ,\ \rangle)$ is a symplectic lattice with a
symplectic form $\langle\ ,\ \rangle.$ We extend scalars of the lattice
$\Lambda$ to $\BC$. Let
\begin{equation*}
{\mathfrak Y}_n:=\left\{\,L\subset \BC^{2n}\,|\ \dim_\BC L=n,\ \
\langle x,y \rangle=0\quad \textrm{for all}\ x,y\in L\,\right\}
\end{equation*}
be the complex Lagrangian Grassmannian variety parameterizing
totally isotropic subspaces of complex dimension $n$. For the
present time being, for brevity, we put $G=Sp(n,\BR)$ and
$K=U(n).$ The complexification $G_\BC=Sp(n,\BC)$ of $G$ acts on
${\mathfrak Y}_n$ transitively. If $H$ is the isotropy subgroup of
$G_\BC$ fixing the first summand $\BC^n$, we can identify
${\mathfrak Y}_n$ with the compact homogeneous space $G_\BC/H.$ We
let
\begin{equation*}
{\mathfrak Y}_n^+:=\big\{\,L\in {\mathfrak Y}_n\,|\ -i \langle x,{\bar
x}\rangle
>0\quad \textrm{for all}\ x(\neq 0)\in L\,\big\}
\end{equation*}
be an open subset of ${\mathfrak Y}_n$. We see that $G$ acts on
${\mathfrak Y}_n^+$ transitively. It can be shown that ${\mathfrak
Y}_n^+$ is biholomorphic to $G/K\cong \BH_n.$ A basis of a lattice
$L\in {\mathfrak Y}_n^+$ is given by a unique $2n\times n$ matrix
${}^t(-I_n\,\,\Om)$ with $\Om\in\BH_n$. Therefore we can identify
$L$ with $\Om$ in $\BH_n$. In this way, we embed $\BH_n$ into
${\mathfrak Y}_n$ as an open subset of ${\mathfrak Y}_n$. The
complex projective variety ${\mathfrak Y}_n$ is called the $
\textit{compact dual}$ of $\BH_n.$

\newcommand\CA{\mathcal A}
\vskip 0.2cm
Let $\G$ be an arithmetic subgroup of $\G_n$. Let
$E_0$ be a $G$-equivariant holomorphic vector bundle over
$\BH_n=G/K$ of rank $r$. Then $E_0$ is defined by the
representation $\tau:K\lrt GL(r,\BC).$ That is, $E_0\cong
G\times_K \BC^r$ is a homogeneous vector bundle over $G/K$. We
naturally obtain a holomorphic vector bundle $E$ over
$\CA_{n,\G}:=\G\ba G/K.$ $E$ is often called an
$\textit{automorphic}$ or $ \textit{arithmetic}$ vector bundle
over $\CA_{n,\G}$. Since $K$ is compact, $E_0$ carries a
$G$-equivariant Hermitian metric $h_0$ which induces a Hermitian
metric $h$ on $E$. According to Main Theorem in \cite{Mf1}, $E$
admits a $ \textit{unique}$ extension ${\tilde E}$ to a smooth
toroidal compactification ${\tilde \CA}_{n,\G}$ of $\CA_{n,\G}$
such that $h$ is a singular Hermitian metric $ \textit{good}$ on
${\tilde \CA}_{n,\G}$. For the precise definition of a
$\textit{good metric}$ on $\CA_{n,\G}$ we refer to \cite[p.\,242]{Mf1}.
According to Hirzebruch-Mumford's Proportionality
Theorem\,(cf.\,\cite[p.\,262]{Mf1}), there is a natural metric on
$G/K=\BH_n$ such that the Chern numbers satisfy the following
relation
\begin{equation*}
c^{\al}\big({\tilde E}\big)=(-1)^{{\frac 12}n(n+1)}\,
\textmd{vol}\left( \G\ba \BH_n\right)\,c^{\al}\big( {\check E}_0\big)
\end{equation*}
for all $\al=(\al_1,\cdots,\al_r)$ with nonegative integers
$\al_i\,(1\leq i\leq r)$ and $\sum_{i=1}^r\al_i={\frac 12}n(n+1),$
where ${\check E}_0$ is the $G_{\BC}$-equivariant holomorphic
vector bundle on the compact dual ${\mathfrak Y}_n$ of $\BH_n$
defined by a certain representation of the stabilizer $
\textrm{Stab}_{G_\BC}(e)$ of a point $e$ in ${\mathfrak Y}_n$.
Here $\textmd{vol}\left( \G\ba \BH_n\right)$ is the volume of
$\G\ba\BH_n$ that can be computed\,(cf.\,\cite{Si1}).

\vskip 5mm
As before we consider the Siegel-Jacobi modular group $\G_{n.m}:=\G_n\ltimes H_\BZ^{(n,m)}$ with $\G_n=Sp(n,\BZ).$
For an arithmetic subgroup $\G$ of $\G_n$, we set
\begin{equation*}
\CA_{n,m,\G}:=\G_*\ba \BH_{n,m}\qquad {\rm with}\ \G_*=\G \ltimes H_\BZ^{(n,m)}.
\end{equation*}

\vskip 2mm\noindent
{\bf Problem 2.} Compute the cohomology $H^\bullet (\CA_{n,m,\G},*)$ of $\CA_{n,m,\G}.$ Investigate the intersection cohomology of $\CA_{n,m,\G}.$

\vskip 2mm\noindent
{\bf Problem 3.} Generalize the trace formula on the Siegel modular variety obtained by Sophie Morel to the universal abelian variety. For her result
on the trace formula on the Siegel modular variety, we refer to her paper, {\it Cohomologie d'intersection des vari{'e}t{\'e}s modulaires de Siegel, suite.}

\vskip 2mm\noindent
{\bf Problem 4.} Develop the theory of the stability of Jacobi forms using the Siegel-Jacobi operator. The theory of the stability involves in
the theory of unitary representations of the infinite dimensional symplectic group $Sp(\infty,\BR)$ and the infinite dimensional unitary group $U(\infty)$.

\vskip 2mm\noindent
{\bf Problem 5.} Compute the geodesics, the distance between two points and curvatures explicitly in the Siegel-Jacobi space $(\BH_{n,m},ds^2_{n,m;A,B}).$

\vskip 2mm
Siegel proved the following theorem for the Siegel space $(\BH_n, ds^2_{n;1}).$
%\vskip 0.2cm\noindent
\begin{theorem}\,({\bf Siegel\,\cite{Si1}}).
(1) There exists exactly one geodesic joining two arbitrary points
$\Om_0,\,\Om_1$ in $\BH_n$. Let $R(\Om_0,\Om_1)$ be the
cross-ratio defined by
\begin{equation*}
R(\Om_0,\Om_1)=(\Om_0-\Om_1)(\Om_0-{\overline
\Om}_1)^{-1}(\overline{\Om}_0-\overline{\Om}_1)(\overline{\Om}_0-\Om_1)^{-1}.
\end{equation*}
For brevity, we put $R_*=R(\Om_0,\Om_1).$ Then the symplectic
length $\rho(\Om_0,\Om_1)$ of the geodesic joining $\Om_0$ and
$\Om_1$ is given by
\begin{equation*}
\rho(\Om_0,\Om_1)^2=\s \left( \left( \log { {1+R_*^{\frac 12}
}\over {1-R_*^{\frac 12} } }\right)^2\right),
\end{equation*} where
\begin{equation*}
\left( \log { {1+R_*^{\frac 12} }\over {1-R_*^{\frac 12} }
}\right)^2=\,4\,R_* \left( \sum_{k=0}^{\infty} { {R_*^k}\over
{2k+1}}\right)^2.
\end{equation*}

\noindent (2) For $M\in Sp(n,\BR)$, we set
$${\tilde \Om}_0=M\cdot \Om_0\quad \textrm{and}\quad {\tilde \Om}_1=M\cdot
\Om_1.$$ Then $R(\Om_1,\Om_0)$ and
$R({\tilde\Om}_1,{\tilde\Om}_0)$ have the same eigenvalues.

\vskip 2mm\noindent
\noindent (3) All geodesics are symplectic images of the special
geodesics
\begin{equation*}
\alpha(t)=i\,\textrm{diag}(a_1^t,a_2^t,\cdots,a_n^t),
\end{equation*}
where $a_1,a_2,\cdots,a_n$ are arbitrary positive real numbers
satisfying the condition
$$\sum_{k=1}^n \left( \log a_k\right)^2=1.$$
\end{theorem}
\noindent The proof of the above theorem can be found in
\cite{Si1}, pp.\,289-293.

\vskip 2mm\noindent
{\bf Problem 6.} Solve Problem 4 and Problem 5 in Section 3.
Express the center of the algebra $\BD (\BH_{n,m})$ of all $G^J$-invariant differential operators on $\BH_{n,m}$ explicitly.
Describe the center of the universal enveloping algebra of the Lie algebra of the Jacobi group $G^J$ explicitly.

\vskip 3mm\noindent
{\bf Problem 7.} Develop the spectral theory of the Laplacian $\Delta_{n,m;A,B}$ on $\G_*\ba \BH_{n,m}$ for an arithmetic subgroup of $\G_{n,m}.$
Balslev \cite{B} developed the spectral theory of the Laplacian $\Delta_{1,1;1,1}$ on $\G_*\ba \BH_{1,1}$ for certain arithmetic subgroup of $\G_{1,1}.$

\vskip 3mm\noindent
{\bf Problem 8.} Develop the theory of harmonic analysis on the Siegel-Jacobi disk $\BD_{n,m}.$

\vskip 3mm\noindent
{\bf Problem 9.} Study unitary representations of the Jacobi group $G^J$. Develop the theory of the orbit method for the Jacobi group $G^J.$

\vskip 3mm\noindent
{\bf Problem 10.} Attach Galois representations to cuspidal Jacobi forms.

\vskip 3mm\noindent
{\bf Problem 11.} Develop the theory of automorphic $L$-function for the Jacobi group $G^J(\mathbb A)$.

\vskip 3mm\noindent
{\bf Problem 12.} Find the trace formula for the Jacobi group $G^J(\mathbb A)$.

\vskip 3mm\noindent
{\bf Problem 13.} Decompose the Hilbert space $L^2\big( G^J(\BQ)\ba G^J(\mathbb A) \big)$ into irreducibles explicitly.

%\vskip 3mm \noindent
%{\bf Problem 14.} Establish the analogue of Langlands program for the Jacobi group $G^J(\mathbb A)$.

\vskip 3mm\noindent
{\bf Problem 14.} Construct Maass-Jacobi forms. Express the Fourier expansion of a Maass-Jacobi form explicitly.

\vskip 3mm\noindent
{\bf Problem 15.} Investigate the relations among Jacobi forms, hyperbolic Kac-Moody algebras, infinite products, the monster group and the Moonshine (cf.\ \cite{YJH5}).

\vskip 3mm\noindent
{\bf Problem 16.} Provide applications to physics (quantum mechanics, quantum optics, coherent states,$\cdots$), the theory of elliptic genera, singularity theory of K. Saito
etc.

\end{section}

\vskip 10mm
\centerline{\bf Acknowledgements}
\vskip 0.53cm
I would like to give my hearty thanks to Eberhard Freitag and Don Zagier for their advice and their interest in this subject. In particular, it is a pleasure to thank E. Freitag
for letting me know the paper \cite{M1} of Hans Maass.


\vspace{10mm}


\begin{thebibliography}{99}

\bibitem{A1} A. N. Andrianov, {\em Modular descent and the Saito-Kurokawa lift}, Invent. Math. {\bf 289},
Springer-Verlag (1987).

\bibitem{B} E. Balslev, Spectral theory of the Laplacian on the modular Jacobi group manifold,
preprint, Aarhus University (2012).

\bibitem{Be1} R. Berndt, {\em Zur Arithmetik der elliptischen
Funktionenk{\"o}rper h{\"o}herer Stufe}, J. reine angew. Math.,
{\bf 326}(1981), 79-94.

\bibitem{Be2} R. Berndt, {\em Meromorphic Funktionen auf Mumfords Kompaktifizierung der
universellen elliptischen Kurve $N$-ter Stufe}, J. reine angew.
Math., {\bf 326}(1981), 95-103.

\bibitem{Be3} R. Berndt, {\em Shimuras Reziprozit{\"a}tsgesetz f{\"u}r den K{\"o}rper der arithmetischen elliptischen Funktionen
beliebiger Stufe }, J. reine angew. Math., {\bf 343}(1983), 123-145.

\bibitem{Be4} R. Berndt, {\em Die Jacobigruppe und die W{\"a}rmeleitungsgleichung }, Math. Z., {\bf 191}(1986), 351-361.

\bibitem{Be5} R. Berndt, {\em The Continuous Part of $L^2(\G^J\ba G^J)$ for the Jacobi Group}, Abh. Math. Sem. Univ. Hamburg., {\bf 60}(1990), 225-248.


\bibitem{Be6} R. Berndt and S. B{\"o}cherer, {\em Jacobi Forms and Discrete Series Representations of the Jacobi Group}, Math. Z., {\bf 204}(1990), 13-44.

\bibitem{Be7} R. Berndt, {\em On Automorphic Forms for the Jacobi Group}, Jb. d. Dt. Math.-Verein., {\bf 97}(1995), 1-18.


\bibitem{BS} R. Berndt and R. Schmidt, {\em Elements of the Representation Theory of the Jacobi Group}, Birkh{\"a}user, 1998.


\bibitem{BC} D. Bump and Y. J. Choie, {\em Derivatives of modular forms of negative weight}, Pure Appl. Math. Q. {\bf 2} (2006), no. 1, 111-133.

\bibitem{EZ} M. Eichler and D. Zagier, {\em The Theory of Jacobi Forms}, Progress in Mathematics {\bf 55}, Birkh{\"a}user, Boston, Basel and Stuttgart, 1985.

%\bibitem{FC} G. Faltings and C.-L. Chai, {\em Degeneration of abelian varieties}, Ergebnisse der Math. {\bf 22}, Springer-Verlag, Berlin-Heidelberg-New York (1990).

\bibitem{FF} A. J. Feingold and I. B. Frenkel, {\em A Hyperbolic Kac-Moody Algebra and the Theory of Siegel Modular Forms of genus $2$}, Math. Ann., {\bf 263}(1983), 87-144.

\bibitem{Fr1} E. Freitag, {\em Stabile Modulformen}, Math. Ann. {\bf 230} (1977), 162--170.

\bibitem{Fr2} E. Freitag, {\em Siegelsche Modulfunktionen}, Grundlehren de mathematischen Wissenschaften {\bf 55}, Springer-Verlag, Berlin-Heidelberg-New York (1983).


\bibitem{G} S. Gelbart, {\em Weil's Representation and the Spectrum of the Metaplectic Group}, Lecture Notes in Math. $ \textbf{530}$, Springer-Verlag,
Berlin and New York, 1976.

\bibitem{Gri} V. A. Gritsenko, {\em The action of modular operators on the Fourier-Jacobi coefficients of modular forms}, Math. USSR Sbornik, {\bf 74}(1984), 237-268.

\bibitem{HC1} Harish-Chandra, \textit{Representations of a semisimple Lie group on a Banach space. I.,} Trans. Amer. Math. Soc. {\bf 75} (1953), 185-243.

\bibitem{HC2} Harish-Chandra, \textit{The characters of semisimple Lie groups,} Trans. Amer. Math. Soc. {\bf 83} (1956), 98-163.

\bibitem{He1} S. Helgason, {\em Differential operators
on homogeneous spaces,} Acta Math. {\bf 102} (1959), 239-299.

\bibitem{He2} S. Helgason, {\em Groups and geometric analysis,} Academic Press, New York (1984).

%\bibitem{Hi} F. Hirzebruch, {\em Automorphe Formen und der Satz von Riemann-Roch}, Symposium Internacional de Topologia, Unesco (1958).

\bibitem{Ho} R. Howe, \textit{Perspectives on invariant theory: Schur duality, multiplicity-free actions and beyond,} The Schur lectures
(1992)\,(Tel Aviv), Israel Math. Conf. Proceedings, {\bf vol.\ 8} (1995), 1--182.

\bibitem{Ig} J. Igusa, {\em Theta Functions,} Springer-Verlag, Berlin-Heidelberg-New York (1971).


\bibitem{Ik} T. Ikeda, {\em On the lifting of elliptic cusp forms to Siegel cusp forms of degree $2n$}, Ann. Math. {\bf 154} (2001), 641--681.

\bibitem{Ik1} T. Ikeda, {\em Pullback of the lifting of elliptic cusp forms and Miyawaki's conjecture}, Duke Math. J. {\bf 131} (2006), no. 3, 469--497.

\bibitem{IOY} M. Itoh, H. Ochiai and J.-H. Yang, {\em Invariant differential operators on Siegel-Jacobi space,} preprint (2013).

\bibitem{J} C. G. J. Jacobi, {\em Fundamenta nova theoriae functionum ellipticum}, K{\"o}nigsberg, (1829).


\bibitem{KV}  M. Kashiwara and M. Vergne, {\em On the Segal-Shale-Weil Representations and Harmonic Polynomials}, Invent. Math. $ \textbf{44}$ (1978), 1--47.

\bibitem{Ko1} W. Kohnen, {\em Modular forms of half-integral weight on $\G_0(4)$}, Math. Ann. {\bf 248} (1980), 249--266.

\bibitem{Ko2} W. Kohnen, {\em Lifting modular forms of half-integral weight to Siegel modular forms of even degree}, Math. Ann. {\bf 322} (2003), 787--809.

\bibitem{KW} A. Kor{\'a}nyi and J. Wolf, \emph{Generalized Cayley transformations of bounded symmetric domains,}
Amer. J. Math. {\bf 87} (1965), 899-939.

\bibitem{Kr1} J. Kramer, {\em A geometrical approach to the theory of Jacobi forms}, Compositio Math., {\bf 79}(1991), 1-19.

\bibitem{Kr2} J. Kramer, {\em An arithmetic theory of Jacobi forms in higher dimensions }, J. reine angew. Math., {\bf 458}(1995), 157-182.

\bibitem{Kuz} N. V. Kuznetsov, {\em A new class of identities for the Fourier coefficients of modular forms}, Acta Arith. (1975), 505--519.

\bibitem{LV} G. Lion and M. Vergne, {\em The Weil representation, Maslov index and Theta series}, Progress in Mathematics, {\bf 6}, Birkh{\"a}user, Boston, Basel
and Stuttgart, 1980.

\bibitem{Ma} J. Marklof, {\em Pair correlation densities of inhomogeneous quadratic forms}, Ann. of Math., {\bf 158} (2003), 419-471.

\bibitem{M1} H. Maass, {\em Die Differentialgleichungen in der Theorie der Siegelschen Modulfunktionen}, Math. Ann. {\bf 126} (1953), 44--68.

\bibitem{M2} H. Maass, {\em Siegel modular forms and Dirichlet series,} Lecture Notes in Math. {\bf 216}, Springer-Verlag, Berlin-Heidelberg-New York (1971).

\bibitem{M3} H. Maass, {\em {\" U}ber eine Spezialschar von Modulformen zweiten Grades I}, Invent. Math. {\bf 52} (1979), 95-104.

\bibitem{M4} H. Maass, {\em {\" U}ber eine Spezialschar von Modulformen zweiten Grades II}, Invent. Math. {\bf 53} (1979), 249--253.

\bibitem{M5} H. Maass, {\em {\" U}ber eine Spezialschar von Modulformen zweiten Grades III}, Invent. Math. {\bf 53} (1979), 255--265.

\bibitem{Min} H. Minkowski, {\em Gesammelte Abhandlungen:} Chelsea, New York (1967).

\bibitem{Mf1} D. Mumford, {\em Hirzebruch's Proportionality Theorem in the Non-Compact Case}, Invent. Math. {\bf 42} (1977), 239--272.

\bibitem{Mf2} D. Mumford, {\em Tata Lectures on Theta I}, Progress in Mathematics, {\bf 28}, Birkh{\"a}user, Boston, Basel and Stuttgart, 1983.

\bibitem{Mf3} D. Mumford, M. Nori and P. Norman, {\em Tata Lectures on Theta III}, {\bf 97}, Birkh{\"a}user, Boston, Basel and Stuttgart, 1991.

\bibitem{Mu1} A. Murase, {\em $L$-functions attached to Jacobi forms of degree $n$. Part I : The Basic Identity}, J. reine angew. Math., {\bf 401} (1989), 122-156.

\bibitem{Mu2} A. Murase, {\em $L$-functions attached to Jacobi forms of degree $n$. Part II : Functional Equation}, Math. Ann., {\bf 290 } (1991), 247-276.

\bibitem{MS} A. Murase and T. Sugano, {\em Whittaker-Shintani Functions on the Symplectic Group of Fourier-Jacobi Type}, Compositio Math., {\bf 79} (1991), 321-349.

\bibitem{PS} I. Piateski-Sharpiro, \emph{Automorphic Functions and the Geometry of Classical Domains,} Gordan-Breach, New York (1966).

\bibitem{Pit} A. Pitale, {\em Jacobi Maass forms}, Abh. Math. Sem. Univ. Hamburg {\bf 79} (2009), 87-111.

\bibitem{Ru} B. Runge, {\em Theta functions and Siegel-Jacobi functions}, Acta Math., {\bf 175} (1995), 165-196.

\bibitem{Sa1} I. Satake, {\em Fock Representations and Theta Functions}, Ann. Math. Studies, {\bf 66} (1971), 393-405.

\bibitem{Sa2} I. Satake, \emph{Algebraic Structures of Symmetric Domains,} Kano Memorial Lectures 4, Iwanami Shoton, Publishers and Princeton University Press (1980).

\bibitem{Sh1} G. Shimura, {\em On modular forms of half integral weight }, Ann. of Math. {\bf 97} (1973), 440-481.

\bibitem{Sh2} G. Shimura, {\em On certain reciprocity laws for theta functions and modular forms}, Acta Math. {\bf 141} (1979), 35-71.

\bibitem{Sh3} G. Shimura, \emph{Invariant differential operators on hermitian symmetric spaces,} Ann. Math. {\bf 132}\,(1990), 237-272.



\bibitem{Si1} C.~L.~Siegel, \emph{Symplectic Geometry,} Amer. J. Math. {\bf 65} (1943), 1-86; Academic Press, New York and London (1964);
Gesammelte Abhandlungen, no.\,\,41, vol. II, Springer-Verlag (1966), 274-359.

\bibitem{Si2} C.~L.~Siegel, \emph{Gesammelte Abhandlungen I-IV}, Springer-Verlag(I-III: 1966; IV: 1979).

\bibitem{Si3}  C. L. Siegel, {\em Topics in Complex Function Theory\,: Abelian Functions and Modular Functions of Several Variables}, vol. III, Wiley-Interscience, 1973.

\bibitem{Ta} Y.-S. Tai, {\em On the Kodaira Dimension of the Moduli Space of Abelian Varieties}, Invent. Math. {\bf 68} (1982), 425--439.

\bibitem{W} H. Weyl, \textit{The classical groups: Their invariants and representations,} Princeton Univ. Press, Princeton, New Jersey, second edition (1946).


\bibitem{YJH0}  J.-H. Yang, {\em The Siegel-Jacobi Operator}, Abh. Math. Sem. Univ. Hamburg {\bf 63} (1993), 135--146.

\bibitem{YJH1} J.-H. Yang, {\em Vanishing theorems on Jacobi forms of higher degree}, J. Korean Math. Soc., {\bf 30}(1)(1993), 185-198.

\bibitem{YJH2} J.-H. Yang, {\em Remarks on Jacobi forms of higher degree}, Proc. of the 1993 Workshop on Automorphic Forms and
Related Topics, edited by Jin-Woo Son and Jae-Hyun Yang, the Pyungsan Institute for Mathematical Sciences, (1993), 33-58.

\bibitem{YJH3} J.-H.~Yang, {\em Singular Jacobi Forms,} Trans. Amer. Math. Soc. {\bf 347\,(6)} (1995), 2041-2049.

\bibitem{YJH4} J.-H.~Yang, {\em Construction of vector valued modular forms from Jacobi forms,} Canadian J. of Math. {\bf 47\,(6)} (1995), 1329-1339.

\bibitem{YJH5} J.-H.~Yang, {\em Kac-Moody algebras, the monstrous moonshine, Jacobi forms and infinite products,} Proceedings of the 1995 Symposium on Number theory, geometry and
related topics, the Pyungsan Institute for Mathematical Sciences (1996), 13--82 or arXiv:math.NT/0612474.

\bibitem{YJH6} J.-H.~Yang, {\em A geometrical theory of Jacobi forms of higher degree,} Proceedings of Symposium on Hodge Theory
and Algebraic Geometry\,(\,edited by Tadao Oda\,), Sendai, Japan (1996), 125-147 {\it or} Kyungpook Math. J. {\bf 40\,(2)} (2000), 209-237 {\it or} arXiv:math.NT/0602267.

\bibitem{YJH7} J.-H.~Yang, {\em The Method of Orbits for Real Lie Groups,} Kyungpook Math. J. {\bf 42\,(2)} (2002), 199-272 {\it or} arXiv:math.RT/0602056.

\bibitem{YJH8} J.-H. Yang, {\em A note on a fundamental domain for Siegel-Jacobi space,} Houston Journal of Mathematics {\bf 32\,(3)} (2006), 701--712.

\bibitem{YJH9} J.-H. Yang, {\em Invariant metrics and Laplacians on Siegel-Jacobi space,} Journal of Number Theory {\bf 127} (2007), 83--102.

\bibitem{YJH10} J.-H. Yang, {\em A partial Cayley transform for Siegel-Jacobi disk,} J. Korean Math. Soc. {\bf 45}, No. 3 (2008), 781-794.

\bibitem{YJH11} J.-H. Yang, \emph{Invariant metrics and Laplacians on Siegel-Jacobi disk,} Chinese Annals of Mathematics, Vol. 31B(1), 2010, 85-100.

\bibitem{YJH12} J.-H. Yang, {\em A Note on Maass-Jacobi Forms,} Kyungpook Math. J. {\bf 43}, no. 4 (2003), 547-566.

\bibitem{YJH13} J.-H. Yang, {\em A Note on Maass-Jacobi Forms II,} Kyungpook Math. J. {\bf 53}, no. 1 (2013), 49-86.

\bibitem{YJH14} J.-H. Yang, Y.-H. Yong, S.-N. Huh, J.-H. Shin and G.-H. Min,
{\em Sectional Curvatures of the Siegel-Jacobi Space,} Bull. Korean Math. Soc. {\bf 50} (2013), No. 3, pp. 787-799.

\bibitem{YJH15} J.-H. Yang, {\em Invariant differential operators on the Minkowski-Euclid space,} J. Korean Math. Soc. {\bf 50}, No. 2 (2013), 275-306.

\bibitem{YJH16} J.-H. Yang, {\em The Schr{\"o}dinger-Weil representation and theta sums,} preprint (2013).

\bibitem{YY} J. Yang and L. Yin, Derivation of Jacobi forms from connections, arXiv:1301.1156v1 [math.NT] 7 Jan 2013.

\bibitem{Z} D. Zagier, {\em Sur la conjecture de Saito-Kurokawa (d'apr${\grave e}$s H. Maass)}: Seminaire Delange-Pisot-Poitou, Paris, 1979-80, Progress in
 Mathematics {\bf 12}, Birkh{\"a}user, Boston, Basel and Stuttgart (1981), 371--394.


\bibitem{Zh} N. A. Zharkovskaya, {\em The Siegel operator and Hecke operators,} Functional Anal. Appl. {\bf 8} (1974), 113-120.

\bibitem{Zi}  C. Ziegler, {\em Jacobi Forms of Higher Degree}, Abh. Math. Sem. Hamburg {\bf 59} (1989), 191--224.


\end{thebibliography}
\end{document}